\documentclass{article}
\usepackage[english]{babel} % English formatting
\usepackage[utf8]{inputenc} % Standard encoding
\usepackage[a4paper,left=3cm,bottom=3cm]{geometry} % Page formatting
\usepackage{indentfirst} % Indents the first paragraph
\usepackage{pdflscape} % Make pages landscape
\usepackage{graphicx} % Advanced graphics package
\usepackage[export]{adjustbox} 
\usepackage[colorlinks=true,citecolor=blue,urlcolor=blue,linkcolor=black]{hyperref} % Link colors
\usepackage{flafter} % Reference any 'float'
\usepackage[framemethod=tikz]{mdframed} % Box off stuff
\usepackage{color} % Color support
\usepackage{wrapfig} % Text flowing around figures
\usepackage{enumerate} % numbered lists 
\usepackage{enumitem} % control over the layout of list envns, supersedes enumerate 
\usepackage{sectsty} % controls section headers
\usepackage{caption, subcaption} % customize captions in floating envs
\usepackage{lastpage} % for running footer page count, e.g., '3 of 5', if desired. 
\usepackage{textcomp} % supports text companion fonts 
\usepackage{enumerate} % numbered lists 
\usepackage{enumitem} % control over the layout of list envns, supersedes enumerate 
\usepackage{sectsty} % controls section headers
\usepackage[export]{adjustbox}

\usepackage[
maxnames=99,
giveninits=true,
doi=false,
isbn=false]{biblatex} % Imports biblatex package
\renewbibmacro{in:}{}
\DefineBibliographyExtras{english}{%
	\uspunctuation%
}
\DeclareFieldFormat[misc]{title}{\mkbibquote{#1}}
\usepackage{amsmath} % Maths type package
\usepackage{bm} % Bold font maths
\usepackage{amsthm} % Theorem environment 
\usepackage{arydshln} % Draw dashed lines in array/tabular env: \hdashline and \cdashline, etc
\usepackage{bigints} % Draw big integral sign when needed: \bigint, \bigints(ss), \bigoint, etc
\usepackage{amsfonts} % Additional fonts etc
\usepackage{amssymb} % Additional symbols; arrows, \circdast, etc
\usepackage{mathtools} % Really just for \coloneqq 
\usepackage{listings} % Auto latex formatting for code; requires sty package 
\usepackage{tikz} % As a last resort 
\usepackage{mathrsfs} % script font 
\usepackage{faktor} % faktor font 
\usepackage{musicography} % for flat and sharp operators 
\usepackage{breqn} %auto line breaking for displayed equations
\usepackage{url} % for citing websites, etc 
\usepackage{array,multirow,booktabs} % for fancier tables
\usepackage{cleveref}

\usepackage{multicol} % Multi Column Environment
\usepackage{float} % Lets you add images to multi-column environment

\usepackage{fancyhdr} % Fancy headers
\renewcommand{\headrulewidth}{0.4pt}

\numberwithin{equation}{section}

\theoremstyle{plain}
\newtheorem{thm}{Theorem}[section]
\newtheorem{cor}{Corollary}[section]
\newtheorem{lem}{Lemma}[section]

\theoremstyle{definition}
\newtheorem{defn}{Definition}[section]

\theoremstyle{remark}
\newtheorem{rem}{Remark}[section]

\newcommand{\norm}[1]{\left\Vert #1 \right\Vert}
\newcommand{\ip}[1]{\left\langle #1 \right\rangle}
\newcommand{\abs}[1]{\left\vert #1 \right\vert}
\newcommand{\C}{\mathbb{C}}
\newcommand{\R}{\mathbb{R}}
\newcommand{\N}{\mathbb{N}}
\newcommand{\E}{\mathbb{E}}

\DeclareMathOperator{\tr}{tr}
\newcommand{\bmat}[1]{\begin{bmatrix} #1 \end{bmatrix}}

\newcommand{\scr}[1]{\mathscr{#1}}
\newcommand{\dif}{\mathop{}\!\mathrm{d}}

\DeclareMathOperator*{\argmin}{arg\,min}
\DeclareMathOperator{\Ima}{Im}

\DeclareMathOperator{\Cond}{Cond}

\DeclareMathOperator{\Var}{Var}
\DeclareMathOperator{\Cov}{Cov}

\DeclareMathOperator*{\esssup}{ess\,sup}
\newcommand{\bs}[1]{\boldsymbol{#1}}

\title{An Optimal Weighted Least-Squares Method for Operator Learning}
\usepackage{authblk}
\author[1,2]{John Turnage}
\author[3]{Matthew Lowery}
\author[5]{John Jakeman}
\author[4]{Zachary Morrow}
\author[1,2]{\\Akil Narayan}
\author[3]{Varun Shankar}

\affil[1]{Department of Mathematics, University of Utah}
\affil[2]{Scientific Computing and Imaging Institute, University of Utah }
\affil[3]{Kahlert School of Computing, University of Utah}
\affil[4]{Scientific Machine Learning, Sandia National Laboratories}
\affil[5]{Optimization and Uncertainty Quantification, Sandia National Laboratories}
\date{\today}
\begin{document}
\maketitle
\begin{abstract}
We consider the problem of learning an unknown, possibly nonlinear operator between separable Hilbert spaces from supervised data. Inputs are drawn from a prescribed probability measure on the input space, and outputs are (possibly noisy) evaluations of the target operator. We regard admissible operators as square-integrable maps with respect to a fixed approximation measure, and we measure reconstruction error in the corresponding Bochner norm. For a finite-dimensional approximation space $V$ of dimension $N$, we study weighted least squares estimators in $V$ and establish probabilistic stability and accuracy bounds in the Bochner norm. We show that there exist sampling measures and weights---defined via an operator-level Christoffel function---that yield uniformly well-conditioned Gram matrices and near-optimal sample complexity, with a number of training samples $M$ on the order of $N \log N$. We complement the analysis by constructing explicit operator approximation spaces in cases of interest: rank-one linear operators that are dense in the class of bounded linear operators, and rank-one polynomial operators that are dense in the Bochner space under mild assumptions on the approximation measure. For both families we describe implementable procedures for sampling from the associated optimal measures. Finally, we demonstrate the effectiveness of this framework on several benchmark problems, including learning solution operators for the Poisson equation, viscous Burgers' equation, and the incompressible Navier-Stokes equations.

\end{abstract}

%\tableofcontents

%~~~~~~~~~ Begin Sections

\pagenumbering{arabic} 
\setcounter{page}{1}

\section{Introduction}

Operator learning seeks to approximate mappings between function spaces—such as PDE solution operators—directly in the continuum, rather than via fixed discretizations. This perspective has gained prominence with the rise of neural operator architectures. Among these, DeepONets approximate operators using branch–trunk sensor networks \cite{lu2021learning}; Fourier Neural Operators (FNOs) learn spectral-domain convolutions that generalize across grids \cite{li2020fourier}; and many related families—including encoder–decoder, kernel-based, random-feature, and graph-based models—approximate operators through latent finite-dimensional surrogates \cite{bhattacharya2021model, lowery2024kernel, nelsen2021random, li2020neural}.

Early theoretical work established universal approximation guarantees for such models \cite{lanthaler2022error, kovachki2021universal, kovachki2023neural}, while more recent efforts have developed quantitative error bounds under Lipschitz \cite{bhattacharya2021model, chen2023deep, schwab2023deep, liu2024deep} or holomorphic \cite{herrmann2024neural, lanthaler2022error} regularity assumptions. However, these advances primarily concern representational capacity. In practical settings, where training data is scarce and expensive—often involving costly numerical solutions of PDEs—the dominant source of error arises from generalization rather than approximation. Recent work has begun to emphasize this finite-data regime \cite{adcock2024optimal, chen2023deep}.

Most theoretical treatments frame operator learning in terms of parametric latent spaces, where encoder–decoder architectures serve as the approximation class, and error is measured in terms of reconstruction fidelity. 
In contrast, we consider approximation directly in a Bochner space of operators, $L^2_\rho(\mathcal X; \mathcal Y)$, where $\mathcal X$ and $\mathcal Y$ are separable Hilbert spaces of input and output functions and $\rho$ is a prescribed probability measure on $\mathcal X$.
Approximants are built in a fixed, finite-dimensional subspace $V\subset L^2_\rho(\mathcal X; \mathcal Y)$, and we measure error in the corresponding Bochner norm. This removes any assumptions about encoder optimality and places operator learning squarely in an approximation-theoretic setting. 

To address the finite-data regime, we adopt and extend recent developments in weighted least squares (WLS) approximation. When input data is sampled proportionally to the Christoffel function (or its weighted or induced variants), the empirical $L^2$ problem exhibits high-probability stability and optimal approximation rates with near-minimal sample complexity. These results, originally developed for scalar-valued approximations \cite{stabilityAccuracyLS, optimalWeightedLS} and recently generalized to Hilbert-valued settings \cite{adcock2022sparse, adcock2022towards, adcock2024optimalsamp}, provide a principled foundation for operator learning in Bochner $L^2$ spaces.
 
Within this framework, we make the following contributions:
\begin{itemize}
\item We construct orthonormal operator families---both linear and nonlinear---that generate approximation subspaces $V\subset L^2_\rho(\mathcal X; \mathcal Y)$ and establish density results under mild assumptions. These density results provide universal approximation statements for our approach.
\item We adapt WLS theory to the operator learning setting, introducing operator-valued Christoffel densities and sampling routines tied to orthonormal operator bases.
\item We prove high-probability stability and accuracy bounds for the WLS estimator, separating projection error, solver noise, and output discretization effects.
\item We demonstrate the framework on Poisson, viscous Burgers', and Navier–Stokes PDEs, and study how sampling, approximation spaces, and evaluation norms influence performance.
\end{itemize}

The theoretical foundations of weighted least squares in Hilbert-valued function spaces are now well established. Our contribution lies in adapting this framework to the operator learning setting, where approximation occurs directly in a Bochner space of operators. In this context, we construct explicit orthonormal operator bases and derive corresponding operator-level Christoffel functions that define optimal sampling measures. These constructions are not only novel but also practically necessary: they enable implementable operator approximation without reliance on encoder–decoder architectures or latent representations, and provide theoretical underpinnings for recently proposed approaches along these lines \cite{sharma_polynomial_2025}. 

Section \ref{sec:Op_learn_problem} formalizes the operator learning problem, specifying the admissible operator class, approximation measure, and noise model. Section \ref{sec: Main Results: LS-OP} then introduces the discrete weighted least-squares estimator, the associated optimal sampling measures, and the main stability and accuracy results. Concrete operator approximation spaces and sampling procedures are constructed in Section \ref{sec: Construct Approx Space}, and numerical results are presented in Section \ref{sec: Numerical Results}, where we apply the framework to Poisson, viscous Burgers', and Navier–Stokes solution operators.

\section{The Operator Learning Problem}\label{sec:Op_learn_problem}

We now formalize the learning problem. Let $\mathcal X$ and $\mathcal{Y}$
be separable Hilbert spaces, and suppose we aim to recover an unknown operator $K: \mathcal{X} \to \mathcal{Y}$ from a finite number of noisy measurements. To quantify error, we fix a probability measure  $\rho$ on $\mathcal{X}$ and define the admissible class of operators to be the Bochner space $L^2_\rho(\mathcal X; \mathcal Y)$ equipped with the induced norm. 

Given a probability measure $\rho$ on $\mathcal{X}$, the Bochner space $L^2_\rho(\mathcal X; \mathcal Y)$ is the Hilbert space of Borel-measurable operators
\begin{align}\label{eq:XY-def}
  L^2_\rho(\mathcal X; \mathcal Y)\ni A: (\mathcal X, \scr F, \rho) \to (\mathcal Y, \mathscr{B}(\mathcal Y, \norm{\cdot}_{\mathcal Y}))
\end{align}
 endowed with the norm and inner product
\begin{subequations}
\begin{align}\label{eq:L2-norm}
  \norm{A}^2_{L^2_\rho(\mathcal X; \mathcal Y)} \coloneqq \int_{\mathcal X}\norm{A(f)}_{\mathcal Y}^2 \dif \rho(f) < \infty, \\
\label{eq:L2-ip}
  \ip{A,B}_{L^2_\rho(\mathcal X; \mathcal Y)} \coloneqq \int_{\mathcal X}\ip{A(f), B(f)}_{\mathcal Y}\dif\rho(f),
\end{align}
  where  $\mathscr{F}$ is a $\sigma$-algebra on $\mathcal{X}$, $\mathscr{B}(\mathcal{Y}, \|\cdot\|_{\mathcal{Y}})$ is the Borel $\sigma$-algebra on $\mathcal{Y}$ with respect to the norm $\|\cdot\|_{\mathcal{Y}}$, and $\left\langle \cdot, \cdot \right\rangle_{\mathcal{Y}}$ is the inner product on $\mathcal{Y}$. 
\end{subequations}
 When there is no ambiguity, we write $L^2_\rho = L^2_\rho(\mathcal{X}; \mathcal{Y})$. 
 
In many applications, $\rho$ is specified indirectly as the law of an $\mathcal X$-valued random variable 
\begin{align}\label{eq:X-RV-def}
  X:(\Omega, \scr G, \mathbb P) \to (\mathcal X, \scr B(\mathcal X, \norm{\cdot}_{\mathcal X})),
\end{align}
so that $\rho = X_{\#}\mathbb{P}$ is the push-forward of $\mathbb P$ under $X$.\footnote{We provide a brief introduction to Hilbert-space-valued random variables and the Bochner integral in \cref{subsec: Prob Measures on Sep H Space}.} We refer to $\rho$ as the \textit{approximation measure} on $\mathcal{X}$ because our primary goal is to construct approximations to $K$ that are accurate with respect to the Bochner norm induced by $\rho$.

We are given access to noisy evaluations of $K$ at a finite collection of inputs. That is,
we observe data of the form
\begin{equation}\label{eq:data_model}
g^i = K(f^i) + \eta^i, \hspace{0.5cm} i = 1,\cdots,M,
\end{equation}
where the inputs $f^i\in \mathcal X$ are drawn from a \emph{sampling measure} $\mu$ on \(X\) (specified
in Section \ref{sec: Main Results: LS-OP}), and $\eta^i$ models observation noise.

We adopt a deterministic noise model $\eta^i = \eta(f^i),$ with 
\begin{align}\label{eq: noise_defn}
  \eta \in \widetilde{L}^\infty_\rho \coloneqq L^\infty_{\rho}(\mathcal X;\mathcal Y) \cap L^2_\rho(\mathcal X; \mathcal Y)  = \{A\in L^2_\rho(\mathcal{X};\mathcal{Y})\text{ : } \esssup_{f\in \textrm{supp } \rho} \norm{A(f)}_{\mathcal Y} < \infty\}.
\end{align}
This setting is natural for our motivating applications, in which $K$ is the solution
operator of a PDE, and the observations $g^i$ are produced by a deterministic numerical
solver. In that case, $\eta$ accounts for discretization error between the true solution
$K(f^i)$ and the solver output, as well as truncation errors arising from representing
$f^i$ and $g^i$ in finite-dimensional spaces. Both contributions are deterministic
functions of $f^i$, and $\norm{\eta}_{L^\infty_\rho(\mathcal X;\mathcal Y)}$ provides a uniform bound on such
discretization errors. 
For similar analyses using a stochastic noise model in the setting of scalar-valued function approximation, see e.g. \cite{stabilityAccuracyLS,optimalWeightedLS}. 
 
We assume an \textit{a priori} prescription of an $N$-dimensional approximation subspace ($N \in \N$) of $L^2_\rho(\mathcal X;\mathcal Y)$ spanned by an orthonormal basis $\Phi_n$:
\begin{align}\label{eq: approx_space}
  V &\coloneqq \mathrm{span}_{n\in [N]}\{\Phi_{n}\}, &
  \left\langle \Phi_n, \Phi_m \right\rangle_{L^2_\rho} &= \delta_{m,n}, &
  m, n &\in [N].
\end{align}
The optimal estimator of $K$ in $V$ is the $L^2_{\rho}$-orthogonal projection
\begin{align}\label{eq:ls}
  K^*_{V} &\coloneqq \Pi_{V}K = \argmin_{A \in V}\norm{K-A}_{L^2_{\rho}}, & \epsilon_{V}(K) \coloneqq \left\| K^\ast_{V} - K \right\|_{L^2_\rho} = \min_{A \in V} \| K - A \|_{L^2_\rho}.
\end{align}
The projection $K_V^*$ and error $\epsilon_V(K)$ are defined entirely at the
continuous level and are not directly computable from finitely many data pairs
$(f^i,g^i)$.

In the next section we derive a discrete weighted least-squares problem, built from
\eqref{eq:data_model} and the basis \eqref{eq: approx_space}, that produces a computable estimator $\tilde K_V^*$ of $K_V^*$ and admits high-probability stability and accuracy guarantees.

\section{Discrete Least-Squares Theory}\label{sec: Main Results: LS-OP}
In this section we introduce the discrete weighted least-squares estimator corresponding to
the continuous projection problem \eqref{eq:ls} and state the main stability and
accuracy results. We work in the Bochner space $L^2_\rho(\mathcal X;\mathcal Y)$ and in the
$N$-dimensional approximation space $V \subset L^2_\rho(\mathcal X;\mathcal Y)$ with orthonormal basis
$\{\Phi_n\}_{n \in [N]}$ defined in Section \ref{sec:Op_learn_problem}.

As in Section \ref{sec:Op_learn_problem}, we assume access to noisy data pairs $(f^i,g^i)$ of the form \eqref{eq:data_model},
where $f^i \sim \mu$ are drawn i.i.d.\ from a sampling measure $\mu$ on $\mathcal X$. A central task in this section is to construct sampling measures $\mu$
and weights $w$ for which the resulting discrete least-squares problem is stable and
near-optimal in sample complexity. We identify a construction $\mu = \mu(V,\rho)$ that is tantamount to an operator-approximation analogue of \textit{leverage score sampling} or \textit{optimal sampling} \cite{optimalWeightedLS,mahoney_randomized_2011,woodruff_sketching_2014}. Our main results echo known results in the finite-dimensional (function approximation) setting: there is an optimal choice of sampling measure $\mu$, a function only of $V$ and $\rho$, for which stability and optimal accuracy can be achieved with high probability using a near-optimal number of $M \gtrsim N \log N$ measurements.

\begin{figure}[h]
  \centering
  \includegraphics[width=0.75\linewidth]{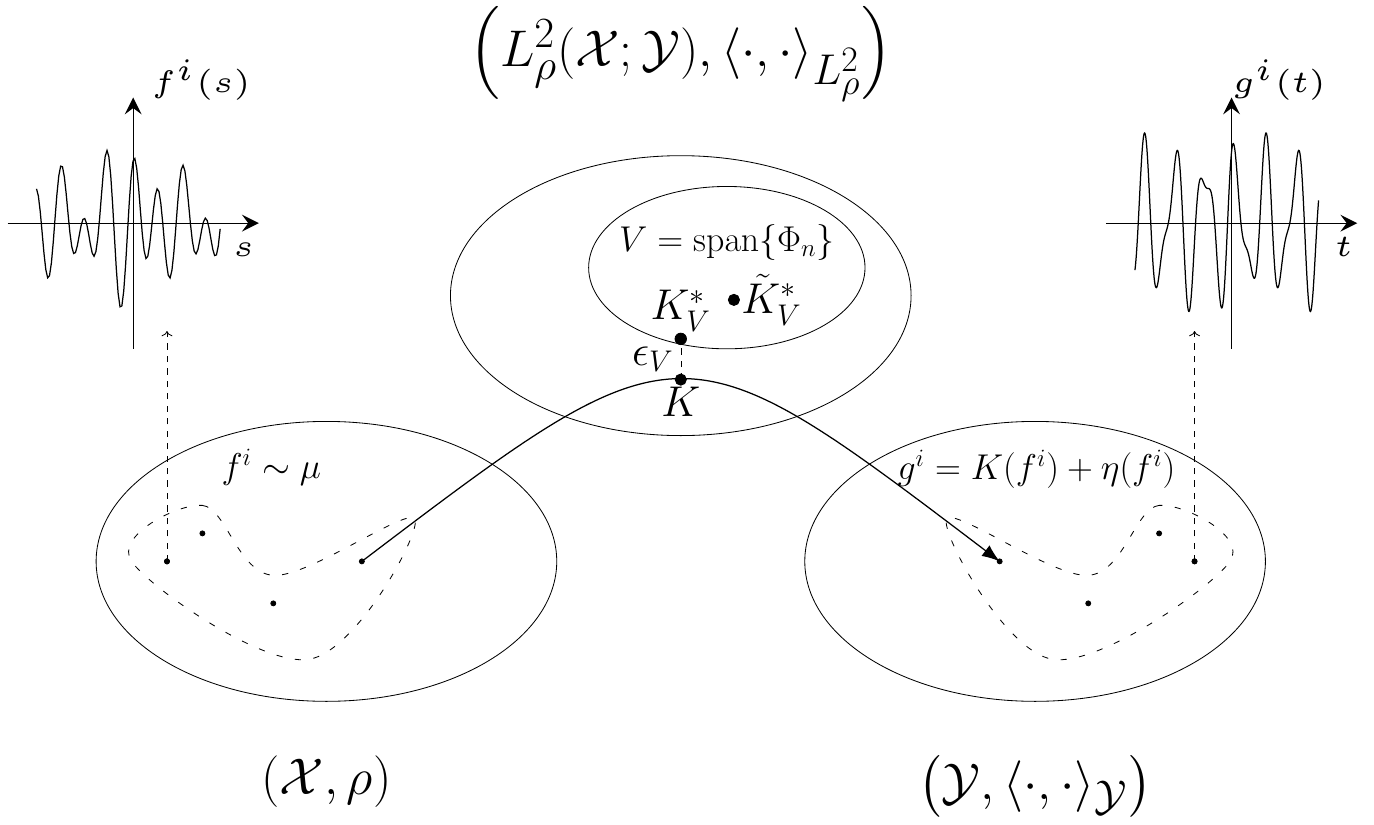}
  \caption{A Schematic of the Least-Squares Problem.}
  \label{fig:LS-Diagram}
\end{figure}

\subsection{Problem Formulation}

With the data $(f^i,g^i)$ and  a positive weight functional $w:\mathcal X\to \R_{+}$ (analogous to quadrature weight function), we formulate the following discrete weighted least-squares problem for the approximation $\tilde{K}_{V}^*$:
\begin{align}\label{eq:discrete-ls}
  \tilde{K}_{V}^* &\in \argmin_{A \in V} \norm{K + \eta - A}_{L^2_{\rho, M}}, & 
  \norm{A}_{L^2_{\rho,M}}^2 &\coloneqq \frac{1}{M} \sum_{i\in [M]} w(f^i)\norm{A(f^i)}^2_{\mathcal{Y}}.
\end{align}
The relationship between the weight functional $w$ and the approximation measure $\rho$ will be made clear shortly. 
The problem above is a linear least squares problem, with an explicitly computable unique solution when the discrete semi-norm $\|\cdot\|_{L^2_{\rho,M}}$ is a proper norm on $V$: namely, 
$$\tilde{K}^*_V = \sum_{n\in [N]}c_n\Phi_n, \quad \text{with} \quad  \bs{c} = \bs{G}^{-1}\bs{g}, \quad (\bs{G})_{j,k} = \ip{\Phi_k, \Phi_j}_{L^2_{\rho,M}},\quad \text{and}\quad (\bs g)_j = \ip{K + \eta, \Phi_j}_{L^2_{\rho,M}}. $$

A schematic of this discrete least-squares problem is given in Figure \ref{fig:LS-Diagram}. When evaluating $K$ is expensive, the principal cost of this procedure lies in generating the data $\{g^i\}_{i\in [M]}$; solving the least-squares systems is typically negligible by comparison. 

If the inputs $f^i$ are selected as i.i.d.\ samples from $\mu$, then in the asymptotically large-$M$ regime, we expect that $\tilde{K}_{V}^\ast$ converges to $K_{V}^\ast + \Pi_{V}\eta$ so long as the weighted discrete least-squares formulation \eqref{eq:discrete-ls} converges to its companion continuum formulation in \eqref{eq:ls}. To ensure this asymptotic consistency, we make the following assumptions:
\begin{align}\label{eq:asymptotic-assumptions}
  f^i &\stackrel{\mathrm{iid}}{\sim} \mu, &  
  \rho &\ll \mu, & 
  w &= \frac{\mathrm{d} \rho}{\mathrm{d} \mu} >0, 
\end{align}
where the absolute continuity requirement $\rho \ll \mu$ means
\begin{align*}
  \mu(E) = 0 \hskip 5pt \Rightarrow \hskip 5pt \rho(E) = 0 \hskip 5pt \forall \hskip 3pt E \in \mathscr F.
\end{align*}
Under \eqref{eq:asymptotic-assumptions}, the discrete semi-norm \eqref{eq:discrete-ls} is an unbiased
Monte Carlo estimator of the Bochner norm \eqref{eq:L2-norm}, and in the limit
$M \to \infty$ the discrete problem \eqref{eq:discrete-ls} recovers the continuous
projection problem \eqref{eq:ls}.

The remainder of this section studies the pre-asymptotic, finite-data regime. We introduce
a Christoffel function and the associated weighted Nikolskii constant that control the
concentration of the Gram matrix $\bs G$ around the identity, identify an optimal sampling
measure $\mu$ for which the least-squares problem is uniformly well-conditioned with
high probability, and derive stability and accuracy bounds for the estimator
$\tilde K_V^*$.  Computational considerations are discussed in \cref{sec: Construct Approx Space}.

A summary of the notation we use throughout this paper is provided in \cref{tab:notation}.

\begin{table}[htb]
  \begin{center}
  \resizebox{\textwidth}{!}{
    \renewcommand{\tabcolsep}{0.4cm}
    \renewcommand{\arraystretch}{1.3}
    {\small
    \begin{tabular}{@{}ccp{0.8\textwidth}@{}}
      \toprule
      $(\mathcal{X}, \mathscr{F}, \rho)$, $\|\cdot\|_{\mathcal{X}}$, $\left\langle \cdot,\cdot\right\rangle_{\mathcal{X}}$ & \eqref{eq:XY-def} & Hilbertian measure space of input functions, norm, and inner product \\
      $X \sim \rho$ & \eqref{eq:X-RV-def} & Random variable realizing functions on $\mathcal{X}$ and its distribution \\
      $\mathcal{Y}$, $\|\cdot\|_{\mathcal{Y}}$, $\left\langle \cdot, \cdot \right\rangle_{\mathcal{Y}}$ & \eqref{eq:XY-def} & Hilbertian space of output functions, norm, and inner product \\
      $L^2_\rho(\mathcal{X} ;  \mathcal{Y}), \norm{\cdot}_{L^2_\rho}$ & \eqref{eq:L2-norm} & Space of Bochner square-integrable mappings from $\mathcal{X}$ to $\mathcal{Y}$ and its norm\\
      $V, \{\Phi_n\}_{n\in [N]}$ & \eqref{eq: approx_space} & $N$-dimensional subspace of $L^2_\rho(\mathcal X; \mathcal{Y})$ and associated orthonormal basis\\
      $M$, $N$ & & Number of samples and dimension of approximation subspace $V\subset L^2_\rho(\mathcal X;\mathcal Y)$\\
      $K^*_{V}, \epsilon_{V}(K)$ & \eqref{eq:ls} & Optimal estimator of $K\in L^2_\rho$ in $V$ and associated projection error\\
      $g^i, \eta$ & \eqref{eq:data_model} & Observation of a sample $f^i\sim \mu$ with noise $\eta(f^i)$ \\
      $\mu, w$ & \eqref{eq:asymptotic-assumptions} & Sampling measure and weight function defining the discrete least squares problem\\
      $\bs{G}$ &\eqref{eq:normal-equations} & The Gram matrix associated to the discrete least squares problem\\
      $\tilde{K}^*_{V}$, $\norm{\cdot}_{L^2_{\rho,M}}$ & \eqref{eq:discrete-ls} & Solution to the discrete weighted least squares problem under the discrete semi-norm\\
      $\kappa_{w}^{\infty}$ & \eqref{eq: Nikolskii} & The weighted Nikolskii constant\\
      $\Pi_W$ & \eqref{eq:discrete-ls} & The orthogonal projection operator onto a subspace $W$\\
      $\{\xi_j\}, \{\psi_j\}$ & & Orthonormal bases for $\mathcal{X}$ and $\mathcal{Y},$ respectively\\
      $\N_F^{\infty}$ & \eqref{eq: inf_multi_idx_finite_supp} & The set of countably infinite multi-indices with finite support \\
    \bottomrule
    \end{tabular}
  }
    \renewcommand{\arraystretch}{1}
    \renewcommand{\tabcolsep}{12pt}
  }
  \end{center}
  \caption{Notation used throughout this paper.} \label{tab:notation}
\end{table}

\subsection{Stability}
We provide results that motivate a proper construction of the sampling measure $\mu$ to ensure stability of the least squares problem \eqref{eq:discrete-ls}. The results we present are operator generalizations of well-known least squares procedures for approximating vectors or functions, cf. \cite{mahoney_randomized_2011,woodruff_sketching_2014,stabilityAccuracyLS,optimalWeightedLS}.

The stability of the least squares problem depends on the Gram matrix \(\bs{G}\) concentrating to the identity matrix \(\bs{I}\). This ensures the problem is well-conditioned and numerically reliable, as formalized in the following lemma, with proof provided in Appendix \ref{Pf: comp_norms}.
\begin{lem}\label{lem:comp_norms}
    For all $\delta \in (0,1)$
    $$\Cond(\bs G) \le \frac{1+\delta}{1-\delta} \iff \norm{\bs{G} - \bs{I}}_2 \le \delta \iff  (1-\delta)\norm{A}^2_{L^2_\rho} \le \norm{A}_{L^2_{\rho,M}}^2 \le (1+\delta)\norm{A}_{L^2_\rho}^2\hspace{0.25cm}(\forall A \in V).$$
\end{lem}
Such pre-asymptotic stability is therefore equivalent to a norm comparison of $\norm{\cdot}_{L^2_\rho}$ and $\norm{\cdot}_{L^2_{\rho,M}}$ and  is controlled in large part by the following constant: 
\begin{align}\label{def: kappa_inf}
  \kappa^{\infty}_{w,V} &\coloneqq \norm{ \kappa_{w,V}}_{L^{\infty}_{\rho}(\mathcal X; \R)}, &
  \kappa_{w,V}(f) &\coloneqq \sum_{n\in [N]}w(f)\norm{\Phi_n(f)}^2_{\mathcal{Y}}
\end{align}
where $\kappa_{w,V}$ is the reciprocal of the weighted Christoffel function of the subspace $V$.\footnote{Hereafter, we will drop the explicit dependence of $\kappa_{w,V}^\infty$ and $\kappa_{w,V}$ on $V$ if the context is clear.} Via the Cauchy-Schwarz inequality, $\kappa_w^{\infty}$ yields a \textit{weighted Nikolskii-type} inequality that relates $L^2$ and $L^\infty$ measurements of operators:
\begin{align}\label{eq: Nikolskii}
  \kappa_w^{\infty} \coloneqq \sup_{A \in V\backslash\{0\}} \frac{\|A\|^2_{L^\infty_\rho(\mathcal{X};\mathcal{Y})}}{\|A\|^2_{L^2_\rho(\mathcal{X};\mathcal{Y})}} = 
    \inf_{\alpha \in \R_+ }\left\{\alpha : \norm{A}^2_{L^\infty_\rho(\mathcal X; \mathcal Y)} \le \alpha \norm{A}^2_{L^2_\rho(\mathcal X; \mathcal Y)} \hspace{0.25cm} (\forall A \in V)\right\}.
\end{align}
Both $\kappa_w$ and $\kappa^{\infty}_w$ depend only on $V$ and $w$, not on the choice of orthonormal basis $\{\Phi_n\}_{n\in [N]}$ \cite{adcock2022towards}. For further details on the Christoffel function, we refer the reader to \cite{narayan2017christoffel}. The constant $\kappa^\infty_w$ for general $(w,V)$ can be infinite, but a strict lower bound is easily obtainable. Since
\begin{align*}
\kappa_w^{\infty} \ge \kappa_w(f) &=
\sum_{n\in[N]}w(f)\norm{\Phi_n(f)}_{\mathcal Y}^2
  \hspace{0.5cm} (\rho \text{ a.e.}),
\end{align*}
then the assumptions \eqref{eq:asymptotic-assumptions} imply
$$\kappa_w^{\infty} = \int_{\mathcal X} \kappa_w^{\infty}\dif \mu  \ge \int_{\mathcal X}\kappa_w(f)\dif \mu(f) = \int_{\mathcal X}\sum_{n\in[N]}\norm{\Phi_n(f)}^2_{\mathcal Y}\dif \rho(f) = \sum_{n\in [N]}\norm{\Phi_n}^2_{L^2_\rho(\mathcal X;\mathcal Y)} = N.$$
The following result articulates how $\kappa_w^\infty$ influences pre-asymptotic stability of least squares.
\begin{thm}[Sample Complexity for Weighted Least Squares Stability]\label{thm:Samp Complexity}
Let $\epsilon, \delta \in (0,1),$ $\mu$ be a probability measure on $\mathcal X$ satisfying \eqref{eq:asymptotic-assumptions}, and $\{f^i\}_{i\in[M]}$ be iid samples drawn under $\mu$. If $w$ is the weight functional specified in \eqref{eq:asymptotic-assumptions} and 
\begin{align}\label{eq:SampleComplexity} 
    M &\ge c_{\delta}\kappa^{\infty}_{w} \log\left(\frac{2N}{\epsilon}\right), 
\end{align}
with $c_{\delta} \coloneqq \left[\delta + (1-\delta)\log(1-\delta)\right]^{-1} > 1$, then 
$$\Pr\{\norm{\bs G - \bs I}_2 \le \delta\} \ge 1-\epsilon. $$
\end{thm}
See Appendix \ref{Pf: Sample Complexity} for the proof.  Although the multiplicative factor of $c_\delta$ in the sample complexity estimate \eqref{eq:SampleComplexity} satisfies $\lim_{\delta \to 0} c_\delta = + \infty$, it is quite small for the moderate values of $\delta$ required in practice: e.g., $c_\delta \approx 6.5$ when $\delta = \frac{1}{2}$.

Hence, minimizing $\kappa_w^\infty$ provides a provable strategy for ensuring least squares stability. Moreover, there is an explicit choice of sampling $\mu$ that minimizes $\kappa_w^\infty$.

\begin{thm}[The Optimal Sampling Measure]
Defining the sampling measure $\mu$ and corresponding weight functional $w$ via
\begin{align}\label{eq:opt weight and measure}
  \dif \mu(f) &= \frac{\sum_{n\in [N]}\norm{\Phi_n(f)}^2_{\mathcal Y}}{N}\dif \rho(f), & 
    w(f) &= \frac{\dif \rho}{\dif \mu}(f) = \frac{N}{\sum_{n\in[N]}\norm{\Phi_n(f)}^2_{\mathcal Y}}
\end{align}
satisfies \eqref{eq:asymptotic-assumptions}, and minimizes the weighted Nikolskii constant, resulting in  $\kappa_w^{\infty} = N$.
\end{thm}
\begin{proof}
First, we note that $\mu$ is in fact a probability measure, since
$$ \int_{\mathcal X}\dif \mu = \int_{\mathcal X}\frac{\sum_{n\in[N]}\norm{\Phi_n(f)}^2_{\mathcal{Y}}}{N}\dif \rho(f) = \frac{1}{N}\sum_{n\in[N]}\norm{\Phi_n}^2_{L^2_\rho(\mathcal X;\mathcal Y)} = 1.$$
That $\frac{\dif \mu}{\dif \rho} = w^{-1}$ is immediate from the above as well. Finally, under this choice of weight functional, the Nikolskii constant is given by
$$\kappa_w^{\infty} = \esssup_{f\in (\mathcal X, \rho)} \kappa_w(f) = \esssup_{f\in (\mathcal X, \rho)} \left(\frac{N}{\sum_{n\in[N]}\norm{\Phi_n(f)}^2_{\mathcal Y}}\right)\sum_{n\in[N]}\norm{\Phi_n(f)}^2_{\mathcal Y} = N,$$
which is minimal. 
\end{proof}

We summarize the results of this section in the following corollary. 
\begin{cor}[Optimal Stability and Sample Complexity Estimate]\label{cor:OptStab and Sample Complexity Estiamte} Let $\epsilon, \delta \in (0,1)$ and $\{f^i\}_{i\in [M]}$ be iid samples drawn under the optimal sampling measure \eqref{eq:opt weight and measure}. Then with
$$M\ge c_\delta N\log\left(\frac{2N}{\epsilon}\right),$$
the weighted least squares problem \eqref{eq:discrete-ls} with optimal weight functional as in \eqref{eq:opt weight and measure} satisfies 
\begin{align*}\Pr\left\{\norm{\bs G - \bs I}_2 \le \delta\right\} 
& \ge 1-\epsilon.
\end{align*}
\end{cor}

The result above is quite general. However, as the next result demonstrates, imposing additional structural assumptions on the approximation space $V$ allows for refined sample complexity estimates. 
\begin{thm}[Sample Complexity for Block-Separable Spaces]\label{thm: samp complexity block learning}
Let $\mathcal{Y}_h\subset \mathcal{Y}$ be any finite-dimensional subspace, and let $P\subset L^2_{\rho}(\mathcal X;\R)$ be any $N_{\mathrm{eff}}$-dimensional subspace of square-integrable scalar-valued functions on $\mathcal X$. Consider the tensor-product approximation space
$$V = \mathcal Y_{h}\otimes P \subset L^2_\rho(\mathcal X;\mathcal Y).$$
Let $\epsilon,\delta \in (0,1)$, and let $\{f^i\}_{i\in [M]}$ be iid samples drawn under the optimal sampling measure \eqref{eq:opt weight and measure}. Then with
  $$M \ge c_{\delta}N_{\mathrm{eff}}\log\left(\frac{2N_{\mathrm{eff}}}{\epsilon}\right),$$ 
the weighted least squares problem \eqref{eq:discrete-ls} with optimal weight functional as in \eqref{eq:opt weight and measure} satisfies 
\begin{align*}\Pr\left\{\norm{\bs G - \bs I}_2 \le \delta\right\} \ge 1-\epsilon.
\end{align*}
  In particular, the required sample size depends on  $N_{\mathrm{eff}} = \dim P$ alone, and is independent of $\mathcal{Y}_h.$
\end{thm}
\begin{proof}
  The full proof is provided in \ref{Pf: Block learning}, but here we provide a brief sketch, since we enforce this structure for many of the approximation spaces considered in Section \ref{sec: Construct Approx Space}. With any orthonormal bases $\{\psi_i\}$ of $\mathcal Y_h$ and $\{\varphi_j\}$ of $P$, the rank-one elements $\Phi_{ij}(f) = \varphi_j(f)\psi_i$ form an $L^2_\rho$-orthonormal basis of $V$, and the empirical Gram matrix decomposes block-diagonally as $\bs G = \bigoplus_i \bs G_P.$ Thus, stability reduces to concentration of the scalar block $\bs G_P$, which under the optimal sampling measure depends on $\kappa_{w,P}^{\infty} = \dim P = N_{\mathrm{eff}},$ yielding the bound above. 
\end{proof}
\subsection{Accuracy}\label{ssec: Accuracy}

With Theorem \ref{thm:Samp Complexity} established, we now turn to a detailed discussion of the accuracy of the approximation \eqref{eq:discrete-ls}. First, \cref{thm: L2 error in expectation} characterizes the approximation of the full infinite-dimensional problem. A similar result for a computationally feasible discretized problem, where the output space is restricted to a finite-dimensional subspace \(\mathcal{Y}_h \subset \mathcal{Y}\), is given in \cref{thm: Truncated Approx Error in Prob}. 

Both results rely on the existence of a norm equivalence between \(\norm{\cdot}_{L^2_{\rho,M}}\) and \(\norm{\cdot}_{L^2_\rho}\), established in Lemma \ref{lem:comp_norms}. 
In these statements we do not assume that the sampling measure \(\mu\) is optimal; instead, we highlight the impact of \(\kappa_w^{\infty}\) on the approximation, where applicable, though we reiterate that under the optimal sampling conditions \eqref{eq:opt weight and measure}, \(\kappa_w^{\infty} = N\).

\subsubsection{Approximation of the Infinite-Dimensional Problem}
Recall that stability guaranteed by, e.g., \cref{cor:OptStab and Sample Complexity Estiamte}, only holds with a certain probability. If this condition fails, then we can no longer guarantee equivalence between $\norm{\cdot}_{L^2_\rho}$ and $\norm{\cdot}_{L^2_{\rho,M}}$ on $V$. To address this situation, we could simply aggressively truncate the estimator if this equivalence does not hold:
\begin{align}\label{def: Conditioned LSP} \tilde{K}^{C}_{V}\coloneqq \begin{cases}
    \tilde{K}^*_{V} & \text{if } \norm{\bs{G}-\bs{I}}_2 \le \delta\\
    0 & \text{else}
\end{cases}.\end{align}
A softer truncation would entail thresholding the operator:
\begin{align}\label{def Trunc LSP}
  \tilde{K}^T_{V} &\coloneqq T_{\tau}\circ \tilde{K}^*_{V}, & T&: \mathcal{Y} \to \mathcal{Y}, & T(g)&=  \frac{g}{\norm{g}_{\mathcal{Y}}} \min\{\norm{g}_{\mathcal{Y}}, \tau\},
\end{align}
where we can choose $\tau > 0$ as any value satisfying
\begin{align}\label{eq:tau-choice}
\norm{K}_{L^\infty_\rho(\mathcal X;\mathcal Y)} \coloneqq \esssup_{f\in \mathcal{X}} \norm{K(f)}_{\mathcal{Y}} \le \tau,
\end{align}
assuming the essential supremum is finite. With these preliminaries behind us, we state the main results. 
\begin{thm}[$L^2_\rho$ Approximation Error in Expectation, cf. \cite{stabilityAccuracyLS,optimalWeightedLS}]\label{thm: L2 error in expectation}
Let $\epsilon, \delta \in (0,1)$ and $\{f^i\}_{i\in [M]}\subset \mathcal X$ be iid samples drawn under a probability measure $\mu$. If $M \ge c_\delta \kappa_{w}^{\infty}\log(2N/\epsilon),$ with $c_\delta$ as in \eqref{eq:SampleComplexity}, then 
\begin{enumerate}
\item $\E\left[\norm{K - \tilde{K}^T_{V}}_{L^2_\rho}^2\right] \le  \left(1 + \frac{2\gamma(N)}{(1-\delta)^2}\right)\epsilon_{V}(K)^2 + \frac{2(\gamma(N) + 1)}{(1-\delta)^2}\norm{\eta}^2_{L^2_\rho}  + 4\tau^2\epsilon,$ and 
\item $\E\left[\norm{K - \tilde{K}^C_{V}}_{L^2_\rho}^2\right] \le  \left(1 + \frac{2\gamma(N)}{(1-\delta)^2}\right)\epsilon_{V}(K)^2 + \frac{2(\gamma(N) + 1)}{(1-\delta)^2}\norm{\eta}^2_{L^2_\rho}  + \epsilon\norm{K}^2_{L^2_\rho},$
\end{enumerate}
where $\gamma(N) \coloneqq (c_{\delta}\log(2N/\epsilon))^{-1}.$ 
\end{thm}
\begin{rem}
In the noiseless case $\eta = 0$, the factor of $2$ in front of $\gamma(N)$ can be removed. 
\end{rem}
See Appendix \ref{Pf: Proof of L2 Err in Expectation} for the proof. 

\subsubsection{Approximation of the Discretized Problem}
Somewhat more general results can be obtained under a $\mathcal{Y}$-discretization model. As mentioned in the introduction, a key application area for this discrete least squares theory lies in the approximation of solution operators to PDEs. Crucially, we are interested in approximations that can be \emph{computed}---something that cannot be accomplished completely for operators $K$ taking values in an infinite-dimensional Hilbert space. Though the noise $\eta$ may in principle represent such discretization error, it is more natural to associate this noise to discretization error in the input space $\mathcal X$. Here we provide a more careful analysis  which explicitly accounts for discretization error in the output space $\mathcal Y$. See \cite{adcock2022towards} for similar results.  

Let $\mathcal Y_h$ be a finite dimensional subspace of $\mathcal Y$. Restricting the admissible set of operators to $L^2_\rho(\mathcal X; \mathcal Y_h) \subset L^2_\rho(\mathcal X;\mathcal Y)$, the  solution to the discrete least squares problem \eqref{eq:discrete-ls} is given by
\begin{align}\label{eq:trunc_disc_LS}
    \tilde{K}_{V_h} \in \argmin_{A\in V_h} \{\norm{K + \eta - A}_{L^2_{\rho,M}}\},
\end{align}
where $V_h$ is an $N$ dimensional subspace of $L^2_\rho(\mathcal X; \mathcal Y_h).$ We note that the $L^2_{\rho,M}$-norm \eqref{eq:discrete-ls} and semi-inner product \eqref{eq:discrete_IP} remain unchanged. Finally, for $K\in L^2_\rho(\mathcal X; \mathcal Y),$ we define
$\Pi_{\mathcal{Y}_h}K\in L^2_\rho(\mathcal X;\mathcal Y_h)$ by the almost everywhere defined operator
$$(\Pi_{\mathcal{Y}_h}K)(f) = \Pi_{\mathcal{Y}_h}(K(f)),$$
where $\Pi_{\mathcal{Y}_h}$ is the orthogonal projector of $\mathcal Y$ onto $\mathcal Y_h$. 

\begin{thm}[Truncated Approximation Error in Probability, cf. \cite{adcock2022towards}]\label{thm: Truncated Approx Error in Prob}
    Let $\{f^i\}_{i\in [M]} \subset \mathcal X$, and let $w: \mathcal X\to \R_+$ be any functional for which $w(f^i)$ is finite for all $i\in [M].$ If there exists a $\delta \in  (0,1)$ such that $(1-\delta)\norm{\cdot}_{L^2_\rho}^2 \le \norm{\cdot}_{L^2_{\rho,M}}$ on a subspace $V_h \subset L^2_\rho(\mathcal X; \mathcal Y_h),$ then the approximation $\tilde{K}_{V_h}$ given by \eqref{eq:trunc_disc_LS} is unique and satisfies
    \begin{equation}
    \begin{aligned}
    \norm{K-\tilde{K}_{V_h}}_{L^2_\rho(\mathcal X;\mathcal Y)} &\le \inf_{A\in V_h}\left\{\norm{K-A}_{L^2_\rho(\mathcal X;\mathcal Y) } +\frac{1}{\sqrt{1-\delta}}\norm{K-A}_{L^2_{\rho,M}}\right\}\\
    &\hspace{0.5cm} +\norm{K-\Pi_{\mathcal{Y}_h}K}_{L^2_\rho(\mathcal X; \mathcal Y)} + \frac{1}{\sqrt{1-\delta}}\norm{\eta}_{L^2_{\rho,M}}.
    \end{aligned}
    \end{equation}
\end{thm}
The above theorem shows that the error in \eqref{eq:discrete-ls} can be decomposed into three terms: The first is a best-approximation error in $V_h$, the second is a discretization error arising from the restriction to the subspace $\mathcal Y_h$, and the third is an error due to noise. The proof may be found in Appendix \ref{Pf: Trunc Approx Error in Prob}. We emphasize that the above result holds for \emph{any} collection of points $\{f^i\}_{i\in [M]} \subset \mathcal X$ and $\emph{any}$ well-defined weight functional $w$, provided that the stability condition holds. Of course, stability can be guaranteed with high probability provided that $(f^i, w, M)$ are taken as in Corollary \ref{cor:OptStab and Sample Complexity Estiamte}.

\section{Operator Basis Construction and Optimal Sampling Procedures}\label{sec: Construct Approx Space}
In this section we prescribe explicit constructions for approximation spaces $V\subset L^2_\rho(\mathcal X;\mathcal Y)$ and explicit procedures for sampling from the optimal measures \eqref{eq:opt weight and measure} associated to these spaces. 
Explicit procedures cannot be expected for arbitrarily general classes of measures, so we assume that \(\rho\) is a product measure to ensure feasibility.
Furthermore, we divide our analysis into two cases: approximation using linear operators and approximation using nonlinear operators. 

\subsection{Setup and Assumptions on \(\rho\)}

Let \(\{\xi_j\}_{j \in \N}\) and \(\{\psi_j\}_{j \in \N}\) be orthonormal bases of \(\mathcal{X}\) and \(\mathcal{Y}\), respectively. The functions \(\xi_j\) are easier to prescribe, as we assume access to knowledge of the input measure \(\rho\) on \(\mathcal{X}\). While the choice of \(\psi_j\) is, in principle, arbitrary, effective approximation requires that the \(\psi_j\) be selected to represent the output functions well.

For now, we focus on the impact of the choice of \(\xi_j\). By Theorem \ref{thm: prop_rand_H_Vars}, the approximation measure \(\rho\) on \(\mathcal{X}\) is fully characterized by the distribution of the \(\rho\)-random generalized Fourier coefficients
\begin{align*}
  f \sim \rho \hskip 10pt \Longrightarrow \hskip 10pt \hat{f}_j \coloneqq\ip{\xi_j, f}_{\mathcal{X}}  \eqqcolon \xi^*_jf \sim \dif \rho_j,
\end{align*}
where $\xi_n^* : \mathcal{X} \rightarrow \C$ is the linear functional for which $\xi_n$ is the Riesz representer, and $\rho_j$ is the $\xi_j$-marginal of $\rho$: that is, the full distribution of $(\hat{f}_1, \hat{f}_2, \ldots )^{\top} \in \R^\infty$ fully characterizes $\rho$. 

Throughout, we assume that 
\begin{align}\label{assump: sq_sum of var} 
\sum_{j \in \N} \sigma^2_j < \infty, \hspace{0.5cm} \text{where} \hspace{0.5cm} \sigma^2_j \coloneqq \rho_{j,2}
\end{align}
denotes the second (uncentered) moment of \(\rho_j\). This condition ensures that \(f \sim \rho\) has a mean in \(\mathcal{X}\) and finite variance on \(\mathcal{X}\). We further assume that \(\sigma_j > 0\) for every \(j\). If any \(\sigma_j = 0\), the corresponding \(\xi_j\) can be removed from the basis without affecting any \(L^2\)-type results under the measure \(\rho\). When the collection \(\hat{f}_j\) is uncorrelated, we have $\mathrm{Cov}(\hat{f}_i, \hat{f}_j) = \sigma_i^2 \delta_{i,j}$. 
Even if the covariance is non-diagonal and known, diagonalization can be achieved numerically through a decorrelation transform applied to \(\hat{f}_j\).

To make computations more feasible, we impose a stronger independence condition on the generalized Fourier coefficients by assuming that \(\rho\) is a product measure:
\begin{align}\label{def: product meas rho} 
\dif \rho = \bigotimes_{j \in \N} \dif \rho_j,
\end{align}
where each \(\rho_j\) is considered as a measure on \(\R\), and \(\rho\) is formally treated as operating on \(\R^\infty\), the infinite \(\ell^2\) sequence of Fourier coefficients of \(f\). Under this assumption, the collection \(\{\hat{f}_j\}_{j \in \N}\) is independent. 

We slightly abuse notation by making the identification:
\begin{align*}
(\mathcal{X}, \ip{\cdot, \cdot}_{\mathcal{X}}, \rho) \cong (\ell^2(\N), \ip{\cdot, \cdot}_{\ell^2}, \bigotimes_{j \in \N} \dif \rho_j),
\end{align*}
even though the left-hand side measures sets in \(\mathcal{X}\), while the right-hand side measures sets in \(\ell^2(\N)\). 
This identification allows us to explicitly represent \(f \in \mathcal{X}\) by \(\hat{f} \in \ell^2\) under the linear isometry \(\xi_j \mapsto e_j\). However, we emphasize that this product structure is not required for any theoretical results and is assumed purely for computational convenience.

\subsection{Linear Operators}
In this paper, we will choose candidate subspaces \(V\) for linear operators that are spanned by basis elements of the form:
\begin{align}\label{eq:Phin-linear}
  \Phi_{\bs n} \coloneqq \sigma_{n_1}^{-1}\psi_{n_2}\otimes \xi_{n_1} = \sigma_{n_1}^{-1}\ip{\xi_{n_1},\cdot}_{\mathcal{X}} \psi_{n_2} = \sigma_{n_1}^{-1}\psi_{n_2}\xi_{n_1}^*,
\end{align}
where $\bs{n} = (n_1, n_2) \in \N^2$ is some given multi-index. The linear operators $\Phi_{\bs{n}}$ are the analogue of rank-one matrices: they extract the index-$n_1$ Fourier coefficient of the input, scale it by $\sigma_{n_1}^{-1}$ for normalization, and this forms the index-$n_2$ Fourier coefficient for the output.  The approximation space $V$ we consider in this section is the $N$-dimensional subspace given by,
\begin{align}\label{eq:V-linear}
  V = \mathrm{span}\left\{ \Phi_{\bs n} \;\;\big|\;\; {\bs n} \in \Lambda \right\},
\end{align}
where $\Lambda \subset \N_0^2$ satisfies $|\Lambda| = N$.

\subsubsection{Density of Finite-Rank Projection Operators}\label{sec:lin_sub_w_prod_meas}
Since initially we will discuss using the specific class of linear operators \eqref{eq:Phin-linear}, it is natural to ask whether this set of candidates can accurately approximate a general, say bounded, linear operator.  
A linear approximation could be formed by truncating the input:
\begin{gather*}
  (K\Pi_n)f \coloneqq \sum_{i\in [n]} \hat{f}_i K\xi_i, \hskip 10pt K \in \mathfrak{B}(\mathcal{X}; \mathcal{Y}) \\ 
  \mathfrak{B}(\mathcal{X}; \mathcal{Y}) \coloneqq \left\{ K \in \mathcal{L}(\mathcal{X}, \mathcal{Y}) \;\;\big|\;\; \exists\; C \in [0, \infty) \textrm{ such that } \left\| K x \right\|_{\mathcal{Y}} \leq C \|x\|_{\mathcal{X}} \;\forall\; x \in \mathcal{X}\right\}
\end{gather*}
where $\mathcal{L}(\mathcal{X},\mathcal{Y})$ denotes the collection of linear operators from $\mathcal{X}$ to $\mathcal{Y}$, and $\Pi_n$ is the orthogonal projection onto $\mathrm{span}\{\xi_i\}_{i\in [n]}$. This particular linear approximation is salient since it lies in the subspace $V$ defined in \eqref{eq:V-linear} if we choose both $\{\psi_j\}$ and $\Lambda$ appropriately.

The main question we consider is whether such a truncation of $K$ to $K \Pi_n$ is sensible in our $L^2(\mathcal{X};\mathcal{Y})$ setting. We formalize this in the following result, the proof of which may be found in \ref{Pf: Bochner_density_linear}.

\begin{thm}[Bochner Norm Approximation by Finite-Rank Projections]\label{thm: Bochner density Finite Rank in Linear Op}
Let $\mathcal{X}$ and $\mathcal{Y}$ be separable Hilbert spaces, and let $\rho$ be a Borel probability measure on $\mathcal{X}$ satisfying
$$\int_{\mathcal{X}} \|f\|_{\mathcal{X}}^2 \, \dif \rho(f) < \infty.$$
Let $\{\xi_j\}_{j=1}^\infty \subset \mathcal{X}$ be a complete orthonormal system, and let $\Pi_n : \mathcal{X} \to \mathcal{X}$ denote the orthogonal projection onto $\mathrm{span} \{\xi_1, \dots, \xi_n\}$. Then for any bounded linear operator $K \in \mathfrak{B}(\mathcal{X}; \mathcal{Y})$, we have
$$\lim_{n \to \infty} \|K - K \Pi_n\|_{L^2_\rho(\mathcal{X}; \mathcal{Y})} = 0.$$
\end{thm}
The density result above is a universal approximation statement for target \textit{linear} operators.  Many operators of interest -- including integral kernel operators -- are Hilbert-Schmidt and may be well approximated in such a space. However, a subspace $V$ comprised of linear operators as in \eqref{eq:V-linear} may do a poor job of approximating a highly non-linear operator in $L^2_\rho(\mathcal X; \mathcal Y)$. 

\subsubsection{Sampling}\label{sssection: Linear Approximation space}

As shown above, \(\text{span}\{\psi_{n_2} \otimes \xi_{n_1}\}_{\bs{n} \in \N^2}\) is \(L^2_\rho\)-dense in \(\mathfrak{B}(\mathcal{X}; \mathcal{Y})\). Therefore, the subspaces \(V\) defined as in \eqref{eq:V-linear} are sensible candidates for approximating linear operators. To ensure orthonormality of the basis functions \(\Phi_{\bs{n}}\), 
we require that the \(\hat{f}_j\) are uncorrelated and centered (mean-zero), and that \(\{\psi_j\}_{j \in \N}\) are $\mathcal{Y}$-orthonormal. Under these conditions, the basis functions \(\Phi_{\bs{n}}\) are orthonormal.
\begin{align*}
   \ip{\Phi_{\bs n}, \Phi_{\bs m}}_{L^2_\rho(\mathcal X; \mathcal Y)} & = \E_{f\sim \rho} \ip{\Phi_{\bs n}f, \Phi_{\bs m}f}_{\mathcal Y} %\\
   = \E \ip{\sigma_{n_1}^{-1}\psi_{n_2} \hat{f}_{n_1}, \sigma_{m_1}^{-1}\psi_{m_2} \hat{f}_{m_1}}_{\mathcal Y} = \delta_{\bs n, \bs m}
\end{align*}
This orthonormality property can be exploited for efficient sampling of the optimal measure associated with \(V\) in \eqref{eq:V-linear}.

These orthonormal bases enable the construction of an efficient and straightforward sampling procedure.
To achieve this, we recall that \(\rho\) is assumed to be a product measure with marginals \(\rho_j\) associated with the random variables \(\hat{f}_j\) and define a modified class of univariate probability measures \(\tilde{\rho}_j\) with densities given by:
\begin{align}\label{eq:lin-ind-density}
\dif \tilde{\rho}_j(\hat{f}_j) = \frac{\hat{f}_j^2}{\sigma_j^2} \dif \rho_j(\hat{f}_j).
\end{align}
Next, using the explicit representation \eqref{eq:opt weight and measure} of the optimal measure \(\mu\), we express \(\mu\) as:
\begin{align*}
\dif \mu(f) &= \frac{1}{N} \sum_{\bs{n} \in \Lambda} \left\| \Phi_{\bs{n}} f \right\|_{\mathcal{Y}}^2 \dif \rho(f) 
= \frac{1}{N} \sum_{\bs{n} \in \Lambda} \frac{1}{\sigma^2_{n_1}} |\hat{f}_{n_1}|^2 \bigotimes_{j \in \N} \dif \rho_j(\hat{f}_j) \\
&= \frac{1}{N} \sum_{\bs{n} \in \Lambda} \dif \tilde{\rho}_{n_1}(\hat{f}_{n_1}) \dif \rho_{\backslash n_1}(\hat{f}_{\backslash n_1}),
\end{align*}
where \(\rho_{\backslash j}\) denotes the measure associated with \(\bigotimes_{k \in \N \backslash \{j\}} \dif \rho_k\), and \(\hat{f}_{\backslash j} = \left(\hat{f}_k\right)_{k \in \N \backslash \{j\}}\).

Finally, to further simplify the structure of \(\mu\), we define weights \(v_j = v_j(\Lambda)\), revealing that \(\mu\) is a mixture of measures:
\begin{align*}
v_j &\coloneqq \frac{\left| \left\{\bs{n} \in \Lambda \;\big|\; n_1 = j \right\}\right|}{|\Lambda|}, &
\dif \mu(f) &= \sum_{n_1 \in \Pi_1\Lambda} v_{n_1} \left(\dif \tilde{\rho}_{n_1}(\hat{f}_{n_1}) \dif \rho_{\backslash n_1}(\hat{f}_{\backslash n_1})\right).
\end{align*}

Since \(\sum_j v_j = 1\), \(\dif \mu\) is a mixture of measures, where each individual measure is a product of univariate measures. This structure enables efficient sampling, and a detailed sampling method is provided in \cref{sec:approx_induced_samp_LS}.

\subsubsection{Relationship to Kernel Learning}\label{ssec: lin_kernel_learning}
Using linear subspaces $V$ as described in this section is equivalent to a particular type of kernel learning. Let $\mathcal{X}$ be a Hilbert space of functions over domain $\Omega_x$ with respect to measure $\nu$, and $\mathcal{Y}$ be a Hilbert space of functions over domain $\Omega_y$ with respect to measure $\eta$. Formally:
\begin{align}\label{eq:func_space_defs}
  \mathcal{X} &= H_\nu(\Omega_x) \textrm{ over measure space } (\Omega_x,\mathscr B(\Omega_x, \norm{\cdot}_{\Omega_x}), \nu), \\
  \mathcal{Y} &= H_\eta(\Omega_y) \textrm{ over measure space } (\Omega_y , \mathscr B(\Omega_y, \norm{\cdot}_{\Omega_y}), \eta)
\end{align}
For simplicity, let both function spaces to be equipped with the standard $L^2$ inner product, so that
\begin{align}\label{eq:linear-kernel-operator}
  (\Phi_{\bs{n}}(f))(y) &= \int_{\Omega_x}k_{\bs n}(x,y)f(x)\dif \nu(x), &
  k_{\bs n}(x,y) &= \sigma^{-1}_{n_1}\psi_{n_2}(y)\xi_{n_1}(x), & y &\in \Omega_y.
\end{align}
Then our least squares problem \eqref{eq:discrete-ls} is equivalent to the  kernel learning problem
\begin{align*}
  \argmin_{k \in K_{\Lambda}} \frac{1}{M}\sum_{i\in [M]}w(f^i)\int _{\Omega_y}\left(g^i(y) - \int_{\Omega_x} k(x,y) f^i(x) \dif \nu(x)\right)^2\dif \eta(y),
\end{align*}
with $g^i$ as in \eqref{eq:data_model},
which selects a kernel from the hypothesis class,
\begin{align*}
  K_\Lambda = \left\{ \sum_{{\bs n} \in \Lambda} c_{\bs n} k_{\bs n}(x,y) \;\;\big|\;\; (c_{\bs n})_{{\bs n} \in \Lambda} \subset \R^N \right\}.
\end{align*}

\subsection{Nonlinear Operators}

To prepare for the construction of nonlinear operators, we build on the linear kernel operators introduced in the previous section. From this point onward, we make specific choices for $\mathcal{X}$ and $\mathcal{Y}$ as follows:
\begin{align*}
  \mathcal X = L^2_\nu(\Omega_x; \R) \hspace*{0.5cm}\text{and}\hspace*{0.5cm} \mathcal Y = L^2_\eta(\Omega_y;\R),
\end{align*}

\subsubsection{A First Attempt: Multilinear Operators}
A natural candidate for generalization of \eqref{eq:linear-kernel-operator} would be the following type of nonlinear kernel operator,
\begin{align}\label{eq:multilinear-operator}
  (M_k f)(y) &= \int_{\times^{r}\Omega_x} \prod_{j=1}^r f(x_j)k(\mathbf x, y)\bigotimes_{j=1}^r \dif \nu( x_j),  &
  {\mathbf x} &= (x_1, \ldots, x_r)^{\top} \in \Omega_x^r, & 
  r &\in \N.
\end{align}
For $r = 1$, these are linear kernel operators like those in \eqref{eq:linear-kernel-operator}. We refer to $M_k$ as a \textit{multilinear kernel operator} because it is a multilinear function of the entries in the $r$-fold vector duplication $(f(x_1), \ldots, f(x_r))$.  One can show that if $k\in L^2_{\bigotimes^r\nu \otimes \eta}(\times^r\Omega_x \times \Omega_y),$ then $M_k \in L^2_\rho(\mathcal X;\mathcal Y)$; see e.g., Section 3.5 of \cite{sullivan2015introduction}. Hence, to identify a precise form for the multilinear kernel $k$ that is a direct generalization of the one in \eqref{eq:linear-kernel-operator}, we introduce the set of infinite multi-indices with finite support:
\begin{align}\label{eq: inf_multi_idx_finite_supp}
  \N_F^\infty &\coloneqq \left\{ \bs{n} = (n_0, n_1, n_2, \ldots )^{\top} \in \N_0^\infty \;\;\big|\;\; n_0 \in \N \textrm{ and } \sum_{j \in \N} n_j < \infty \right\}.
\end{align}
For a multi-index $\bs n$, we denote $\bs n_{\sim 0}$ the multi-index with the first element truncated: i.e., $(n_1, n_2, \dots).$
Then for $\bs{n} \in \N_F^\infty$ with $r = \|\bs{n}_{\sim 0}\|_0$, we identify $k_{\bs{n}}$ as an $r$-fold multilinear kernel,
\begin{align}\label{eq:kn-multilinear}
  k_{\bs n}(\mathbf x,y) = \psi_{n_0}(y) \prod_{j \in \N} \prod_{q = 1 + \sum_{k \in [j-1]} n_k}^{n_j +\sum_{k \in [j-1]} n_k} \xi_{j}(x_q) \hskip 5pt \Longrightarrow \hskip 5pt 
  (M_{k_{\bs n}}f)(y) = \psi_{n_0}(y)\prod_{j \in \N} \hat{f}_j^{n_j}
\end{align}
For example, if $\bs{n} = (3, 2, 0, 3, 0, 0, \ldots)$, then $(M_{k_{\bs n}}f)(y) = \psi_3(y) \hat{f}_1^2 \hat{f}_3^3$. We can then fix a finite set $\Lambda \subset \N_F^\infty$ and define a subspace of \textit{nonlinear} operators $V = \mathrm{span}\left\{ M_{k_{\bs{n}}}\right\}_{\bs{n} \in \Lambda}$. 

This construction in principle achieves what we desired: an explicit and computationally encodable set of (polynomial) nonlinear operators. However, sampling from the optimal measure is quite difficult with this construction. In the next section, we introduce a simpler strategy that avoids these requirements and constructs an orthogonal basis for essentially the same class of multilinear kernel operators. 

\subsubsection{Rank-One Orthogonal Polynomial Operators}
To address the challenges of orthogonalizing the multilinear polynomial operators introduced in the previous section, we reformulate these operators to be inherently orthogonal. Since the operators defined in \eqref{eq:multilinear-operator} with $k=k_{\bs{n}}$ depend on a finite number of Fourier coefficients, orthogonal polynomials in finitely many variables naturally arise as a key tool in this reformulation.

To construct these orthogonal polynomials, we continue to assume the measure $\rho$ on $\mathcal{X}$ has product structure as in \eqref{def: product meas rho}, but now make the additional assumption that for every $j \in \N$, $\dif {\rho}_j$ has finite moments of all orders, and that the cumulative distribution function $x \mapsto \rho_j((-\infty, x])$ has an infinite number of points of increase on $\R$. Under these conditions, there exists a family of univariate polynomials, $\{p_n^j\}_{n\in \N}$, satisfying 
\begin{align}\label{eq:op-def}
  \ip{p_n^j, p_m^j}_{L^2_{\rho_j}(\R)} &= \delta_{n,m}, & \deg p_n^j &=n, & p_n^j(x) &= \sum_{k \in [n]_0} c^j_{n,k} x^k, & c^j_{n,n} &> 0.
\end{align}
Note that $j$ in $p^j_{n}$ and $c^j_{n,k}$ is an index and not an exponent. Since $\rho_j$ is a probability measure, then $p_0^j \equiv 1$ for every $j$. Given a multi-index $\bs \alpha \in \N_0^\infty$ with $|\bs \alpha| < \infty$, we define,
\begin{align}\label{def: multivar PolyOP}
  P_{\bs \alpha}(f) &=\psi_{ \alpha_0}\prod_{j\in \N} p_{ \alpha_j}^j(\hat{f}_j), & \deg P_{\bs \alpha} &\coloneqq \sum_{j \in \N} |\alpha_j|,
\end{align} 
which is a nonlinear operator mapping $f \in \mathcal{X}$ to a scaled multiple of $\psi_{\alpha_0} \in \mathcal{Y}$. A straightforward computation shows that 
  $\ip{P_{\bs \alpha}, P_{\bs \beta}}_{L^2_\rho(\mathcal X; \mathcal Y)} = \delta_{\bs \alpha,\bs \beta}$, 
so that $\{P_{\bs \alpha}\}_{|\bs \alpha| < \infty}$ is an orthonormal set of operators. Hence, given $\Lambda \subset \N_F^\infty$ with $|\Lambda| = N$, we can define $V$ as an $N$-dimensional space of nonlinear (polynomial) operators with an easily identifiable orthonormal basis: 
\begin{align}\label{def: nonlinear approx space}
  V &= \mathrm{span}\{P_{\bs\lambda}\;\big|\; \bs \lambda \in \Lambda\}, &
  P_{\bs \alpha} &= \sum_{\substack{\bs \beta \in \N_F^\infty \\ \beta_0 =  \alpha_0, \bs \beta_{\sim 0} \leq  \bs \alpha_{\sim 0}}} c_{\bs \alpha,\bs \beta} M_{k_{\bs \beta}}, & 
  c_{\bs \alpha,\bs \beta} &= \prod_{j \in \N} c^j_{ \alpha_j, \beta_j},
\end{align}
where the last two equations above result by inspecting \eqref{def: multivar PolyOP}, \eqref{eq:op-def}, and \eqref{eq:kn-multilinear}.
Hence, the orthogonal polynomial operators \eqref{def: multivar PolyOP} are essentially equivalent to the multilinear ones, and in particular for any finite $\Lambda \subset \N_F^\infty$,
\begin{align*}
  \Lambda \textrm{ is monotone lower } \Longrightarrow \mathrm{span}\{P_{\bs \lambda} \;\big|\; \bs \lambda \in \Lambda\} = \mathrm{span}\{ M_{k_{\bs \lambda}}\;\big|\;\bs \lambda \in \Lambda\},
\end{align*}
where a monotone lower $\Lambda$ satisfies $(\Lambda - \mathbf{e_j}) \cap \N_F^\infty \subset \Lambda$ for every $j \in \N$. Finally, the connection between $P_{\bs{\alpha}}$ and $M_k$ in \eqref{def: nonlinear approx space} reveals that $P_{\bs \alpha}$ is a particular type of nonlinear kernel operator of the form \eqref{eq:multilinear-operator}:
\begin{align*}
  (P_{\bs \alpha} f)(y) &= \int_{\times^r \Omega_x} \prod_{j \in [r]} f(x_j) \widetilde{k}_{\bs \alpha}(\bs{x}, y) \bigotimes_{j=1}^r \dif\nu(x_j), & 
  \widetilde{k}_{\bs \alpha}(\bs{x},y) &\coloneqq \sum_{\substack{\bs \beta \in \N_F^\infty \\ \beta_0 =  \alpha_0, \bs \beta_{\sim 0} \leq \bs \alpha_{\sim 0}}} c_{\bs \alpha,\bs \beta} k_{\bs \beta}(\bs{x},y),
\end{align*}
where $k_{\bs\beta}(\bs{x},y)$ is as in \eqref{eq:kn-multilinear}.

\subsubsection{Density of Orthogonal Polynomial Operators} 
Our construction of $P_{\bs \alpha}$ was motivated by analytical expediency. As in Section \ref{sec:lin_sub_w_prod_meas}, we now ask whether these operators are also sufficiently expressive so as to approximate any element of $L^2_\rho(\mathcal X; \mathcal Y)$. This question connects to a classical problem in approximation theory -- the completeness of polynomial chaos expansions -- and, in particular, to the \emph{moment determinacy} of the approximation measure $\rho.$ 

A probability distribution $\pi$ on $(\R, \mathscr{B}(\R))$ satisfies the moment determinacy property if it is uniquely determined by its sequence of moments $\{\pi_k\}_{k \in \N_0}$, with $\pi_k = \int_{\R}x^k\dif \pi(x)$, i.e., if no other Borel probability measure has the same moment sequence. (Equivalently, $\pi$ has a unique solution to the Hamburger moment problem.) This property allows us to conclude operator density.

\begin{thm}[Density of Orthogonal Polynomial Operators in $L^2_\rho(\mathcal{X}; \mathcal{Y})$]\label{thm:nonlinear-operator-density}
Let $\mathcal{X}, \mathcal{Y}$ be separable Hilbert spaces with complete orthonormal bases $\{\xi_j\}_{j\ge 1}$ and $\{\psi_k\}_{k\ge 0}$. Let $\rho = \bigotimes_{j= 1}^{\infty}\rho_j$ be a product measure on $\mathcal{X}$ such that:
\begin{enumerate}[label = (\roman*)]
    \item The coordinates $\hat{f}_j\coloneqq \ip{\xi_j, f}_{\mathcal{X}}$ are independent with laws $\rho_j$, 
    \item Each $\rho_j$ has finite moments of all orders, a distribution function with an infinite number of points of increase on $\R$, and is moment determinate.
\end{enumerate}
  For each $j$, let $\{p_n^j\}_{n\ge 0}$ be the orthonormal polynomials in $L^2_{\rho_j}(\R),$ and for $\bs \alpha \in \N_F^{\infty}$.
  Then $\{P_{\bs \alpha}\}_{\bs \alpha \in \N_F^{\infty}}$, with $P_{\bs{\alpha}}$ as in \eqref{def: multivar PolyOP} is a complete orthonormal basis of $L^2_\rho(\mathcal X; \mathcal Y)$. In particular, for any $K \in L^2_\rho(\mathcal X; \mathcal Y)$ and any exhaustion $\{\Lambda_N\}_{N\ge 1}$ of $\N_{F}^{\infty}$ by finite sets: 
$$\lim_{N\to \infty} \int_{\mathcal{X}}\norm{K(f) - \sum_{\bs \alpha \in \Lambda_N} c_{\bs \alpha} P_{\bs \alpha}(f)}_{\mathcal Y}^2\dif \rho(f) = 0,$$
where $c_{\bs \alpha} = \ip{K, P_{\bs \alpha}}_{L^2_{\rho}(\mathcal X; \mathcal Y)}.$
\end{thm}
\begin{proof}
This is an immediate consequence of Corollary 3.10 and Remark 3.12 of \cite{ernst2012convergence}, which establishes completeness of the tensor-product polynomial system for $L^2_\rho(\mathcal X; \R)$ under the given assumptions and extends to $\mathcal{Y}$-valued functions by a standard Bochner-space argument, respectively. 
\end{proof}
Moment determinacy holds for many common measures, including all Gaussian and Jacobi families. Sufficient conditions include compact support, sub-exponential tails ($\int e^{a\abs{x}}\dif \pi < \infty$ for some $a\in \R_+$), or Carleman's condition $\sum \pi_{2n}^{-1/2n} = \infty$ on the even moments; see Theorem 3.4 of \cite{ernst2012convergence} for further details. We summarize that \cref{thm:nonlinear-operator-density} is a universal approximation theorem for any $L^2_\rho$ operator, ensuring adequate model capacity for our chosen approach. Naturally, choosing the `right' finite-dimensional subspace $V$ remains a challenge.

\subsubsection{Sampling}
Since \(\{P_{\bs \alpha}\; : \; \bs{\alpha}\in \Lambda\}\) forms an orthonormal basis for \eqref{def: nonlinear approx space}, we can, similar to \cref{sssection: Linear Approximation space}, derive explicit strategies for sampling from the optimal measure \(\mu\). By definition,
\begin{align*}
    \dif \mu(f) & = \frac{1}{\abs{\Lambda}}\sum_{\bs \alpha \in \Lambda} \norm{P_{\bs \alpha} f}^2_{\mathcal Y}\dif \rho(f)= 
    \frac{1}{\abs{\Lambda}}\sum_{\bs \alpha \in \Lambda}\prod_{j \in \N}\left(p^j_{\alpha_j}(\hat{f}_j)\right)^2 \dif \rho(f)= \frac{1}{|\Lambda|} \sum_{\bs \alpha \in \Lambda} \dif\rho_{\bs \alpha_{\sim 0}}(\hat f),
\end{align*}
where $\rho_{\bs \alpha_{\sim 0}}$ is a product measure on $\ell^2$ defined as 
\begin{align}\label{def: ind prod measure nonlin}
  \dif \rho_{\bs \alpha_{\sim 0}}(\hat{f}) &= \bigotimes_{j\in \N}\dif\widetilde{\rho}_{j, \alpha_j}(\hat{f}_j), & \dif \tilde{\rho}_{j,k}(\hat f_j) =  \left(p^j_{k}(\hat{f}_j)\right)^2 \dif \rho_j(\hat{f}_j).
\end{align}
Thus, $\rho_{\bs \alpha_{\sim 0}}$ is a product measure on $\ell^2$ comprised of marginal measures $\widetilde{\rho}_{j,k}$, with $\widetilde{\rho}_{j,0} = \rho_j$. Therefore, the induced optimal sampling measure $\mu$ is a mixture of measures,
\begin{align*}
  \dif \mu (f) &= \sum_{\bs \beta \in \N^\infty_0} w_{\bs \beta} \dif \rho_{\bs \beta }(\hat f), & w_{\bs \beta} = \frac{\abs{\{\bs \alpha \in \Lambda \text{ : }\bs \beta = \bs \alpha_{\sim 0} \} }}{\abs{\Lambda}}
\end{align*}
and one may therefore sample from $\dif \mu$ by using the same sampling procedure outlined in Section \ref{sssection: Linear Approximation space}. 
For details on efficiently sampling from the composite optimal measures ${\rho}_{\bs \alpha{\sim 0}}$, see \cref{sec:approx_induced_samp_LS}.

\subsection{Discrete Measures}\label{ssec: opt samp from discrete}
In some practical scenarios, the measure $\rho$ is not known, or may not be represented as a product measure. The theoretical results we've presented all still hold, but this makes numerical implementation challenging since our previous sampling schemes are not applicable. One strategy in such cases is to approximate $\rho$ with a discrete measure: Discretization may be one way to approximately identify an orthonormal basis, or $\rho$ may only be known through an available finite database of samples. With this discretized measure setup, there are constructive ways to sample from the optimal measure corresponding to the discretized measure, which results in near-optimal accuracy guarantees with respect to the discretized measure \cite{migliorati_multivariate_2021,adcock2020near}. Stronger assumptions are needed to achieve near-optimal accuracy with respect to the original (in principle unknown) measure $\rho$. We briefly describe the results here.
Let $X = \{x_i\}_{i=1}^S \subset \mathcal X,$ and define the discrete uniform measure
\begin{align}\label{def: discrete measure tau}
  \upsilon = \frac{1}{S}\sum_{k\in [S]}\delta_{x_k},
\end{align}
on $\mathcal X$, where $\delta_x$ is the Dirac mass centered at $x \in \mathcal{X}$. (Using non-uniform weights poses no significant challenge.) For a given, finite-dimensional operator subspace $V$ of dimension $N$, we can no longer identify a basis of orthonormal operators if $\upsilon$ is not a product measure, but since it's a measure of finite support, one can computationally identify the optimal measure, which is also of finite support, as the leverage score distribution of $V$ on the discrete measure $\upsilon$. In the function approximation context, this idea was investigated in \cite{migliorati_multivariate_2021,adcock2020near}. Computationally, this can be accomplished by identifying any orthonormal basis for the range of an $S \times N$ Vandermonde-like matrix; a typical algorithm employs the QR decomposition. Using this procedure, one can optimally sample from $\upsilon$, and recover near-best accuracy guarantees in the discretized $L^2_\upsilon(\mathcal{X}; \mathcal{Y})$ norm. This identifies a direct and feasible computational procedure for (at least approximate) optimal sampling under the discretized measure $\upsilon$ in \eqref{def: discrete measure tau}, and extends the applicability of our operator least squares framework to this setup. If, in addition, one can guarantee $V$-norm equivalence between $L^2_\upsilon(\mathcal{X}; \mathcal{Y})$ and $L^2_\rho(\mathcal{X}; \mathcal{Y})$, then stability with respect to the $L^2_\rho(\mathcal{X}; \mathcal{Y})$ norm can be achieved. We present details of all these arguments in \cref{app:sampling-discrete-measures}.

\section{Numerical Results}\label{sec: Numerical Results}

\begin{figure}[t]
  \centering
  \includegraphics[scale = 1.1]{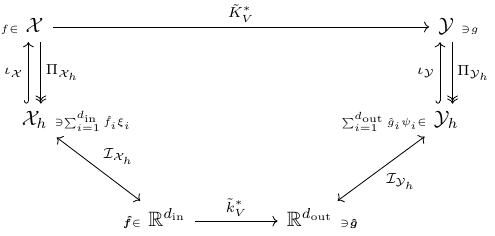}
  \caption{A schematic of the encoder-decoder structure of the computationally feasible operator learning problem. Above, $\Pi_W$ is the orthogonal projection onto a finite dimensional subspace $W$ from a Hilbert space $\mathcal{H}$, $\iota_{\mathcal{H}}$ is the natural embedding of $W$ into $\mathcal{H}$, and $\mathcal{I}_{W}$ is the identification map which---for a fixed, ordered basis of $W$---collects the expansion coefficients of $w\in W$ in a $\dim W$-dimensional vector. We may therefore approximate $\tilde{K}^*_V \approx \iota_{\mathcal{Y}} \circ \mathcal{I}_{\mathcal{Y}_h}^{-1}\circ  \tilde{k}_V^* \circ \mathcal{I}_{\mathcal{X}_h}\circ \Pi_{\mathcal{X}_h}$: that is, $\tilde{K}^*_V$ is the approximate operator between the infinite dimensional spaces $\mathcal{X}$ and $\mathcal{Y}$, whereas $\tilde{k}_V^*$ is a finite dimensional encoding.} 
  \label{fig:Comm-Diagram}
\end{figure}

In this section we evaluate optimally weighted least squares (WLS) estimators in $V\subset L^2_\rho(\mathcal X;\mathcal Y)$ and examine three complementary capabilities of the framework in practice. First, on a linear benchmark (Poisson) we ask not only whether the estimator is stable and accurate, but also whether it \emph{recovers operator structure}—in particular, diagonality in a chosen spectral basis—and whether it reconstructs the associated Green's kernel. Second, on a nonlinear benchmark (viscous Burgers') we study \emph{discretization design}: how the choice of index set and the output truncation dimension affects error as viscosity decreases and the output develops higher-frequency content. Third, on a large nonlinear example (2D incompressible Navier–Stokes) we address \emph{applicability to fixed, precomputed datasets} for which the sampling distribution cannot be controlled, with emphasis on how the evaluation norm influences approximation quality and can serve as an implicit regularizer. Throughout, we report conditioning of the weighted Gram matrix and approximation error in the Bochner norm. 

Computational feasibility requires finite-dimensional encodings. We therefore learn a mapping from a $d_{\mathrm{in}}$-dimensional encoding of $\mathcal X$ to a $d_{\mathrm{out}}$-dimensional encoding of $\mathcal Y$ (cf. \cref{fig:Comm-Diagram}). We use spectral encodings: $f\in\mathcal X$ and $g\in\mathcal Y$ are represented by expansion coefficients in fixed orthonormal bases of $\mathcal X_h$ and $\mathcal Y_h$. The choice of $\mathcal X_h$ is guided by an energy criterion: for a centered product measure $\rho=\bigotimes_j\rho_j$ with coefficient marginals $\hat f_j\sim\rho_j$ and second moments $\rho_{j,2}$, we have $
\mathbb E_{f \sim \rho} \|f\|_{\mathcal X}^2=\sum_{j\ge1}\rho_{j,2}$. 
We truncate to $j \in [d_{\mathrm{in}}]$ so that the partial sum captures a fixed fraction (e.g., $95\%$) of the input energy. We choose $\mathcal Y_h$ in a problem-dependent manner to ensure that truncation error is negligible in the reported norms.

In all experiments we employ the tensor-product operator space
$$
V=P\otimes \mathcal Y_h \subset L^2_\rho(\mathcal X;\mathcal Y),
\qquad
P=\mathrm{span}_{\bs\lambda\in\Lambda}\{p_{\bs\lambda}\}\subset L^2_\rho(\mathcal X_h;\mathbb R).
$$
This implies that $P$ can be spanned by orthonormal functions and so we can sample from the induced optimal measure, see \cref{sec:approx_induced_samp_LS}.
Then by \cref*{thm: samp complexity block learning}, the WLS stability bound depends only on $N_{\mathrm{eff}}:=\dim P$,
$$
M \ge c_\delta N_{\mathrm{eff}}
\log\Big(\tfrac{2N_{\mathrm{eff}}}{\epsilon}\Big)
\quad\Longrightarrow\quad
\mathbb P\left\{\|\bs G-\bs I\|_2\ge \delta\right\} \le \epsilon,
$$
entirely independent of $d_{\mathrm{out}}=\dim\mathcal Y_h$. This mitigates concerns about choosing a sufficiently expressive $\mathcal Y_h$. The second consequence is computational complexity. The weighted Gram matrix $\mathbf G$ decomposes into $d_{\mathrm{out}}$ identical $N_{\mathrm{eff}}\times N_{\mathrm{eff}}$ diagonal blocks, reducing the solve to an $M \times N_{\mathrm{eff}}$ matrix least-squares problem with total cost $\mathcal O \left( N_{\mathrm{eff}}^{3}+N_{\mathrm{eff}}^{2}d_{\mathrm{out}} \right)$ .\footnote{This estimate ignores logarithmic factors in $N_{\mathrm{eff}}$, which are typically dominated by forming the least-squares system when $d_{\mathrm{out}}<M$. It also excludes the cost of generating data and projecting onto $\mathcal X_h$ and $\mathcal Y_h$, which is often the practical bottleneck. For a more complete discussion of the computational complexity, including the dependence on $M$, see \cref{appendix: comp_complexity}. }

In each experiment, the basis of $P$ is indexed by a structured set $\Lambda\subset\mathbb N_0^{d_{\mathrm{in}}}$, such as weighted $\ell^p$ balls :
\begin{equation}\label{eq:ell_p_ball}
  \Lambda_{p,\bs \gamma}(k) \coloneqq \left\{ \bs \lambda \in \mathbb{N}_0^{d_{\mathrm{in}}} : \|\bs{\gamma} \odot \bs{\lambda} \|_p \le k \right\}  
\end{equation}
or weighted hyperbolic cross sets
\begin{equation}\label{eq: hc_set}\Lambda_{\mathrm{HC},\bs{\gamma}}(k) \coloneqq \left\{ \bs \lambda \in \mathbb{N}_0^{d_{\mathrm{in}}} : \|\bs{\gamma} \odot \log (\bs{\lambda} + \bs{1})\|_1 \le \log(k+1) \right\},
\end{equation}
where $\|\cdot\|_p$ is the standard $\ell^p$ norm on $d_{\mathrm{in}}$-vectors, $\odot$ denotes Hadamard (elementwise) multiplication, $\bs{1}$ is the $d_{\mathrm{in}}$-vector of ones, and $\log(\cdot)$ operates on vectors componentwise.
The weights $\bs\gamma \in \R^{d_{\mathrm{in}}}$, $\bs{\gamma} > \bs{0}$, encode anisotropy, which is advantageous when the target operator exhibits uneven directional importance or smoothness across input coordinates, allowing substantial reductions in $\dim P$ for a given accuracy.

Finally, we consistently place a product prior on input coefficients in a fixed orthonormal tensor–product basis
$\{\xi_{\bs j}\}_{\bs j\in\mathbb N^{d_{\mathrm{phys}}}}$ of $\mathcal X=L^2(\Omega)$, with $\Omega\subset\mathbb R^{d_{\mathrm{phys}}}$ rectangular. (We emphasize that $d_{\mathrm{phys}}$ here is the dimension of the physical domain of elements in $\mathcal{X}$, and is not related to $d_{\mathrm{in}}$ or $d_{\mathrm{out}}$.) In particular, we take 
$$
\rho \;=\; \bigotimes_{\bs j\in\mathbb N^{d_{\mathrm{phys}}}} \rho_{\bs j}, 
\qquad \hat f_{\bs j}\sim \rho_{\bs j}=\mathrm{Jac}(\alpha_{\bs j},\alpha_{\bs j}),
$$
where the symmetric Jacobi law is defined for $\alpha > -1$, is supported on $[-1,1]$, has density $
\mathrm d\mathrm{Jac}(\alpha,\alpha)(t)\propto (1-t^2)^{\alpha}\,\mathbf 1_{[-1,1]}(t)\dif t$, and is 
centered with variance $\Var(\hat f_{\bs j})=(2\alpha_{\bs j}+3)^{-1}.$\footnote{This is equivalent to the Beta distribution $\mathrm{Beta}(\alpha+1,\alpha+1)$ affinely mapped to have support on $[-1,1]$.}
We prefer this to a popular, alternative Gaussian process prior because, in the spectral encoding used for WLS, it yields a diagonal coefficient covariance, simple sampling and truncation control, and better conditioning of the weighted Gram matrices under the induced sampling method; by contrast, GP priors typically induce dense coefficient covariances in polynomial spaces, which can degrade conditioning. To control energy decay we set $\alpha_{\bs j}\asymp \norm{\bs j}_p^{k}$ for some $p\in[1,\infty]$ and $k>0$. It is not difficult to show that 
$$
\mathbb E_\rho \norm{f}_{\mathcal X}^2
= \sum_{\bs j\in\mathbb N^{d_{\mathrm{phys}}}}\Var(\hat f_{\bs j})
\asymp \sum_{m=1}^\infty m^{d_{\mathrm{phys}}-1-k},
$$
so that $f\in\mathcal X$ almost surely if and only if $k>d_{\mathrm{phys}}$.

\subsection{Poisson's Equation}\label{ssec: Poisson}

Below, we aim to learn the \emph{linear} solution operator $K = \Delta_{\Omega}^{-1} : \mathcal{X} \to \mathcal{Y}$ for the two-dimensional Poisson problem:
\[
\Delta u = f, \qquad u|_{\partial\Omega} = 0, \qquad \Omega = (0,1)^2.
\]
This example introduces the numerical aspects of our operator learning framework within a simplified linear setting. With this goal, we set $\mathcal X=\mathcal Y=H^1_0(\Omega) \subset L^2(\Omega)$ endowed with the $L^2$ inner product and use the orthonormal sine basis
$$
\xi_{\mathbf n}(x_1,x_2)=2\sin(\pi n_1 x_1)\sin(\pi n_2 x_2),\qquad \mathbf n=(n_1,n_2)\in\N^2,
$$
for which $\Delta$ is diagonal with eigenvalues $-\pi^2(n_1^2+n_2^2)$.
The approximation measure $\rho$ on $\mathcal X$ is defined by independent marginals for the generalized Fourier coefficients $\hat f_{\mathbf n}\coloneqq\langle \xi_{\mathbf n},f\rangle_{L^2(\Omega)}$, with $\hat f_{\mathbf n}\sim \rho_{\mathbf{n}}$, where each $\rho_{\mathbf n}$ is the symmetric Jacobi probability measure $\mathrm{Jac}(\alpha_{\mathbf n},\alpha_{\mathbf n})$ on $[-1,1]$ and
$\alpha_{\mathbf n} := \norm{\mathbf n}_1^{3}.$
This choice enforces $\ell^2$-decay of second moments so that $f\in \mathcal{X}$ ($\rho$-a.s.). For computation we  work in the 1225 dimensional input and output spaces
$$
\mathcal X_h=\mathcal Y_h=\mathrm{span}\{\xi_{\mathbf{n}} : \mathbf n\in[35]^2\}.
$$

We note that $\mathcal{X}_h$ captures $95\%$ of the energy in $\mathcal{X}$. We solve the Dirichlet Poisson problem by a truncated Fourier-sine spectral Galerkin expansion in the eigenbasis, retaining the 1225 modes capturing $95\%$ of the input $L^2$ energy under $\rho$.

The approximation spaces are built from the linear, rank-one $L^2_\rho$-orthonormal operators in \eqref{eq:Phin-linear}.
In particular, for $k\in\{100,200,\dots,1100,1200\}$, we define 
$$
V(k):=\mathrm{span}\{\Phi_{\bs \lambda}:\ \bs \lambda\in \Lambda(k)\times [35]^2\},
$$
where $\Lambda(k)\subset[35]^2$ collects the first $k$ indices in $[35]^2$ in lexicographic order. Thus every $A\in V(k)$ acts only on the $k$ selected input modes.

This setting satisfies the hypotheses of Theorem \ref{thm: samp complexity block learning}, with $N_{\mathrm{eff}}(k)\coloneqq\abs{\Lambda(k)}=k$. We therefore optimally sample $M(k) = \left\lceil c_\delta N_{\mathrm{eff}}(k)\log \left(\frac{2N_{\mathrm{eff}}(k)}{\epsilon}\right)\right\rceil$. We take $\delta=\epsilon=\tfrac12$ (so $c_\delta\approx6.5$). For comparison we also form least-squares systems from $M(k)$ Monte Carlo samples drawn from $\rho$.

\begin{figure}[htb]
  \centering
\includegraphics[width=\linewidth]{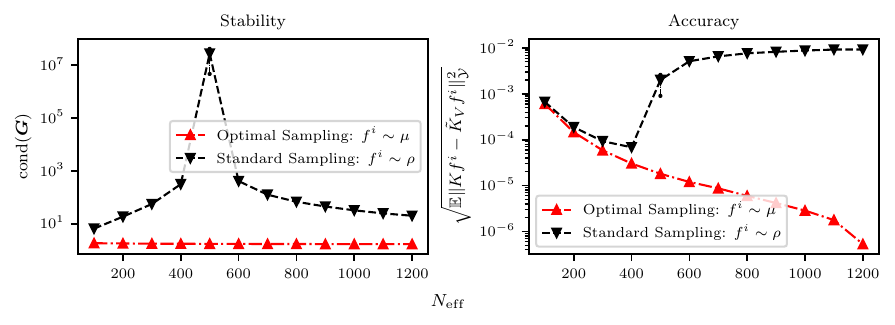}
  \caption{The stability and accuracy of the discrete least squares problem, analyzed as a function of the effective dimension and the sampling method. Here, $M = c_{\frac{1}{2}}N_{\mathrm{eff}}\log(4N_{\mathrm{eff}})$.}
  \label{fig:Cond_Num_Poisson}
\end{figure}

Figure \ref{fig:Cond_Num_Poisson} reports (i) the condition number $\mathrm{cond}(\bs G)$ of the Gram matrix, averaged over 3 trials for each $k$, and (ii) the empirical $L^2_\rho$-Bochner norm error of the estimator $\tilde K_{V(k)}$ on 500 test samples $f^i\sim \rho$:
$$
\norm{K - \tilde{K}^*_{V(k)}}^2_{L^2_\rho(\mathcal X; \mathcal Y)} \approx \frac{1}{500}\sum_{i\in [500]}\norm{Kf^i - \tilde{K}^*_{V(k)}(f^i)}^2_{\mathcal Y}, \qquad f^i \overset{\mathrm{iid}}{\sim}\rho.
$$
As ensured by Lemma \ref{lem:comp_norms}, optimal sampling yields, with high probability, uniformly well-conditioned systems satisfying $\mathrm{cond}(G)\ \le\ \frac{1+\delta}{1-\delta}=3,$
while Monte Carlo sampling exhibits larger and more variable condition numbers which generally grow with the dimension of the approximation space. We note that there is a spike in $\mathrm{cond}(\bs G)$ around $N_{\mathrm{eff}}= 500.$ This is likely a finite-sample effect: for moderate dimensions, with $M\sim N_{\mathrm{eff}}\log N_{\mathrm{eff}}$, random fluctuations in the empirical covariance can produce small eigenvalues. As $N_{\mathrm{eff}}$ increases further, the growing sample size improves concentration around the identity and the condition number decreases again. For low-dimensional approximation spaces, both strategies show decreasing approximation error, and for this simple linear problem, optimally sampling does not aid in accuracy. However, as the dimension of the approximation space increases, the accuracy of the unweighted least squares problem degenerates, again around $N_{\mathrm{eff}}=500.$ Though $\mathrm{cond}(\bs G)$ then decreases as $N_{\mathrm{eff}}$ grows, this only reflects an improvement in numerical stability of the linear solve; it does not mitigate the fact that, under Monte Carlo sampling, we are fitting an increasingly high-dimensional model from relatively few samples. In this regime, the estimator over-fits sampling fluctuations in the training data, so the test error continues to grow even though the Gram matrix is better conditioned. 
 
\begin{figure}[htb]
\centering
  \begin{subfigure}[t]{0.3\textwidth}
    \centering
    \includegraphics[height = 4.5cm]{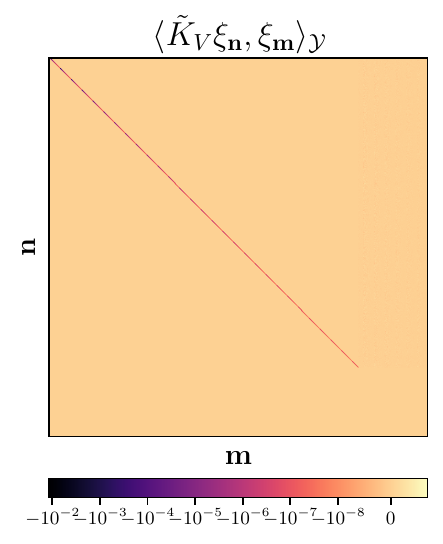}
    \caption{A representation of  $\tilde{K}_{V(1000)}$ as a $1225\times1225$ matrix.}
    \label{fig:Possion_Op}
  \end{subfigure}
  \hfill
  \begin{subfigure}[t]{0.68\textwidth}
    \centering
    
    \includegraphics[width = \linewidth]{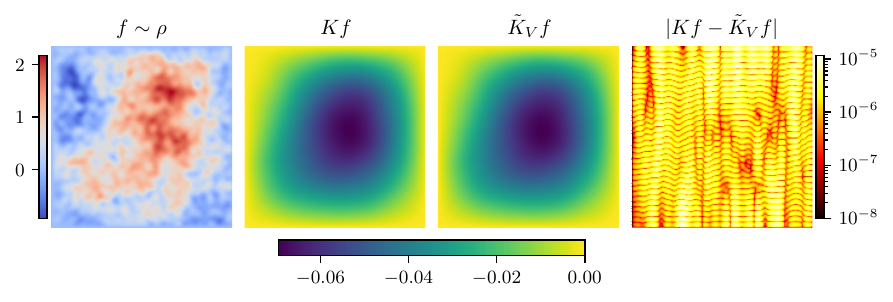}
    
    \caption{A comparison of the action of $K= \Delta^{-1}_{\Omega}$ and $\tilde{K}_{V(1000)}$ on the median error test sample, $f$. Here, $\norm{Kf - \tilde{K}_{V}f}_{\mathcal Y} = 2.845\times 10^{-6}.$}
    \label{fig:Poisson_action}
  \end{subfigure}
  \caption{A representation of $\tilde{K}_{V(1000)}$ and its action on a test sample.}
\end{figure}

In \cref{fig:Possion_Op}, a representation of $\tilde{K}_{V(1000)}$ is given. Observe that the matrix is diagonal, as it should be, since $K= \Delta^{-1}$ is diagonalized by our choice of $\{\xi_{\mathbf n}\}.$ We note that the eigenvalues do not decay monotonically in this representation only because of the choice of lexicographic order. In \cref{fig:Poisson_action}, the action of this approximate operator is compared to the action of the true solution operator on the median error test sample $f\sim \rho,$ achieving an $L^2_{\mathcal{Y}}$-error on the order of $10^{-6}$. This error is due in large part to the final $225$ modes of $f$, which are ignored in the mapping $\tilde{K}_{V(1000)}.$

Finally, we highlight the connection between our operator approximation and classical kernel learning. As discussed in \cref*{ssec: lin_kernel_learning}, when the operator  $K \in L^2_\rho(\mathcal X; \mathcal Y)$ is approximated in a subspace of the form  $V = \mathrm{span}\{\psi_{n_2} \otimes \xi_{n_1}\} $, the least-squares estimator $\tilde{K}^*_V$ corresponds to a learned Green’s function or kernel representation. Concretely, we consider the 1D Dirichlet Poisson solution operator
$K: L^2(0,1)\to H_0^1(0,1)$ defined by $u = Kf,$
where $u$ solves 
$-u''(x) = f(x)$ for $x\in (0,1)$ 
and $u(0) = u(1) = 0.$ This operator admits a Green's function $G$ such that
$$(Kf)(x) = \int_{(0,1)} G(x,y)f(y)\dif y$$
with $G(x, y) = \min\{x,y\} - xy. $
The learned kernel from our method provides an interpretable surrogate that converges to this target as the basis size increases. In \cref*{fig:ker_pt_err}, we visualize the empirical kernel as a bivariate function and directly compare its $L^2$-error against the exact kernel—validating both the spectral convergence and the kernel-learning perspective.

In sum, the optimally sampled WLS fit remains uniformly well conditioned and, in this controlled linear setting, recovers the operator’s diagonal structure. The learned kernel converges to the Green’s function, confirming that the framework reproduces known operator/kernel structure when it exists.

\begin{figure}[htb]
  \centering
    \begin{subfigure}[t]{0.3\textwidth}
      \includegraphics[width = \linewidth]{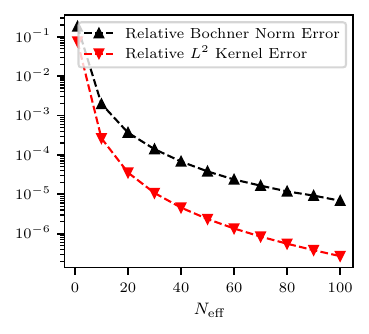}
      \caption{The relative Bochner- and $L^2(\Omega)$-norm errors for the operator and kernel learning problems, as a function of $N_{\mathrm{eff}}$. }
      \label{fig:ker_err}
    \end{subfigure}
    \hfill
    \begin{subfigure}[t]{0.68\textwidth}
      \includegraphics[width = \linewidth]{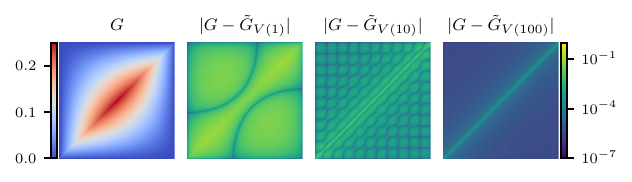}
      \caption{The true Green's kernel, $G$, (left) and the pointwise absolute error in the approximate Green's function, $\tilde{G}_V$ on $\Omega$ for increasing approximation spaces $V(N_{\mathrm{eff}})$ (right).}
      \label{fig:ker_pt_err}
    \end{subfigure}
    \caption{The kernel learning problem for the 1D Poisson problem with Dirichlet boundary conditions. Here, $\mathcal{X}_h = \mathcal{Y}_h = \mathrm{span}_{n\in [208]}\{\sqrt{2}\sin(n\pi x)\}$, and $\rho = \otimes_{n\ge 1} \mathrm{Jac}(\alpha_n, \alpha_n)$ with $\alpha_n = n^{2}$. We note that $\mathcal{X}_h$ captures $99.5\%$ of the forcing energy, and that $M(208)$ optimally sampled forcing terms were used for training, whereas 500 Monte Carlo samples of the forcing were used for testing.}
\end{figure}

\subsection{Viscous Burgers' Equation}\label{ssec: Burgers}
In this section, we aim to learn a solution operator for the viscous Burgers' equation: that is, the map pushing an initial condition $u_0$ to its solution at a fixed time $T$. In particular, we consider the map $K:\mathcal{X}\to \mathcal{Y}$ for which $K(u_0) = u(\cdot, T)$, where $u(\cdot, t)$ satisfies 
\begin{align*}
\begin{cases} 
u_t + uu_x = \nu u_{xx}, & (x,t) \in \Omega \times (0,T], \\
u = 0, & (x,t) \in \partial \Omega \times (0,T].
\end{cases}
\end{align*}
We fix $\Omega = [0,1]$, $T = 0.2$, and explore viscosity parameters  $\nu = 10^{-1}, 10^{-2}$, and $10^{-3}$. This example demonstrates the effectiveness of our proposed method in approximating nonlinear operators using the optimal sampling scheme. 

We generate training data for the viscous Burgers' map via a sine spectral Galerkin discretization $$u(x,t) \approx \sum_{j=1}^{d_{\mathrm{solve}}} \hat{u}_j(t)\sqrt{2}\sin(j\pi x),$$ combined with a first–order IMEX Euler scheme that treats the viscous term implicitly and the nonlinear flux explicitly in a pseudospectral fashion using an oversampled quadrature grid.
For each initial condition, the prescribed Galerkin coefficients are zero–padded to a higher–dimensional $d_{\mathrm{solve}} > 10\max\{d_{\mathrm{in}}, d_{\mathrm{out}}\}$ solver space to mitigate aliasing,  and the resulting state is truncated back to the first $d_{\mathrm{out}}$ modes.

With this goal, we set the input and output spaces to be $\mathcal X=\mathcal Y=H^1_0(\Omega) \subset L^2(\Omega)$ endowed with the $L^2$ inner product and work in the orthonormal sine basis
$$\{\xi_n(x)\}_{n\in \N} \coloneqq \{\sqrt{2}\sin(n\pi x)\}_{n\in \N}.$$ We define the approximation measure $\rho$ on $\mathcal X$ via the independent marginals $\ip{\xi_{n}, u_0}_{L^2(\Omega)}\sim  \rho_{n}$, where each $\rho_n = \mathrm{Jac}(\alpha_n, \alpha_n)$ is the symmetric Jacobi measure on $[-1,1]$ with $\alpha_n \coloneqq n^{2},$
which enforces $\ell^2$ decay. 

We form approximation subspaces of nonlinear operators by the span of the $L^2_\rho$ orthonormal operators $P_{\bs{\lambda}}$ defined in \eqref{def: multivar PolyOP}. (The corresponding orthogonal polynomial families are symmetric Jacobi/Gegenbauer polynomials.)
For the numerical experiments, we restrict to input and output dimensions of $d_{\mathrm{in}} = 20$ and $d_{\mathrm{out}} = 150$, respectively. This choice of $\mathcal{X}_h $ captures $95\%$ of the input energy.
We restrict the univariate polynomial degree defining each basis operator $P_{\bs \lambda}$ by $\lambda_j \in [10].$
Explicitly, we consider approximation spaces $V(k;\ast)$ generated by the index sets of the form 
$$\Lambda(k;\ast) \coloneqq [150]\times \left(\Lambda_{\ast, \bs \gamma}(k) \cap \Lambda_{{\infty}}(10)\right) \subset [150]\times [10]^{20}, \qquad \ast \in \{\mathrm{HC},1\},$$
as defined in \eqref{eq: hc_set} and \eqref{eq:ell_p_ball}.
The anisotropy parameter above is defined by the evenly spaced multi-index
$$\gamma_j = 1 - (j-1)\Delta, \quad \Delta = \frac{0.99}{20}, \quad j \in [20],$$
which de-emphasizes higher degree input frequency modes $\xi_j$. The choice of $\ast$ also effects the regularization: the weighted hyperbolic cross index set de-emphasizes higher degree univariate polynomial degrees $\lambda_j$ far more than the weighted $\ell^1$-ball. This setting falls under the purview of Theorem \ref{thm: samp complexity block learning}, so the sample complexity is driven by the \emph{effective dimension}
$$N_{\mathrm{eff}}(k;\ast) = \abs{\Lambda_{\ast, \bs \gamma}(k) \cap \Lambda_{{\infty}}(10)}.$$
For this problem, we will take $M(k; \ast) = \lceil N_{\mathrm{eff}}(k; \ast) \log N_{\mathrm{eff}}(k; \ast)\rceil$
samples of $u_0.$ This choice of $M$ is insufficient to theoretically guarantee stability, and therefore to apply, e.g., Theorem \ref{thm: Truncated Approx Error in Prob}, but it empirically suffices in practice, as we now demonstrate.

\begin{figure}[htb]
\centering
  \begin{subfigure}[t]{0.48\textwidth}
    \includegraphics[width = \linewidth]{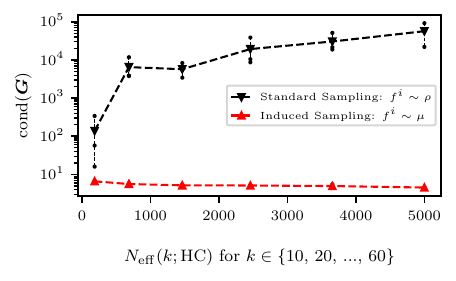}
    \caption{The condition number of the Gramian for subspaces of increasing dimension. Here, $\nu = 0.1.$}
    \label{fig:Cond_Num_Burgers}
  \end{subfigure}
  \hfill
  \begin{subfigure}[t]{0.48\textwidth}
    \includegraphics[width = \linewidth]{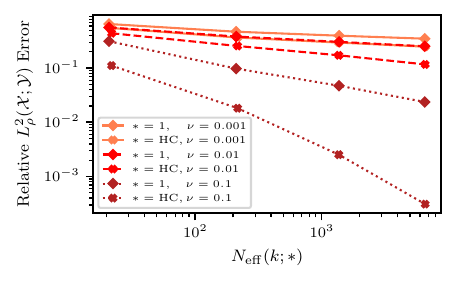}
    \caption{The (empirical) relative $L^2_\rho$ error 
    on $1000$ test samples $u_0^i\sim \rho$ for various subspaces $V(k;\ast)$ and viscosity values $\nu$. }
    \label{fig:Test_Err_Burgers}
  \end{subfigure}
  \caption{Stability and accuracy of the approximation of the Burgers' solution operator.}
\end{figure}

Figure \ref{fig:Cond_Num_Burgers} reports the condition number $\mathrm{cond}(\bs G)$ of the Gram matrix, averaged over 3 trials for each $k$, for $\ast = \mathrm{HC}$ and $k \in \{10,20, \dots, 60\}.$ In each case, stability is compared under optimal and standard sampling. Note that in this nonlinear setting, optimal sampling is vital in ensuring that the discrete least squares problem remains stable; moreover, even though $M(k;\mathrm{HC})$ is theoretically undersampling, the Gram matrices corresponding to optimally sampled initial conditions remain well-conditioned.

Figure \ref{fig:Test_Err_Burgers} shows the \emph{relative} empirical $L^2_\rho$-Bochner norm for the estimator $\tilde{K}_{V(k; \ast)}$ on 1000 test samples $u_0^i\sim \rho$ as a function of $k, \ast$ and the viscosity parameters $\nu$. In each case, $\tilde{K}_{V(k;\ast)}$ was constructed through $M(k,\ast)$ optimally drawn observations, and care was chosen in $(k,\ast)$ to ensure that subspaces of nearly equal dimension are compared. Observe that decreasing viscosity $\nu$ increases the difficulty of the approximation problem, as higher frequency output modes have larger coefficients. This manifests in the error discrepancy in the choice of $\ast$ as well: Empirically, we find that at fixed $N_{\mathrm{eff}}$ the hyperbolic–cross index set yields smaller errors than the $\ell^1$-ball, and that this advantage is more pronounced at lower viscosities (smaller $\nu$). A simple way to understand this is that for larger viscosities the Burgers' dynamics are strongly diffusive: high–frequency modes are damped, and the solution at time $T$ is largely controlled by a few low–frequency input modes and relatively simple interactions between them. Hyperbolic–cross sets are designed to emphasize such “few–active–coordinate’’ interactions, whereas an $\ell^1$-ball, being more isotropic, allocates many basis functions to more complicated combinations of modes. Here, the HC space places its limited degrees of freedom where the solution map actually varies most, which is consistent with the observed performance gap.

\begin{figure}[htb]
    \centering
\includegraphics[width = \textwidth]{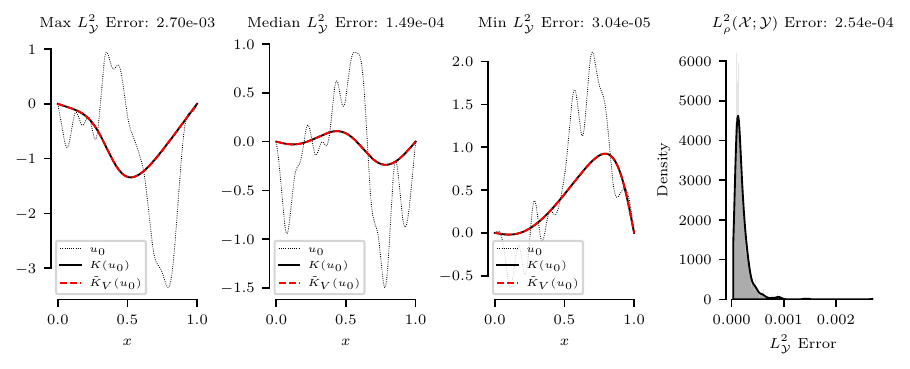}
    \caption{A comparison of the true solution $u= K(u_0)$ and the approximate solution $\tilde{K}^*_{V(60; \mathrm{HC})}(u_0)$ for representative initial conditions $u_0\sim \rho,$ along with a full test error distribution generated by the 1000 test samples. Here $\nu = 0.1.$}
    \label{fig:Test_Err_Dist_Burgers}
\end{figure}

Figure \ref{fig:Test_Err_Dist_Burgers} shows the test error distribution of $\tilde{K}_{V(60; \mathrm{HC})}$ on $1000$ samples $u_0^i\sim \rho$. It also includes a comparison of the action of $\tilde{K}_{V(60; \mathrm{HC})}$ and $K$ on representative initial conditions. We emphasize that the empirical Bochner norm error in this case is on the order of $10^{-4},$ and, as shown in \cref{fig:Test_Err_Burgers}, the relative Bochner norm error is on the order of $10^{-3}.$

Finally, though we took $d_{\mathrm{out}} = 150$ throughout, figure \ref{fig:vb_trunc} shows how such a choice might be made. Figure \ref{fig:out_trunc_err_VB} isolates the truncation error $\norm{K - \Pi_{Y_h}K}_{L^2_\rho}$ as a function of $d_{\mathrm{out}}$ and $\nu$, and shows that once $d_{\mathrm{out}}$ is large enough that the `energy fraction lost' falls below a prescribed tolerance (e.g. $<1\%$), further enlarging $\mathcal{Y}_h$ yields diminishing returns. Figure \ref{fig:vb_truncation_err} displays the full relative Bochner error, so by \cref{thm: Truncated Approx Error in Prob} its behavior reflects the balance between three contributions: best-approximation error in the operator space $V$, the output truncation term quantified in figure \ref{fig:vb_trunc}, and a noise term that includes the numerical PDE discretization used to generate the training data. Because we fix the input truncation $d_{\mathrm{in}}=20$, we are effectively learning the restriction of $K$ to the subspace $\mathcal X_h = \mathrm{span}\{\xi_1,\dots,\xi_{d_{\mathrm{in}}}\}$; the corresponding input truncation error is therefore a background term that cannot be reduced by increasing $d_{\mathrm{out}}$. Thus the saturation of the curves in figure \ref{fig:vb_truncation_err}  should be interpreted as the point at which Y-truncation no longer dominates, and the residual error is governed primarily by the choice of polynomial operator space and the fidelity of the underlying Burgers solver, rather than by the WLS procedure itself.

\begin{figure}[htbp]
  \centering
    \begin{subfigure}[t]{0.48\textwidth}
      \includegraphics[width = \linewidth]{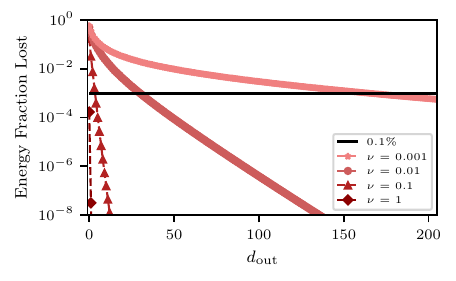}
      \caption{The energy fraction $1 -\frac{ \sum_{j \in [d_{\mathrm{out}}]}\mathbb{E}\left[\abs{\xi_j^*K(f)}^2\right]}{\mathbb{E}\left[\norm{K(f)}_{\mathcal{Y}}^2\right]}$ lost by truncating $\mathcal{Y}_h$. The expectations were approximated via an empirical average over 100 pushforwards of $f \sim \rho$.  }
    \label{fig:out_trunc_err_VB}
    \end{subfigure}
    \hfill
    \begin{subfigure}[t]{0.48\textwidth}
      \includegraphics[width = \linewidth]{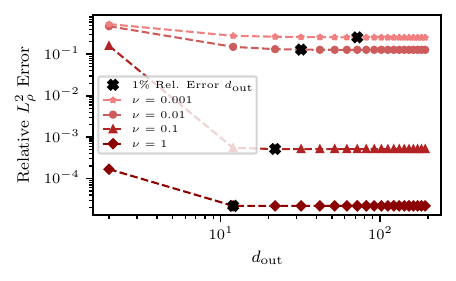}
      \caption{The relative Bochner norm error as a function of $d_{\mathrm{out}} = \dim \mathcal{Y}_h$ and the viscosity parameter $\nu$. Suitable truncation dimension (achieving less that $1\%$ relative truncation error) is shown to increase with decreasing $\nu$. }
      \label{fig:vb_truncation_err}
    \end{subfigure}
    \caption{The effect of $d_{\mathrm{out}}$ in $\mathcal{Y}$ (left) and in $L^2_\rho(\mathcal{X}_h; \mathcal{Y}_h)$ (right). In both bases, we fix $d_{\mathrm{in}} = 20.$}
    \label{fig:vb_trunc}
\end{figure}

\subsection{Navier-Stokes Equation}\label{ssec: Navier Stokes}

In this section, we aim to learn the \emph{nonlinear} solution operator 
\[
K: f \mapsto w(T, \cdot),
\]
for the two-dimensional incompressible Navier-Stokes equations on the torus $\mathbb{T}^2 = [0, 2\pi)^2$ in vorticity–streamfunction form:
\[
\partial_t w + u \cdot \nabla w = \frac{1}{\mathrm{Re}} \Delta w + f, \qquad
u = \nabla^\perp \phi, \qquad
-\Delta \phi = w, \qquad
\int_{\mathbb{T}^2} \phi \dif x = 0,
\]
with initial condition $w(0, \cdot) = 0$, Reynolds number $\mathrm{Re} = 500$ and terminal time $T = 5$. This example is designed to demonstrate that the method proposed in this paper can be applied using numerical basis and data sets constructed independently of the optimal sampling procedure. 

With this goal, we use an empirical dataset $X = \{(f^i,w^i)\}_{i=1}^S$ of size $S = 12000$ consisting of forcings $f^i$ and corresponding terminal-time vorticities $w^i=K(f^i)$, computed on a uniform grid on $\mathbb{T}^2$. A held-out set of $2000$ data pairs is used for testing. The data was generated for use in \cite{kossaifi2023multi}. \footnote{Reference \cite{kossaifi2023multi} provides details on the numerical solver and the dataset is available for download at \url{https://zenodo.org/records/12825163}.}

Let $\upsilon$ denote the empirical distribution of the inputs $\{f^i\}$: i.e., $\upsilon = S^{-1}\sum_{i\in [S]}\delta_{f^i}$. Let $\mathcal X_h\subset L^2(\mathbb{T}^2)$ be the span of the PCA basis $\{\xi_j\}_{j\in [d_{\mathrm{in}}]}$ capturing $95.5\%$ of the forcing energy. Here, $d_{\mathrm{in}} = 142$. We write $\xi_j^* f$ for the PCA coefficients. For the output space, we fix $\alpha\in \R$ and set $\mathcal Y_h(\alpha)\subset H^\alpha(\mathbb{T}^2)$ to be the span of Fourier modes orthonormal in $H^\alpha$:
$$
\psi_{\mathbf n}(\mathbf x) \;=\; \frac{e^{\mathrm{i} \mathbf n\cdot \mathbf x}}{2\pi\,(1+\norm{\mathbf n}_2^2)^{\alpha/2}},
\qquad \mathbf n\in\mathbb{Z}^2,
$$
and truncate to the smallest subspace capturing $99.9\%$  $H^\alpha$-energy of $\{w^i\}$.

\begin{figure}
  \centering
    \begin{subfigure}[t]{0.48\textwidth}
    \includegraphics[width = \linewidth]{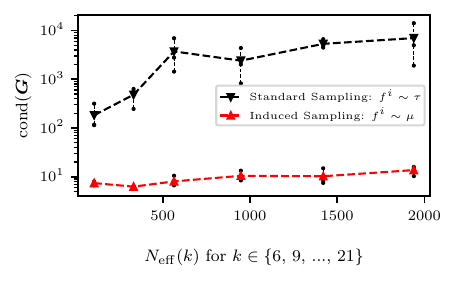}
      \caption{The condition number of the Gramian for subspaces of increasing dimension.}
      \label{fig:Cond_Num_NS}
    \end{subfigure}
    \hfill
    \begin{subfigure}[t]{0.48\textwidth}
    \includegraphics[width = \linewidth]{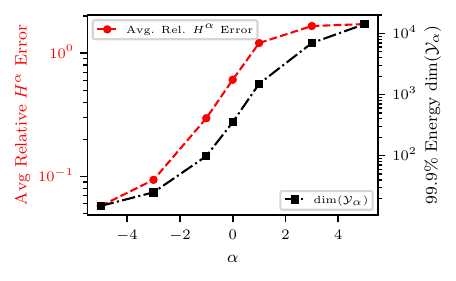}
      \caption{On the left scale (red), we show the average relative $H^{\alpha}$ error on various subspaces $V(28;\alpha)$, computed via 2000 test samples.  On the right scale (black), we show the $\dim(\mathcal{Y}_{\alpha})$ required to obtain 99.9\% output $H^{\alpha}$ Energy.
      }
      \label{fig:err_NS}
    \end{subfigure}
    \caption{Stability and accuracy of the approximation of the Navier-Stokes solution operator.}
  \end{figure}
  
Following \eqref{def: multivar PolyOP}, we parameterize candidate operators by 
\begin{equation}
P_{\bs \lambda}(f) = \psi_{\lambda_0}\prod_{j=1}^{d_{\mathrm{in}}} p^{j}_{\lambda_j}\big(\xi_j^* f\big),
\end{equation}
where, for a product reference measure $\beta=\bigotimes_{j=1}^{d_{\mathrm{in}}}\beta_j$ on the coefficient vector $(\xi_j^* f)_{j=1}^{d_{\mathrm{in}}}$, the $p^{j}_\ell$ are univariate $\beta_j$-orthonormal polynomials. We choose $\beta_j$ to be a Jacobi distribution whose first and second moments match the $j^{th}$ marginal of $\upsilon$.
To encode anisotropy, define weights $$\gamma_j = \sigma_j/ \sigma _{\mathrm{max}},$$
with $\{\sigma_j\}$ the (decreasing) singular values associated with $\{\xi_j\}$.  We restrict the univariate polynomial degrees in $P_{\bs \lambda}$ to $[10]$ and form approximation spaces by defining hyperbolic cross sets generating these degrees. Explicitly, we take
$$V(k;\alpha) = \mathrm{span}_{\bs \lambda}\{P_{\bs \lambda} \text{ : } \bs \lambda \in [\dim(\mathcal Y_h(\alpha))] \times \left(\Lambda_{\mathrm{HC}, \bs \gamma}(k) \cap\Lambda_{\infty}(10)\right)\} = \mathcal{Y}_h(\alpha)\otimes G_{\Lambda}(k)$$

We emphasize that the basis operators are not $L^2_\upsilon$ orthonormal and are merely used to define the space. We form an $L^2_\upsilon(\mathcal X;\mathcal \R)$ orthonormal basis of $G_{\Lambda}$ and the associated optimal sampling measure $\mu(V(k;\alpha), \upsilon)$ using the method described in \cref{app:sampling-discrete-measures}.  Then, by Theorem \ref{thm: samp complexity block learning}, the sample complexity for the discrete least squares problem is defined in terms of the effective dimension $N_{\mathrm{eff}}(k) = \dim(G_{\Lambda}(k))$. In each of the numerical experiments, like the previous section we take, $M(k) = \lceil N_{\mathrm{eff}} \log N_{\mathrm{eff}}\rceil$
samples $f^i \sim X$, drawn either from the discrete measure $\upsilon$ or an optimal sampling measure $\mu(V(k;\alpha), \upsilon).$

Figure \ref{fig:Cond_Num_NS} shows the condition number of the Gram matrix, averaged over three trials, for $k \in \{6,9,12,\dots, 21\}.$ Optimal sampling keeps $\mathrm{cond}(\bs G)$ essentially flat as $N_{\mathrm{eff}}$ grows, whereas Monte Carlo sampling deteriorates. We note that the mild upward drift under $\mu$ reflects our deliberate undersampling by the logarithmic factor. 

In figure \ref{fig:err_NS}, the red curve (left axis) shows the average relative $H^{\alpha}$ error of $\tilde{K}_{V(28;\alpha)}$ on the 2000 test samples $f^i$; the black curve (right axis) reports $\dim \mathcal Y_h(\alpha)$, the smallest Fourier subspace that captures $99.9\%$ of the empirical output $H^\alpha$-energy for that $\alpha$. For each $\alpha$, we therefore choose $\mathcal Y_h(\alpha)$ so that the truncation error $\norm{K - \Pi_{\mathcal Y_h(\alpha)}K}_{L^2_\upsilon(\mathcal X;H^\alpha)}$ is negligible at the scale of the plot. The observed dependence of the average relative $H^{\alpha}$ error on $\alpha$ is then dominated by the approximation properties of the fixed polynomial space $G_{\Lambda}(28)$ and not by output truncation.

Measuring error in $H^\alpha$ (rather than the $L^2$-type norms considered earlier) changes the geometry of the problem via the Fourier weights $(1+\norm{\mathbf n}_2^2)^{\alpha}$. Indeed,
$$
 \norm{g}_{H^\alpha}^2 = \sum_{\mathbf n\in\mathbb{Z}^2} (1+\norm{\mathbf n}_2^2)^{\alpha}\abs{\hat g_{\mathbf n}}^2 = \norm{(I-\Delta)^{\alpha/2} g}_{L^2}^2.
$$
Thus minimizing $\norm{Kf - \tilde K f}_{H^\alpha}$ is equivalent to minimizing the $L^2$-error after applying the smoothing/amplifying operator $(I-\Delta)^{\alpha/2}$. For $\alpha<0$, $(I-\Delta)^{\alpha/2}$ acts as a smoothing operator, so high frequencies are strongly down-weighted, and the metric is dominated by large-scale structures. In this case, the corresponding output spaces $\mathcal Y_h(\alpha)$ remain relatively low dimensional, and the relative errors are modest. As $\alpha$ increases, $(I-\Delta)^{\alpha/2}$ upweights high frequencies, $\dim Y_h(\alpha)$ grows quickly, and the $H^\alpha$-metric becomes increasingly sensitive to small-scale discrepancies. In this regime, the fixed polynomial operator space $G_{\Lambda}(28)$ limits how well fine-scale features can be resolved, which is reflected in the rise and eventual saturation of the error curves with $\alpha$. 

The qualitative effect of this norm choice is illustrated in \cref{fig:NS_action}, where we compare the action of $\tilde{K}_{V(28;-2)}$ to the true operator on the median-error test input: the $H^{-2}$-evaluation emphasizes coherence at large scales and largely ignores small-scale mismatch.

On this fixed, externally generated dataset, induced discrete sampling preserves the conditioning of the Gram matrix, so instability of the least-squares problem is not the dominant source of error. Instead, the evaluation norm $H^\alpha$ acts as a tunable geometric regularizer at the level of the output: negative $\alpha$ discount small-scale errors and yield a relatively easy, low-dimensional approximation problem, whereas positive $\alpha$ emphasize fine scales, expose the finite-resolution limitations of the chosen operator space $V(28;\alpha)$, and make the learning task substantially harder. In this sense, the choice of norm operationalizes  regularization in high-dimensional outputs by specifying which spatial scales of the Navier–Stokes dynamics the surrogate is required to reproduce accurately.

\begin{figure}[htb]
    \centering
    \includegraphics[width=0.7\linewidth]{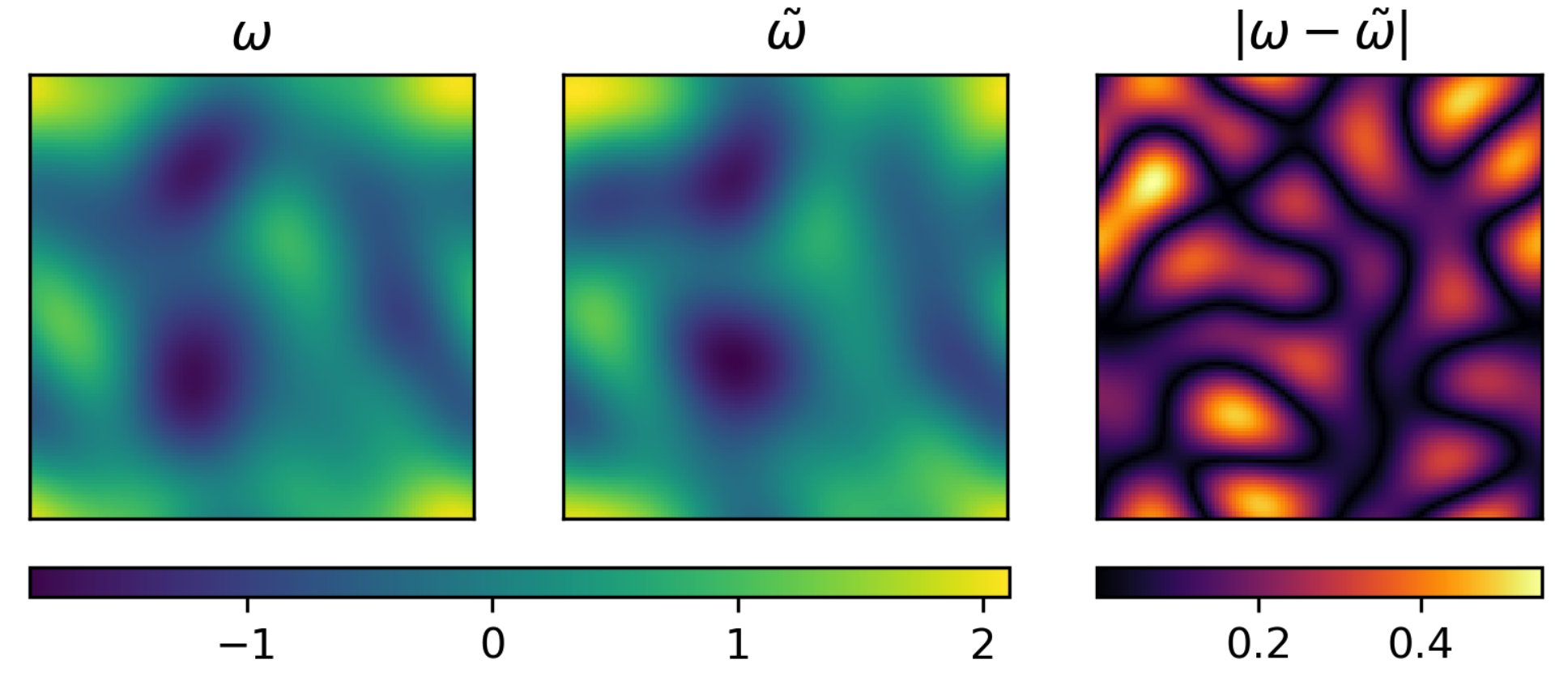}
    \caption{A comparison of $\omega = K(f)$ and $\tilde{\omega} = \tilde{K}_{V(28;-2)}(f)$ on the median error input test sample, $f$. The relative $H^{-2}$ error in this case is $1.05\times 10^{-1}$. }
    \label{fig:NS_action}
\end{figure}

\section{Conclusion}
\label{sec:conclusion}

This work develops a principled framework for learning deterministic operators in the Bochner space $L^2_\rho(\mathcal X;\mathcal Y)$ using optimal weighted least squares. On the approximation side, we construct two explicit families of operator bases---rank-one linear operators and rank-one polynomial operators---and prove density in appropriate operator classes under mild assumptions. These constructions connect naturally to kernel learning (Hilbert--Schmidt kernels in the linear case, kernelized integral representations in the polynomial case) while keeping approximation and error control at the level of operator norms. On the sampling side, we derive implementable optimal measures and weights, including discrete, data-driven variants, that yield uniformly well-conditioned Gram matrices and near-optimal sample complexity. 

Beyond providing existence results, the analysis gives nonasymptotic error bounds and sample-complexity guarantees in operator norms. Once an approximation space $V \subset L^2_\rho(\mathcal X;\mathcal Y)$ has been chosen, the theory can be read as a set of design rules: for a prescribed accuracy in $\norm{K - \tilde K_V}_{L^2_\rho(\mathcal X;\mathcal Y)}$, it specifies a sufficient number of training samples---drawn from an operator-level Christoffel distribution---to achieve this accuracy with high probability. This stands in contrast to most current neural operator approaches, where sample complexity and conditioning are difficult to characterize and architectural choices are typically made heuristically.

The numerical experiments both illustrate the theory and highlight practical design choices that ultimately control accuracy. In the Poisson example, operator-level Christoffel sampling leads to uniformly conditioned weighted least-squares systems, and the estimator recovers the diagonal structure of the spectral solution operator as well as the Green’s kernel, validating the linear theory in a setting with a known target kernel. In the viscous Burgers' experiments, decreasing viscosity shifts output energy toward high frequencies; at fixed sample budget, we observe that accuracy is governed primarily by discretization design---namely, the choice of anisotropic polynomial index sets on the input side and the truncation dimension $d_{\mathrm{out}}$ on the output side---while optimal sampling maintains stability even in deliberately undersampled regimes. In the Navier--Stokes case, based on a fixed externally generated dataset, the induced discrete sampling measure again preserves Gram conditioning, and the choice of evaluation norm $H^\alpha$ acts as a tunable geometric regularizer. Across these examples, the experiments confirm that explicit approximation spaces coupled with optimal sampling deliver stable, sample-efficient operator learning. They also show that, once sampling is chosen according to the theory, the dominant limitations on accuracy come from the approximation space itself---namely, the truncation of the input and output bases and the choice of polynomial index set. 

Several directions emerge from this work. First, the experiments indicate that fixed choices of $d_{\mathrm{in}}$, $d_{\mathrm{out}}$, and $\Lambda$ can dominate the achievable error, even when the sampling measure is optimal. While the truncation levels $d_{\mathrm{in}}$ and $d_{\mathrm{out}}$ can be informed by the energy-based procedures used here, a natural next step is to develop adaptive enrichment strategies for the operator space, for instance via greedy subspace selection that augments $\Lambda$ in directions suggested by residuals of the learned operator. Second, the optimal weighted least-squares estimators developed here can be combined with operator-valued Gaussian process priors to provide uncertainty quantification for both forward and inverse problems, equipping learned operators with principled error bars in $L^2_\rho(\mathcal X;\mathcal Y)$. Third, for applications where structure---such as conservation laws, symmetries, or dissipation---is important, designing operator bases that preserve this structure by construction remains an open and appealing challenge. Finally, while our focus has been on theory and single-method studies, a natural continuation is to benchmark optimal weighted least-squares operator learning against neural operator architectures on standard PDE datasets, comparing not only final accuracy but also sample requirements, training cost, and robustness across norms and discretizations; in such comparisons, the quantitative sample-complexity guarantees and conditioning results developed here provide a clear, theory-based baseline.

\subsection*{Acknowledgments}

This work was supported by Sandia National Laboratories Laboratory Directed Research Development (LDRD) program. Sandia National Laboratories is a multimission laboratory managed and operated by National Technology \& Engineering Solutions of Sandia, LLC, a wholly owned subsidiary of Honeywell International Inc., for the U.S. Department of Energy’s National Nuclear Security Administration under contract DE-NA0003525. SAND2025-14999O. This paper describes objective technical results and analysis. Any subjective views or opinions that might be expressed in the paper do not necessarily represent the views of the U.S. Department of Energy or the United States Government. This article has been authored by an employee of National Technology \& Engineering Solutions of Sandia, LLC under Contract No. DE-NA0003525 with the U.S. Department of Energy (DOE). The employee owns all right, title and interest in and to the article and is solely responsible for its contents. The United States Government retains and the publisher, by accepting the article for publication, acknowledges that the United States Government retains a non-exclusive, paid-up, irrevocable, world-wide license to publish or reproduce the published form of this article or allow others to do so, for United States Government purposes. The DOE will provide public access to these results of federally sponsored research in accordance with the DOE Public Access Plan https://www.energy.gov/downloads/doe-public-access-plan.

In addition, John Turnage and Akil Narayan were partially supported by the U.S.\ Air Force Office of Scientific Research (AFOSR) under Grant Number FA9550-23-1-0749.

%~~~~~~~~~ Bibliography
%\newpage 

\fancyhf{} 
\renewcommand{\headrulewidth}{0pt}
\fancyfoot[C]{\hrulefill\quad\raisebox{-3pt}{Page \thepage \hspace{1pt} of \pageref{LastPage}}\quad\hrulefill}

\phantomsection 
\addcontentsline{toc}{section}{References} 

\printbibliography 

\newpage

\appendix
\section{Appendices}

\subsection{Technical Results}\label{subsec: Proofs in Appendix}

\subsubsection{Preliminaries for Technical Discussions}
Associated to the semi-norm $\|\cdot\|_{L^2_{\rho,M}}$ on candidate operators from $V$ defined in \eqref{eq:discrete-ls}, we will make use of the semi-inner product,
\begin{align}\label{eq:discrete_IP}
\ip{A,B}_{L^2_{\rho,M}} \coloneqq \frac{1}{M}\sum_{i\in [M]}w(f^i)\ip{Af^i, Bf^i}_{\mathcal{Y}}.
\end{align}
Just as with standard least squares problems, solutions to the minimization problem \eqref{eq:discrete-ls} involving the norm $\|\cdot\|_{L^2_{\rho,M}}$ are computed by $\|\cdot\|_{L^2_{\rho,M}}$-projecting $K$ onto $V$, with the caveat that if $\|\cdot\|_{L^2_{\rho,M}}$ is only a semi-norm, then such projections identify affine spaces of minimizers instead of unique minimizers. In detail, we compute the $\|\cdot\|_{L^2_{\rho,M}}$-projection of $K$ onto $V$ by expanding a feasible solution $A = \sum_{n \in [N]} c_n \Phi_n$ and computing the following inner products:
\begin{align*}
  \ip{A, \Phi_k}_{L^2_{\rho,M}} &= \ip{\sum_{n \in [N]} c_n \Phi_n, \Phi_k}_{L^2_{\rho,M}} = \sum_{n \in [N]} c_n \ip{\Phi_n, \Phi_k}_{L^2_{\rho,M}},  \\
  \ip{K + \eta, \Phi_k}_{L^2_{\rho,M}} &= \frac{1}{M} \sum_{i \in [M]} w(f^i) \ip{K(f^i) + \eta(f^i), \Phi_k(f^i)}_{\mathcal{Y}} = \frac{1}{M} \sum_{i \in [M]} w(f^i) \ip{g^i, \Phi_k(f^i)}_{\mathcal{Y}}.
\end{align*}
A necessary and sufficient condition for $A$ to be the $\|\cdot\|_{L^2_{\rho,M}}$- noisy projection of $K$ onto $V$ is that the two inner products above are equal for all $k \in [N]$. Hence, $\bs{c} = (c_n)_{n \in [N]}$ should satisfy the normal equations,
\begin{align}\label{eq:normal-equations}
  \bs{G} \bs{c} = \bs{g}
\end{align}
where,
\begin{align*}
  \mathbf{G}_{{j},{k}}\coloneqq \ip{\Phi_{ {j}}, \Phi_{{k}}}_{L^2_{\rho,M}} &= \frac{1}{M} \sum_{i \in [M]} w(f^i) \ip{\Phi_j(f^i) \Phi_k(f^i)}_{\mathcal{Y}}, &  \mathbf{g}_{{j}} &= \frac{1}{M}\sum_{i\in [M]}w(f^i)\ip{g^i, \Phi_{j}(f^i)}_{\mathcal{Y}}.
\end{align*}
When $\|\cdot\|_{L^2_{\rho,M}}$ is a norm on $V$, $\bs{G}$ is invertible, and there is a unique solution $\tilde{K}^*_{V}$ to \eqref{eq:discrete-ls} whose $V$-expansion coefficients are given by $\bs{c} = \bs{G}^{-1} \bs{g}$. Otherwise, there is an affine space of minimizers whose dimension is $\dim \mathrm{ker}(\bs{G})$, and, for example, the unique $L^2_\rho$-norm minimizing solution has $V$-expansion coefficients $\bs{c} = \bs{G}^{\dagger} \bs{g}$, where $\bs{G}^{\dagger}$ is the Moore-Penrose pseudoinverse.

The strong law of large numbers along with the assumptions \eqref{eq:asymptotic-assumptions} guarantee that with probability 1,
\begin{align*}
  (\bs{G})_{j,k} = \ip{\Phi_k, \Phi_j}_{L^2_{\rho,M}} &\xrightarrow[M\uparrow \infty]{\eqref{eq:asymptotic-assumptions}} \ip{\Phi_k, \Phi_j}_{L^2_\rho} = \delta_{j,k}, \\
  (\bs{g})_j = \ip{K + \eta,\Phi_j}_{L^2_{\rho,M}} &\xrightarrow[M\uparrow \infty]{\eqref{eq:asymptotic-assumptions}} \ip{K + \eta, \Phi_j}_{L^2_\rho}.
\end{align*}
In this case, and in the absence of noise, the normal equations \eqref{eq:normal-equations} assert that $c_j = \ip{K,\Phi_j}_{L^2_\rho}$, i.e., that $\tilde{K}_{V}^\ast = K_{V}^\ast$. We remark that \eqref{eq:asymptotic-assumptions} can be relaxed so long as the resulting Gramian $\bs{G}$ concentrates to a bounded, invertible matrix for large $M$. For simplicity we will continue to assume \eqref{eq:asymptotic-assumptions}.

Of course, practical interest restricts us to the finite $M$ case, and so we must quantitatively establish when and how $\bs{G}$ and $\bs{g}$ concentrate to their infinite-$M$ counterparts. 

\subsubsection{Proof of Lemma \ref{lem:comp_norms}}\label{Pf: comp_norms}
\begin{proof}
The first equivalence is an immediate consequence of the definitions of the condition number and the spectral norm. For the second equivalence, we note that the
expression on the right hand side is itself equivalent to
$$\abs{\norm{A}_{L^2_{\rho,M}}^2 - \norm{A}_{L^2_{\rho}}^2} \le \delta \norm{A}_{L^2_{\rho}}^2 \hspace{0.25cm}(\forall A\in V).$$
Letting $\bs{a}$ be the vector representing $A\in V$ in the orthonormal basis $\{\Phi_n\}_{n\in [N]},$ we have
   $$\norm{A}_{L^2_{\rho,M}}^2 = \bs{a}^T\bs{G}\bs{a} \hspace{0.25cm}\text{and}\hspace{0.25cm} \norm{A}_{L^2_\rho}^2 = \bs{a}^T\bs{I}\bs{a}.$$
   Thus, if $\norm{\bs G - \bs I }_2 \le \delta,$ then
   $$\abs{\norm{A}^2_{L^2_{\rho,M}} - \norm{A}^2_{L^2_\rho} } = \abs{\bs{a}^T(\bs{G}-\bs{I})\bs{a}} \le \norm{\bs{G}-\bs{I}}_2\norm{\bs{a}}_2^2 \le \delta\norm{A}_{L^2_\rho}^2.$$
   Conversely, if $\norm{\bs G - \bs I}_{2} > \delta,$ then there exists an $\bs a \in \R^N$ such that
   $$\abs{\bs a^T(\bs G- \bs I)\bs a} >\delta \norm{\bs a}_2^2,$$
   namely, any $\bs a$ lying in the eigenspace of $(\bs G - \bs I)$ corresponding to the maximum norm eigenvalue of $(\bs G - \bs I).$
\end{proof}

\subsubsection{Proof of Theorem \ref{thm:Samp Complexity}}\label{Pf: Sample Complexity}
\begin{proof}
    This proof is a relatively straightforward application of spectral tail bounds for sums of random matrices. We will make use of the following matrix Chernoff bound, which may be found in \cite{tropp2012user}. 
    
    If $\{\mathbf{Y}_i\}_{i\in [M]}$ are independent $N\times N$ symmetric positive semidefinite matrices satisfying
    $\lambda_{\max}(\mathbf Y_i) = \norm{\mathbf Y_i}_2 \le R $
    almost surely, then with
    $$\zeta_{\min}\coloneqq \lambda_{\min}\left(\sum_{i\in [M]}\E[\mathbf Y_i]\right)\hspace*{0.5cm}\text{and}\hspace*{0.5cm} \zeta_{\max }\coloneqq \lambda_{\max}\left(\sum_{i\in [M]}\E[\mathbf Y_i]\right),$$
    one has the following bounds: 
    \begin{align}\label{eq:Chernoff Bounds}\begin{cases}
        \Pr \left\{\lambda_{\min}\left(\sum_{i\in [M]}\mathbf Y_i\right)\le (1-\delta)\zeta_{\min}\right\} \le N \left(\frac{e^{-\delta}}{(1-\delta)^{1-\delta}}\right)^{\zeta_{\min}/R},&  \delta \in [0,1)\\
        \Pr \left\{\lambda_{\max}\left(\sum_{i\in [M]}\mathbf Y_i\right)\ge (1+\delta)\zeta_{\max}\right\} \le N \left(\frac{e^{\delta}}{(1+\delta)^{1+\delta}}\right)^{\zeta_{\max}/R},&  \delta \in \R_+
    \end{cases}.\end{align}
    
    Now, observe that the Gramian for the weighted least-squares problem can be written  
    $$\bs G = \sum_{i\in [M]}\mathbf X_i,$$ where $\mathbf X_i$ are i.i.d copies of the random symmetric, positive semidefinite matrix
    \begin{align}
    \mathbf X = \mathbf X(f) \coloneqq \frac{1}{M}\left(w(f)\ip{\Phi_{j}(f),\Phi_{k}(f)}_{\mathcal{Y}}\right)_{( {j,k})\in [N]^2}.
    \end{align} 
    Since $\mu, w$ are assumed to be given as in \eqref{eq:asymptotic-assumptions}, we have
    $\sum_{i\in [M]}\E_{f\sim \mu}[\mathbf{X}_i] = M\E[\mathbf X] = \bs I$, so that $ \zeta_{\min} = \zeta_{\max}= 1.$ Then since $\bs G$ is symmetric,
    $$\Pr\left\{\norm{\bs G - \bs I}_2  \ge \delta\right\} = \Pr \left\{ \lambda_{\max}(\bs G) \ge 1 + \delta \right\} + \Pr\left\{\lambda_{\min}(\bs{G}) \le 1 - \delta \right\}.$$ Summing the Chernoff bounds \eqref{eq:Chernoff Bounds}, we find
    \begin{align*}
        \Pr\left\{\norm{\bs G - \bs I}_2   \ge \delta\right\} & \le N\left[ \left(\frac{e^{\delta}}{(1+\delta)^{1+\delta}}\right)^{1/R} +  \left(\frac{e^{-\delta}}{(1-\delta)^{1-\delta}}\right)^{1/R}\right].
    \end{align*}
    Since 
    $\frac{e^{\delta}}{(1 + \delta)^{1+\delta}} \le \frac{e^{-\delta}}{(1-\delta)^{1-\delta}}$
    for all $\delta \in (0,1)$,
    $$\Pr\left\{\norm{\bs G - \bs I}_2   \ge \delta\right\}  \le 2N\left(\frac{e^{-\delta}}{(1-\delta)^{1-\delta}}\right)^{1/R} = 2N\exp\left(-\frac{1}{c_\delta R}\right),$$
    where $c_{\delta}$ is as in \eqref{eq:SampleComplexity}. It remains to be shown that $\norm{\mathbf X}_2$ may be bounded by some $R \in \R_+$. 
    By Cauchy-Schwarz, 
    \begin{align*}
        \norm{\mathbf X}_2&\le \frac{w(f)}{M}\norm{\left(\norm{\Phi_{ {j}}(f)}_{\mathcal{Y}} \norm{\Phi_{ {k}}(f)}_{\mathcal{Y}} \right)_{( {j,k})\in [N]^2}}_2 = \frac{w(f)}{M} \norm{\mathbf{\Phi}\mathbf{\Phi}^T}_2,
    \end{align*}
    where $\mathbf{\Phi}_{ {j}} = \norm{\Phi_{  j}(f)}_{\mathcal{Y}}.$ Thus, 
    $$\norm{\mathbf X}_2 \le \frac{w(f)}{M}\norm{\mathbf \Phi}_2^2 = \frac{w(f)}{M}\sum_{  j \in [N]}\norm{\Phi_{  j}(f)}_{\mathcal{Y}}^2 \le \frac{\kappa_w^{\infty}}{M} \coloneqq R,$$
    and we find that 
    $$\Pr\left\{\norm{\bs G - \bs I}_2   \ge \delta\right\}  \le 2N \exp\left(- \frac{M}{c_{\delta} \kappa^{\infty}_w}\right).$$
    Taking $M \ge c_{\delta}k^{\infty}_{w} \log\left(\frac{2N}{\epsilon}\right)$, we conclude that
    $\Pr\{\norm{\bs G-\bs I}_2 \ge \delta\} \le \epsilon,$
    which completes the proof. 
\end{proof}

\subsubsection{Proof of Theorem \ref{thm: samp complexity block learning}}\label{Pf: Block learning}
\begin{proof}
Let $\{\varphi_j\}_{j\in [S]}$ and $\{\psi_i\}_{i \in [\dim Y_h]}$ be orthonormal bases for $P\subset L^2_\rho(\mathcal X; \R)$ and $\mathcal Y_h$ respectively. Define the rank-one operators 
$$\Phi_{ij}(f) = \varphi_j(f)\psi_i.$$
The operators $\{\Phi_{ij}\}$ form an orthonormal family in $L^2_\rho(\mathcal X; \mathcal Y)$ since 
$$\ip{\Phi_{ij}, \Phi_{i', j'}}_{L^2_\rho(\mathcal X; \mathcal Y)} = \delta_{i, i'} \ip{\varphi_j, \varphi_{j'}}_{L^2_\rho(\mathcal X; \R)} = \delta_{i, i'}\delta_{j,j'}.$$
Similarly, the discrete inner-product defined on the samples $\{f^k\}_{k\in [M]}$ satisfies
 $$\ip{\Phi_{i, j}, \Phi_{i', j'}}_{L^2_{\rho,M}} = \frac{\delta_{i, i'}}{M}\sum_{k\in [M]}w(f^k)\varphi_{j}(f^k)\varphi_{j'}(f^k).$$
Thus, the empirical Gram matrix associated to the subspace $V = \mathrm{span}_{(i,j) \in [\dim \mathcal Y_h]\times [S]}\{\Phi_{ij}\}$ decomposes as the direct sum
$$\bs G = \bigoplus_{i\in [\dim \mathcal Y_h]} \bs G^{i},$$
where $\bs G^{i}$ is the empirical Gram matrix for the scalar-valued basis $\{\varphi_{j}\}_{j\in [S]}$ evaluated at the shared sample set $\{f^k\}_{k\in [M]}.$ The analysis of each block $\bs G^{i}$ is independent and identical, so we may apply Theorem \ref{thm:Samp Complexity} to each. We emphasize that since each $G^{i}$ is defined by the same sample set, a union bound over $i \in [\dim \mathcal Y_h]$ is unnecessary: Each block is guaranteed to concentrate as long as the sample size satisfies
$$M \ge c_{\delta}\kappa_{w}^{\infty}(P) \log \left(\frac{2S}{\epsilon}\right),$$
where $\kappa_w^{\infty}(P)$ is the weighted Nikolskii constant associated to the subspace $P$. Under the optimal sampling measure associated with $P$, $\kappa_w^{\infty}(P) = \dim P = S$, completing the proof. 
\end{proof}

\subsubsection{Proof of Theorem \ref{thm: L2 error in expectation}} \label{Pf: Proof of L2 Err in Expectation}
\begin{proof}
  Note first $T_\tau$ is norm non-expansive on $\mathcal{Y}$ by its definition \eqref{def Trunc LSP}:
  \begin{align*}
    \left\| T_\tau g \right\|_{\mathcal{Y}} \leq \left\| g \right\|_{\mathcal{Y}}.
  \end{align*}
  Then assuming $\tau$ is chosen as in \eqref{eq:tau-choice}, we have for almost all $f \in \mathcal{X}$,
  \begin{align*}
    \left\| K f - \tilde{K}^T_{V} f \right\|_{\mathcal{Y}} = \left\| T_\tau \left( K f - \tilde{K}^\ast_V f \right) \right\|_{\mathcal{Y}} \leq 
  \left\|  K f - \tilde{K}^\ast_V f \right\|_{\mathcal{Y}},
  \end{align*}
  and therefore,
  \begin{align}\label{def: trunc ineq}
    \norm{K-\tilde{K}^T_{V}}_{L^2_{\rho}(\mathcal X; \mathcal Y)} \le \norm{K -\tilde{K}^*_{V}}_{L^2_{\rho}(\mathcal X;\mathcal Y)}.
  \end{align}

Now, we denote by $\dif \mu^M \coloneqq \otimes^M\dif \mu$ the probability measure of the draw $\{f^i\}_{i\in [M]}\subset \mathcal X$ and partition the set $\Omega$ of all possible draws into two events:
\begin{align}\label{def: Prob Space of Draw}\Omega_+ \coloneqq \{\norm{\bs G - \bs I}_{2} \le \delta\} \hspace{0.5cm}\text{and}\hspace{0.5cm} \Omega_- \coloneqq \Omega \setminus \Omega_+.\end{align}
By \cref{thm:Samp Complexity}, the assumptions on $M$ guarantee that
$\Pr\{\Omega_-\} \le \epsilon.$ 
Thus, for the truncated least squares problem \eqref{def Trunc LSP}, we have
\begin{align*}
    \E\left[\norm{K-\tilde{K}^{T}_{V}}^2_{L^2_\rho}\right] \le \int_{\Omega_+} \norm{K-\tilde{K}^*_{V}}^2_{L^2_\rho} \dif \mu^M + 4\tau^2\epsilon,
\end{align*}
where we have used both $\norm{K-\tilde{K}^T_{V}} \le 2 \tau$ and the inequality \eqref{def: trunc ineq}. Similarly, for the conditioned least squares problem \eqref{def: Conditioned LSP}, we have
$$\E\left[\norm{K - \tilde{K}^C_{V}}_{L^2_\rho}^2\right] \le \int_{\Omega_+}\norm{K-\tilde{K}^*_{V}}^2_{L^2_{\rho}}\dif \mu^M  + \epsilon\norm{K}_{L^2_\rho}^2,$$
where we have used that $\tilde{K}^C_{V}\vert_{\Omega_-} = 0$. 

For the remainder of the proof, we place ourselves in the event $\Omega_+$. In this event, the $L^2_{\rho,M}$-orthogonal projection onto $V$, which we denote $\Pi^M_{V}$, is unique and fixes $V$. Thus, with $R \coloneqq (K - \tilde{K}^*_{V}) \perp V$, we have
$$K - \tilde{K}_{V} = K - K^* + \Pi^M_{V}K^* - \Pi_{V}^M(K+\eta)  = R - \Pi_{V}^MR - \Pi^M_{V}\eta,$$
where we have used $\tilde{K}^*_{V} = \Pi^M_{V}(K+\eta), K^*_{V} = \Pi_{V}K,$ and $\Pi^M_{V}\Pi_{V}K = \Pi_{V}K.$
Then,
\begin{equation}\label{ineqs_str exp proof}
\begin{aligned}
\norm{K-\tilde{K}_{V}}^2_{L^2_\rho} &= \norm{R}_{L^2_\rho}^2 + \norm{\Pi^M_{V}R + \Pi^M_{V}\eta}_{L^2_\rho}^2\\& = \epsilon_{V}(K)^2 + \norm{\Pi^M_{V}R}^2_{L^2_\rho} + 2\ip{\Pi^M_{V}R, \Pi^M_{V}\eta}_{L^2_\rho} + \norm{\Pi^M_{V}\eta}^2_{L^2_\rho} \\
& \le \epsilon_{V}(K)^2 + \norm{\Pi^M_{V}R}^2_{L^2_\rho} + 2\norm{\Pi^M_{V}R}_{L^2_\rho}\norm{\Pi^M_{V}\eta}_{L^2_\rho} + \norm{\Pi^M_{V}\eta}^2_{L^2_\rho}\\
&\le \epsilon_{V}(K)^2 + 2\norm{\Pi^M_{V}R}^2_{L^2_\rho} + 2\norm{\Pi^M_{V}\eta}^2_{L^2_\rho}
\end{aligned}
\end{equation}

Now, for any $A\in L^2_\rho(\mathcal X;\mathcal Y),$ $\Pi^M_{V}A = \sum_{k\in [N]}a_k\Phi_k$, where the expansion coefficients $a_k$ are given by the normal equations $\bs G \bs a = \bs d$, with $d_j = \ip{A,\Phi_k}_{L^2_{\rho,M}}.$ Since we are in the event $\Omega_+,$ $\norm{\bs G^{-1}}_2 \le (1-\delta)^{-1},$ and it follows that 
$$\norm{\Pi^M_{V}A}^2_{L^2_\rho} \le \norm{\bs G^{-1}}_2^2\norm{\bs d}_2^2 \le \frac{1}{(1-\delta)^2} \sum_{k\in [N]}\abs{\ip{A,\Phi_k}_{L^2_{\rho,M}}}^2.$$
Thus, for any $A\in L^2_\rho,$
\begin{align*}
\E\left[\norm{\Pi^M_{V}A}_{L^2_\rho}^2\right] & \le \frac{1}{(1-\delta)^2} \sum_{k\in[N]} \E\left[\left(\frac{1}{M}\sum_{i\in[M]}w(f^i)\ip{A(f^i),\Phi_k(f^i)}_{\mathcal Y}\right)^2\right]\\
& = \frac{1}{M^2(1-\delta)^2} \sum_{k\in[N]}\left( \sum_{i\in[M]} \E\left[w(f^i)^2\ip{A(f^i),\Phi_k(f^i)}_{\mathcal Y}^2\right] \right. \\
&\hspace{0.5cm}\left.+ \sum_{j\neq i}\E\left[w(f^i)w(f^j)\ip{A(f^i),\Phi_k(f^i)}_{\mathcal Y}\ip{A(f^j),\Phi_k(f^j)}_{\mathcal Y}\right] \right)\\
& = \frac{1}{M^2(1-\delta)^2}\sum_{k\in [N]}\left(M \E\left[w(f)^2\ip{A(f),\Phi_k(f)}_{\mathcal Y}^2\right] + M(M-1)\E\left[w(f)\ip{A(f), \Phi_k(f)}_{\mathcal Y}\right]^2\right)\\
& = \frac{1}{(1-\delta)^2}\sum_{k\in [N]}\left(\frac{1}{M}\int_{\mathcal X}w(f)^2\ip{A(f),\Phi_k(f)}_{\mathcal Y}^2 \dif \mu(f) + \left(1-\frac{1}{M}\right) \ip{A,\Phi_k}_{L^2_\rho}^2\right)\\
& = \frac{1}{(1-\delta)^2}\sum_{k\in [N]}\left(\frac{1}{M}\int_{\mathcal X}w(f)\ip{A(f),\Phi_k(f)}_{\mathcal Y}^2\dif \rho(f) + \left(1-\frac{1}{M}\right) \ip{A,\Phi_k}_{L^2_\rho}^2\right)\\
& \le \frac{1}{(1-\delta)^2}
\left(\frac{1}{M}\int_{\mathcal X}\sum_{k\in [N]} w(f)\norm{A(f)}_{\mathcal Y}^2\norm{\Phi_k(f)}_{\mathcal Y}^2\dif \rho(f) + \left(1-\frac{1}{M}\right)\norm{A}_{L^2_\rho}^2\right)\\
& \le \frac{1}{(1-\delta)^2}\left(\frac{\kappa_w^{\infty}\norm{A}_{L^2_\rho}^2}{M} + \norm{A}^2_{L^2_\rho}\right)\\
& \le \frac{1}{(1-\delta)^2}\left(\frac{1}{c_{\delta}\log\left(2N/\epsilon \right)}  + 1\right) \norm{A}^2_{L^2_\rho}
\end{align*}
where we have used the independence of the draw, the assumption that $w\dif \mu = \dif \rho$, the definition of $\kappa^{\infty}_{w}$ in \eqref{def: kappa_inf}, and the sample complexity assumption $M\ge c_{\delta}\kappa^{\infty}_w\log(2N/\epsilon).$ Now, observe that for any $A\perp V,$ the second term in the fourth equality above is null, and the estimate can be refined accordingly. Thus, for $R = K-K^*_{V}\perp V,$ we find
$$\E\left[\norm{\Pi^M_{V}R}^2_{L^2_\rho}\right] \le \frac{\gamma(N)}{(1-\delta)^2}\epsilon_{V}(K)^2$$
and for the noise $\eta$, we find 
$$\E\left[\norm{\Pi^M_{V}\eta}^2_{L^2_\rho}\right]\le \frac{\gamma(N) + 1}{(1-\delta)^2}\norm{\eta}^2_{L^2_\rho}.$$
Combining these estimates for \eqref{ineqs_str exp proof} with those for the event $\Omega_-$, we have for the truncated least squares problem
$$\E\left[\norm{K - \tilde{K}^T_{V}}_{L^2_\rho}^2\right] \le  \left(1 + \frac{2\gamma(N)}{(1-\delta)^2}\right)\epsilon_{V}(K)^2 + \frac{2(\gamma(N) + 1)}{(1-\delta)^2}\norm{\eta}^2_{L^2_\rho}  + 4\tau^2\epsilon,$$
and for the conditioned least squares problem
$$\E\left[\norm{K - \tilde{K}^C_{V}}_{L^2_\rho}^2\right] \le  \left(1 + \frac{2\gamma(N)}{(1-\delta)^2}\right)\epsilon_{V}(K)^2 + \frac{2(\gamma(N) + 1)}{(1-\delta)^2}\norm{\eta}^2_{L^2_\rho}  + \epsilon\norm{K}^2_{L^2_\rho}.,$$
\end{proof}

\subsubsection{Proof of Theorem \ref{thm: Truncated Approx Error in Prob}}\label{Pf: Trunc Approx Error in Prob}

\begin{proof}
Let $\{f^i\}_{i\in[M]}\subset \mathcal X$ and $w: \mathcal X \to \R_+$. Suppose that there exists a $\delta \in (0,1)$ for which 
\begin{align}\label{eq: norm_comp_trunc_DLS}
    (1-\delta)\norm{B}^2_{L^2_\rho(\mathcal X;\mathcal Y)} \le \norm{B}_{L^2_{\rho,M}}^2 \hspace{0.5cm}(\forall B \in V_h).
\end{align}
It follows that the semi-inner product $\ip{\cdot, \cdot}_{L^2_{\rho,M}}$ defined in \eqref{eq:discrete_IP} is a true inner product on $V_h,$ so the solution $\tilde{K}_{V_h}$ to \eqref{eq:trunc_disc_LS} is unique. Equivalently, for all $B\in V_h$ the variational equations 
$$\ip{\tilde{K}_{V_h}, B}_{L^2_{\rho,M}} = \ip{K, B}_{L^2_{\rho,M}} + \ip{\eta, B}_{L^2_{\rho,M}} = \ip{\Pi_{\mathcal{Y}_h}K, B}_{L^2_{\rho,M}} + \ip{\eta, B}_{L^2_{\rho,M}}$$ are satisfied, where the second equality comes from the definition of the $L^2_{\rho,M}$-norm and the assumption $B\in V_h.$ Then, since $(\tilde{K}_{V_h} - B)\in V_h$, we have 
$$\ip{\tilde{K}_{V_h}, \tilde{K}_{V_h}-B}_{L^2_{\rho,M}}  = \ip{\Pi_{\mathcal{Y}_h}K, \tilde{K}_{V_h}-B}_{L^2_{\rho,M}} + \ip{\eta, \tilde{K}_{V_h}-B}_{L^2_{\rho,M}}.$$
Subtracting $\ip{B, \tilde{K}_{V_h} - B}_{L^2_{\rho,M}}$ from both sides yields
$$\norm{\tilde{K}_{V_h}-B}_{L^2_{\rho,M}}^2 = \ip{\Pi_{\mathcal{Y}_h}K - B, \tilde{K}_{V_h}-B}_{L^2_{\rho,M}} + \ip{\eta, \tilde{K}_{V_h}-B}_{L^2_{\rho,M}}.$$
Applying the Cauchy-Schwartz inequality to the right hand side, we find
$$\norm{\tilde{K}_{V_h}-B}_{L^2_{\rho,M}} \le \norm{\Pi_{\mathcal{Y}_h}K - B}_{L^2_{\rho,M}} + \norm{\eta}_{L^2_{\rho,M}}.$$ Finally, applying \eqref{eq: norm_comp_trunc_DLS} to the left side of the inequality, we deduce that
\begin{align}\label{eq:trunc_DLS_ineq}
\norm{\tilde{K}_{V_h}-B}_{L^2_\rho(\mathcal X;\mathcal Y)} \le \frac{1}{\sqrt{1-\delta}}\left(\norm{\Pi_{\mathcal{Y}_h}K - B}_{L^2_{\rho,M}} + \norm{\eta}_{L^2_{\rho,M}}\right).
\end{align}
Now, we take $A\in V$ arbitrary and write $\Pi_{\mathcal{Y}_h}A = B.$ Then 
\begin{align*}
    \norm{K - \tilde{K}_{V_h}}_{L^2_\rho(\mathcal X;\mathcal Y)}  &= \norm{\tilde{K}_{V_h} - B + B - \Pi_{\mathcal{Y}_h}K + \Pi_{\mathcal{Y}_h}K - K}_{L^2_\rho(\mathcal X;\mathcal Y)}\\
    &\le \norm{\tilde{K}_{V_h} - B}_{L^2_\rho} + \norm{\Pi_{\mathcal{Y}_h}(A-K)}_{L^2_\rho} + \norm{\Pi_{\mathcal{Y}_h}K - K}_{L^2_\rho}.
\end{align*}
Applying the bound \eqref{eq:trunc_DLS_ineq} to the first term yields
\begin{align*}
    \norm{K - \tilde{K}_{V_h}}_{L^2_\rho(\mathcal X;\mathcal Y)} 
    &\le \frac{1}{\sqrt{1-\delta}}\left(\norm{\Pi_{\mathcal{Y}_h}(K - A)}_{L^2_{\rho,M}} + \norm{\eta}_{L^2_{\rho,M}}\right) \\ &\hspace{0.5cm}+ \norm{\Pi_{\mathcal{Y}_h}(A-K)}_{L^2_\rho} + \norm{\Pi_{\mathcal{Y}_h}K - K}_{L^2_\rho}.
\end{align*}
Finally, we note that since $\Pi_{\mathcal{Y}_h}$ is non-expansive on $\mathcal Y$, $\norm{\Pi_{\mathcal{Y}_h} A} \le \norm{A}$ for both the discrete $L^2_{\rho,M}$ and the continuous $L^2_\rho$ norms on  $L^2_\rho(\mathcal X; \mathcal Y).$ Thus, 
\begin{align*}
    \norm{K-\tilde{K}_{V_h}}_{L^2_\rho(\mathcal X;\mathcal Y)} &\le \norm{K-A}_{L^2_\rho(\mathcal X;\mathcal Y)} +\frac{1}{\sqrt{1-\delta}}\norm{K-A}_{L^2_{\rho,M}}\\
    &\hspace{0.5cm} +\norm{K-\Pi_{\mathcal{Y}_h}K}_{L^2_\rho(\mathcal X; \mathcal Y)} + \frac{1}{\sqrt{1-\delta}}\norm{\eta}_{L^2_{\rho,M}}
\end{align*}
for arbitrary $A\in V$, and the result follows. 
\end{proof}

\subsubsection{Proof of Theorem \ref{thm: Bochner density Finite Rank in Linear Op}}\label{Pf: Bochner_density_linear}
\begin{proof}
Let $K \in \mathfrak{B}(\mathcal{X}; \mathcal{Y})$, and define the sequence of approximants $K_n := K \Pi_n$. For any $f \in \mathcal{X}$, we have
$$K_n(f) = K(\Pi_n f) \to K(f) \quad \text{in } \mathcal{Y} \quad \text{as } n \to \infty,$$
since $\Pi_n f \to f$ in $\mathcal{X}$ and $K$ is continuous. Thus, $K_n(f) \to K(f)$ pointwise for all $f \in \mathcal{X}$.
To estimate the Bochner norm of the error, observe that
$$
\|K - K_n\|_{L^2_\rho}^2 = \int_{\mathcal{X}} \|K(f) - K(\Pi_n f)\|_{\mathcal{Y}}^2 \dif \rho(f).
$$
Using the boundedness of $K$, we have
$$
\|K(f) - K(\Pi_n f)\|_{\mathcal{Y}} \le \|K\|_{\mathfrak{B}} \cdot \|f - \Pi_n f\|_{\mathcal{X}},
$$
and hence
$$
\|K - K_n\|_{L^2_\rho}^2 \le \|K\|_{\mathfrak{B}}^2 \int_{\mathcal{X}} \|f - \Pi_n f\|_{\mathcal{X}}^2  \dif \rho(f).
$$

Now note that for each $f \in \mathcal{X}$, we have $\Pi_n f \to f$ in norm, so $\|f - \Pi_n f\|_{\mathcal{X}}^2 \to 0$, and that
$$
\|f - \Pi_n f\|_{\mathcal{X}}^2 \le \|f\|_{\mathcal{X}}^2.
$$
By assumption, $f \mapsto \|f\|_{\mathcal{X}}^2$ is integrable with respect to $\rho$, so by the dominated convergence theorem,
$$
\lim_{n \to \infty} \int_{\mathcal{X}} \|f - \Pi_n f\|_{\mathcal{X}}^2 \dif \rho(f) = 0.
$$
It follows that
$$
\lim_{n \to \infty} \|K - K \Pi_n\|_{L^2_\rho(\mathcal{X}; \mathcal{Y})} = 0. 
$$
\end{proof}

%%%%%%%%%%%%%%%%%%%%%%%%%%%%%%%%%%%%%%%%%%%%%%%%%%%%%%%
%%   Introduction to Bochner Integral / H-Valued Random 

\subsection{Probability Measures on Separable Hilbert Spaces}\label{subsec: Prob Measures on Sep H Space}
\subsubsection{The Bochner Integral} Here, we briefly introduce an extension of Lebesgue integration, due to Bochner, to functions taking values in a Banach space. Of particular interest will be random variables taking values in a separable Hilbert space, which we introduce in the following section. Our exposition will be brief, as much of the construction is reminiscent of the standard Lebesgue integral. The interested reader can refer to \cite{hsing2015theoretical}---whose exposition we closely follow---for an application oriented introduction or to \cite{Distel_Uhl_1977} for theoretical details.

Let $f$ be a function on a measure space $(\Omega, \scr{F}, \mu)$ taking values in a Banach space $(\mathcal X, \norm{\cdot}_{\mathcal X})$. Such a function is said to be simple if it can be represented as
$$f(\omega) = \sum_{i\in [k]}\chi_{F_i}(\omega)x_i,$$
for some finite $k\in \N$, $F_i\in \scr F$ and $x_i \in \mathcal X$.
\begin{defn}[Definition 2.6.2 of \cite{hsing2015theoretical}]
    A simple function $f(\omega) = \sum_{i\in [k]}\chi_{F_i}(\omega)x_i$ with $\mu(F_i)< \infty$ for all $i\in [k]$ is Bochner-integrable, with its Bochner integral defined as
    $$\int_{\Omega}f \dif \mu = \sum_{i\in [k]}\mu(F_i)x_i.$$
\end{defn}
It is straightforward to show that the above definition is well-defined, so the supporting sets $F_i$ can be taken to be disjoint without loss. The following definition extends the notion of Bochner-integrability to non-simple functions. 
\begin{defn}[Definition 2.6.3 of \cite{hsing2015theoretical}]
    A measurable function $f:(\Omega, \scr{F}, \mu)\to (\mathcal X, \scr{B}(\mathcal X,\norm{\cdot}_{\mathcal X}))$ is Bochner integrable if there exists a sequence $\{f_n\}_{n\in \N}$ of simple Bochner-integrable functions such that
    $$\lim_{n\to \infty}\int_{\Omega}\norm{f_n - f}_{\mathcal X}\dif \mu = 0.$$
    In this case, we define
    $$\int_{\Omega} f \dif \mu = \lim_{n\to \infty}\int_{\Omega}f_n \dif \mu.$$
\end{defn}
The above definition is difficult to work with in practice, as a sequence of approximating simple functions may not be readily constructable. The following theorem provides a simple sufficient condition, provided that we take $\mathcal X$ to be a separable Hilbert space. 
\begin{thm}[Theorem 2.6.5 of \cite{hsing2015theoretical}]\label{thm: suff_boch_int}
    Suppose that $\mathcal X$ is a separable Hilbert space and $f:(\Omega, \scr F, \mu)\to (X, \scr{B}(\mathcal X, \norm{\cdot}_{\mathcal X}))$ is a measurable function. If 
    $$\int_{\Omega}\norm{f}_{\mathcal X}\dif \mu < \infty,$$
    then $f$ is Bochner-integrable. 
\end{thm}
The Bochner integral shares many of the properties of the Lebesgue integral. In particular, the monotonicity property
$$\norm{\int_{\Omega}f \dif \mu}_{\mathcal X} \le \int_{\Omega}\norm{f}_{\mathcal X}\dif \mu$$
is retained, and there is a direct analogue of the dominated convergence theorem. We end with a useful theorem characterizing the interaction between bounded linear operators and the Bochner integral. 
\begin{thm}[Theorem 3.1.7 of \cite{hsing2015theoretical}]\label{thm: boch_int_commute}
    Let $\mathcal X$ and $\mathcal Y$ be Banach spaces, $f$ be a Bochner integrable function from a measure space $\Omega$ to $\mathcal X$, and $T \in \mathfrak{B}(\mathcal{X}, \mathcal Y)$ a bounded linear operator. Then $Tf \in \mathcal Y$ is Bochner integrable, and 
    $$\int_{\Omega}Tf\dif \mu = T\left(\int_{\Omega} f \dif \mu \right).$$
\end{thm}

\subsubsection{Random Elements of a Separable Hilbert Space}
Let $(\Omega, \scr{F}, \mathbb{P})$ be a probability space and
$$X: (\Omega, \scr{F}, \mathbb P)\to (\mathcal X, \scr{B}(\mathcal X, \norm{\cdot}_{\mathcal X}))$$
be a random element of a separable Hilbert space, $\mathcal X$. 
\begin{defn}[Definition 7.2.1 of \cite{hsing2015theoretical}]
    If $\E\left[\norm{X}_{\mathcal X}\right] < \infty$, then the expected value of $X$ is defined as the Bochner integral 
    $$m = \E[X] \coloneqq \int_{\Omega} X \dif \mathbb P.$$
\end{defn}
We note that since $\mathcal X$ is assumed to be separable, by Theorem \ref{thm: suff_boch_int}, the condition in the above definition ensures that $X$ is Bochner integrable so that the expected value is well-defined. A simple application of Theorem \ref{thm: boch_int_commute} leads to an equivalent characterization. Since
$$\E \ip{X,f}_{\mathcal X} = \ip{\E X,f}_{\mathcal X} = \ip{m,f}_{\mathcal X},$$
the expected value $m$ is the representor of the linear functional 
$$f \mapsto \E\ip{X,f}_{\mathcal X}.$$ A similar application of Theorem \ref{thm: boch_int_commute} shows that
$$\E\norm{X-m}_{\mathcal X}^2 = \E\norm{X}^2_{\mathcal X} - \norm{m}_{\mathcal X}^2$$
when $\E\norm{X}_{\mathcal X}^2$ exists and is finite.
 
\begin{defn}[Definition 7.2.3 of \cite{hsing2015theoretical}]
    If $\E\norm{X}_{\mathcal X}^2 < \infty$, then the covariance operator for $X$ is a Hilbert-Schmidt operator $\mathcal{K}\in \mathfrak{B}_{HS}(\mathcal X)$ defined by the Bochner integral
    $$\mathcal{K} = \E\left[(X-m)\otimes (X-m)\right] \coloneqq \int_{\Omega} (X-m)\otimes (X-m) \dif \mathbb P,$$
    where we recall that $f\otimes g \coloneqq f\ip{g, \cdot}_{\mathcal X}. $
\end{defn}
\begin{thm}[Theorem 7.2.5 of \cite{hsing2015theoretical}]\label{thm: prop_rand_H_Vars}
Suppose that $\E\norm{X}_{\mathcal X}^2 < \infty$. Then for $f,g \in \mathcal X$,
\begin{enumerate}
    \item $\ip{\mathcal K f,g}_{\mathcal X} = \E\left[\ip{X-m,f}_{\mathcal X}\ip{X-m,g}_{\mathcal X}\right]$
    \item $\mathcal K$ is a positive definite trace class operator with $\tr \mathcal{K} = \E\norm{X-m}_{\mathcal X}^2$, and 
    \item $\mathbb P\left(X \in \overline{\Ima (\mathcal K)}\right) = 1$.
\end{enumerate}
\end{thm}
From properties 1 and 2 above, we have that $\mathcal K$ is a compact, self-adjoint operator. It therefore admits the eigen-decomposition
$$\mathcal K = \sum_{j\in \N} \lambda_j e_j \otimes e_j,$$
where $\{e_j\}_{j\in \N}$ form an orthonormal basis for $\overline{\Ima(\mathcal K)}$. Then by property 3, the random element $X$ may be written as 
$$X = \sum_{j\in \N}\ip{X,e_j}_{\mathcal X}e_j = m + \sum_{j\in \N}\ip{X-m,e_j}_{\mathcal X}e_j$$
with probability 1. Thus, the one-dimensional projections $\ip{X,e_j}_{\mathcal X}$ determine the distribution of $X$, a fact we will make use of when constructing probability measures on $\mathcal X$.

\subsubsection{Example: The Gaussian Measure}

 One of the simplest and well studied classes of probability measures on a separable Hilbert space $\mathcal X$ are the Gaussian measures, which are defined analogously to the finite dimensional case. Let $X:(\Omega, \scr F, \mathbb P)\to (\mathcal X, \mathcal B(\mathcal X, \norm{\cdot}_{\mathcal X}))$ be a random $\mathcal X$-valued variable, and denote the distribution of $X$ by $\mu_{X}$, the push-forward probability measure on $\mathcal X$. The probability measure $\mu_X$ is Gaussian if its characteristic functional
 $$\phi_X(f) \coloneqq \E\left[e^{i\ip{X,f}_{\mathcal X}}\right] = \int_{\mathcal X}e^{i\ip{g,f}_{\mathcal X}}\dif \mu_{X}(g)$$ is of the form 
 $$\phi_X(f) = \exp\left\{ i\ip{m, f}_{\mathcal X} - \frac{1}{2}\ip{\mathcal K f,f}_{\mathcal X}\right\},$$
 where $m\in \mathcal X$ is fixed and $\mathcal K$ is a trace-class operator. See \cite{pinelis2009optimal} for more details. Recall that $\mathcal K$ is trace-class if 
 $$\tr{\abs{\mathcal K}} < \infty,\hspace{0.5cm} \text{where} \hspace{0.5cm}  
 \tr A \coloneqq \sum_{k\in \N}\ip{A \xi_n, \xi_n}_{\mathcal X}$$
 for an orthonormal basis $\{\xi_n\}_{n\in \N}$ of $\mathcal X$, and $\abs{A} \coloneqq \sqrt{A^*A}$ is the Hermitian square root. If the distribution of $X$ has Gaussian characteristic, as above, it can be shown that $m$ and $\mathcal K$ are the expected value and covariance operator for $X$, respectively. Moreover, $X$ can be represented as 
 $$X = m + \sum_{i\in \N}X_i e_i,$$
 where $\{e_i\}_{i\in \N}$ is an orthonormal eigen-basis of $\mathcal X$ associated to the compact operator $\mathcal K$, and the coefficients $X_i \sim \mathcal{N}(0, \sigma_i^2)$ are independent centered Gaussian random variables with variances $\sigma_i^2 \in \sigma(\mathcal K)$, the eigenvalues of $\mathcal K$. See e.g., \cite{rao2014characterization}.

 We note that there is an equivalent characterization in terms of the dual space, $\mathcal X^*$: namely, a Borel measure $\mu$ on a separable Banach space $\mathcal X$ is said to be a non-degenerate centered Gaussian measure if for every linear functional $f \in \mathcal X^*\setminus \{0\},$ the push-forward measure $\mu \circ f^{-1}$ is a centered Gaussian measure on $\R$. The most well known of the Gaussian measures is likely the classical Wiener measure defined on the space of continuous paths.

\subsubsection{Isserlis' Theorem}
 
\begin{thm}[Isserlis' Theorem]\label{thm:Isserlis} If $\bmat{X_1 & \cdots & X_n}^{\top}$ is a vector of random variables, then
    $$\E[X_1 \cdots X_n] = \sum_{p\in \mathcal{P}_n}\prod_{b\in p}\kappa\left((X_i)_{i\in B}\right),$$
    where $\mathcal{P}_n$ is the set of partitions of $\{1, \cdots, n\}$, the product is over the blocks of $p\in \mathcal{P}_n$, and $\kappa\left((X_i)_{i\in B}\right)$ is the joint cumulant of $(X_i)_{i\in b}.$

    If $\bmat{X_1  \cdots X_n}^{\top}$ is a zero-mean multivariate normal random vector, then
    $$\E[X_1 \cdots X_n] = \sum_{p\in \mathcal{P}^2_n}\prod_{\{i,j\} \in p} \E\left[X_iX_j\right] = \sum_{p\in \mathcal{P}^2_n}\prod_{\{i,j\}\in p}\Cov(X_i, X_j),$$
    where $\mathcal{P}^2_n$ is the set of partitions of $\{1, \cdots, n\}$ whose blocks are \emph{pairs}, and the product is over the pairs contained in $p$. 
\end{thm}
\begin{cor}
    If $X = \bmat{X_1 & \cdots & X_n }^{\top}\sim \mathcal{N}(0, \boldsymbol{\Sigma})$ and $n$ is odd, then
    $$\E[X_1\cdots X_n] = 0$$
    since there exists no partition of $\{1, \cdots n\}$ into pairs. 
\end{cor}

\subsection{Sampling considerations}
This section collects results concerning sampling from the optimal distribution.

\subsubsection{Optimal Sampling for Multilinear Operators $M_k$ in \eqref{eq:multilinear-operator}}
We discuss why we do not form subspaces from the multilinear operators $M_k$. The challenge with this approach is that sampling from the optimal measure $\mu$ is difficult, even when $\rho$ is a product measure. The essential bottleneck is the fact that for $\bs{n}, \bs{m} \in \N_F^\infty$, $M_{k_{\bs{n}}}$ and $M_{k_{\bs{m}}}$ are rarely orthogonal. A simple remedy would be to orthogonalize the operators, requiring the computation of $L^2_\rho(\mathcal{X}; \mathcal{Y})$ inner products. For $\bs{n},\bs{m} \in \N_F^\infty$ with $r = \|\bs{n}_{\sim 0}\|_0 $ and $s = \|\bs{m}_{\sim 0}\|_0 $, a computation yields:
\begin{align}\label{eq:M-multilinear-ip}
    \ip{M_{k_{\bs m }}, M_{k_{\bs n}}}_{L^2_{\rho}} & = \E_{f\sim \rho} \ip{M_{k_{\bs m}}f, M_{k_{\bs n}}f}_{\mathcal{Y}} 
    = \delta_{m_{0}, n_{0}} \E\left[\prod_{i=1}^r \prod_{j=1}^s \hat f_{m_i} \hat f_{n_j}\right],
\end{align}
revealing that if $r + s \geq 2$, we require higher order moments of the Fourier coefficients. In the relatively rare event that any entry in $(m_1, \ldots, m_s, n_1, \ldots, n_r)$ is duplicated at most once in that vector, then the above expectation could be computed explicitly as a function of first and second moments of the Fourier coefficients.

Stronger assumptions could be made to circumvent the need for higher-order moments. For example, if we further assume that $\rho$ is a centered Gaussian measure, i.e., $\hat f_k \coloneqq \ip{f, \xi_k}_{\mathcal X} \sim \mathcal N(0, \sigma^2_k)$, then by Isserlis' theorem, 
\begin{align}\label{eq:M-multilinear-ip-isseralis}
    \ip{M^r_{k_{\bs m}}, M^s_{k_{ \bs n}}}_{L^2_\rho} & = \delta_{m_{0}, n_{0}}\sum_{p\in \mathcal P^2}\prod_{\{k,l\} \in p} \sigma^2_k \delta_{k,l}
\end{align}
where $\mathcal P^2$ is the set of \textit{perfect pair partitions} of $\{m_1, \cdots, m_r, n_1, \cdots, n_s\}$\footnote{Here, we take pair partitions of the \emph{symbols} $m_i,n_j$ so the set is of cardinality $r+s$, even if the \emph{values} of $m_i,n_j$ are duplicated.}, providing an explicit formula for the inner product (see \cref{thm:Isserlis} for details). An immediate corollary of this is that if $r+s$ is odd, then $\ip{M_{k_{\bs{n}}}, M_{k_{\bs{m}}}}_{L^2_\rho} = 0$ since there are no perfect pair partitions of a set of odd cardinality. 

Despite the results on computing inner products, neither the formula \eqref{eq:M-multilinear-ip}, which relies on high-order moments of the Fourier coefficients, nor \eqref{eq:M-multilinear-ip-isseralis}, which imposes a stringent Gaussianity assumption, is particularly practical for numerical computations. Both approaches require computing these inner products and subsequently orthogonalizing the operators.

\subsubsection{Approximate Optimal Sampling for Univariate Least Squares Approximations via Orthogonal Polynomials}\label{sec:approx_induced_samp_LS}
Let $\Lambda \subset \N_0$ have finite size, $|\Lambda | = N$, and let $f \in L^2_\mu(\R)$. We seek to use empirical weighted least squares to approximate a continuum least squares problem. The feasible set of approximants is,
\begin{align*}
  V_\Lambda &= \mathrm{span} \{p_n\}_{n \in \Lambda}, & \dim V_\Lambda &= N,
\end{align*}
where $\{p_n\}_{n \in \N_0}$ are the $L^2_\mu$-orthonormal polynomials, see, e.g., \cite{simon2015operator}. The continuum problem seeks to solve,
\begin{align*}
  v_{\Lambda}^\ast \coloneqq \argmin_{v \in V_\Lambda} \| f - v\|_{L^2_\mu}.
\end{align*}
To approximate this solution, we solve the empirical, weighted problem,
\begin{align*}
  v_\Lambda \coloneqq \argmin_{v \in V_\Lambda} \sum_{m \in [M]} \frac{w_m}{M} \left( f(X_j) - v(X_j) \right)^2,
\end{align*}
where $M \in \N$ is the finite empirical sample size, and $\{X_m\}_{m \in [M]}$ are iid samples from a probability measure $\rho$. To maintain $M$-asymptotic consistency of the objective function relative to the $L^2_\mu$ norm, we impose,
\begin{align*}
  X_m \sim \rho \hskip 10pt \Longrightarrow \hskip 10pt w_m = \frac{\dif{\mu}}{\dif{\rho}}(X_m).
\end{align*}
It's well-known, cf. \cref{eq:SampleComplexity} that choosing $\rho$ as,
\begin{align*}
  \dif{\rho}(x) = \sup_{v \in V_\Lambda, \, \|v\|_{L^2_\mu} = 1} |v(x)|^2 \dif{\mu}(x) = \frac{1}{N} \sum_{n \in \Lambda} p_n^2(x) \dif{\mu}(x),
\end{align*}
is optimal in the sense that a minimal number of samples are required to obtain high-probability subspace embedding theoretical results. 

The next major task we describe is \textit{how} one can feasibly sample from $\rho$ for general $\mu$, $V_\Lambda$. The first step is to recognize that this problem of sampling from a general $\rho$ can be written in terms of specific building blocks. We write $\rho$ as,
\begin{align*}
  \dif{\rho}(x) &= \frac{1}{N} \sum_{n \in \Lambda} \dif{\rho}_n(x), & \dif{\rho}_n(x) &\coloneqq p_n^2(x) \dif{\mu}(x).
\end{align*}
Since $\int \dif{\rho}_n(x) = \int p_n^2(x) \dif{\mu}(x) = 1$, then $\rho_n$ is a probability measure for every $n$. Therefore, $\rho$ is a mixture of the $N$ measures $\rho_n$ and so the ability to sample from every $\rho_n$ implies the ability to sample from $\rho$.

Fix $n$ and consider the problem of sampling from $\rho_n$. To perform inverse transform sampling (one of the standard ways to sample from arbitrary univariate measures), we would need to computationally construct the function
\begin{align*}
  F_n(x) \coloneqq \int_{-\infty}^x p_n^2(s) \dif{\mu}(s) \in [0,1],
\end{align*}
and subsequently perform numerical rootfinding to solve $F_n(x) = U$, where $U \sim \mathrm{Uniform}([0,1])$. Integrating a polynomial integrand is not difficult, but $\mu$ can be arbitrary, which makes numerically approximating $F_n$ challenging \cite{narayan_computation_2018}.

To circumvent this challenge, we can approximate $\mu$ by a finite measure and perform discrete sampling. Here is one strategy to accomplish this: we can replace $\mu$ with the empirical (discrete) measure corresponding to a $Q$-point $\mu$-Gaussian quadrature rule. I.e., for a fixed $Q \in \N$, let $\{x_q\}_{q \in [Q]}$ and $\{w_q\}_{q \in [Q]}$ be the nodes and weights, respectively, for a $Q$-point Gaussian quadrature rule,
\begin{align*}
  \{x_q\}_{q \in [Q]} &= p_Q^{-1}(0), & w_q &= \frac{1}{\sum_{j=0}^{q-1} p_j^2(x_q)}, \hskip 5pt q \in [Q].
\end{align*}
Then we will make the approximation,
\begin{align*}
  \dif{\mu} \approx \dif{\mu_Q} \coloneqq \sum_{q \in [Q]} w_q \delta_{x_q},
\end{align*}
where $\delta_x$ is the Dirac mass centered at $x$. The generic reason to use this measure (as opposed to say any other empirical measure) is that the first $2Q-1$ moments of $\mu_Q$ and $\mu$ coincide. If $\mu$ is sufficiently regular to have, e.g., a moment generating function, then convergence of moment sequences can be translated into convergence in distribution of $\mu_Q$ to $\mu$. In particular, this moment-based convergence analysis is optimal in the sense that it matches as many moments as possible ($2Q-1$) for a fixed number of masses ($Q$). As a general rule of thumb, if $Q \gg n_{\max} \coloneqq \max_{n \in \Lambda} n$, then one can expect that $\mu_Q$ is an excellent approximation to $\mu$ for the purposes of approximating from the subspace $V_\Lambda$.

Once $\mu$ is approximated by $\mu_Q$, we then define an alternative, approximate optimal sampling measure:
\begin{align*}
  \dif{\widehat{\rho}}(x) \coloneqq \frac{1}{N} \sum_{n \in \Lambda} p_n^2(x) \dif{\mu}_Q(x)
\end{align*}
(The hat $\widehat{\cdot}$ on $\rho$ and in the sequel is meant to mimic notation for an empirical statistical estimator.) So long as $Q \geq n_\max$, then $\{p_n\}_{n \in \Lambda}$ is an orthonormal basis in $L^2_{\mu_Q}$, so that again we can write this as a mixture of measures,
\begin{align*}
  \dif{\widehat{\rho}}(x) &= \frac{1}{N} \sum_{n \in \Lambda} \dif{\widehat{\rho}}_n(x), & \dif{\widehat{\rho}}_n(x) &\coloneqq p_n^2(x) \dif{\mu}_Q(x).
\end{align*}
The sampling from $\widehat{\rho}_n$ is relatively straightforward: first we form a $Q \times N$ Vandermonde-like matrix,
\begin{align*}
  \left(\bs{V}\right)_{i,j} &= p_{n_j}(x_i), & \Lambda &= \{n_1, n_2, \ldots, n_N\}, & (i,j) &\in [Q] \times [N].
\end{align*}
We next form a normalized elementwise-squared matrix,
\begin{align*}
  (\bs{W})_{i,j} = \frac{\left(\bs{V} \odot \bs{V}\right)_{i,j}}{\sum_{j=1}^N \left(\bs{V} \odot \bs{V}\right)_{q,j}}
\end{align*}

We subsequently form a $Q \times N$ matrix formed from column-wise cumulative sums of $\bs{W}$:
\begin{align*}
  (\bs{F})_{i,j} = \sum_{k \in [i]} (\bs{W})_{k,j}.
\end{align*}
By construction, $(\bs{F})_{i,j}$ is the cumulative distribution function of $\widehat{\rho}_{n_j}$ evaluated at $x_i$. Therefore, given $j \in [N]$, the following procedure generates a random sample from $\widehat{\rho}_{n_j}$:
\begin{itemize}
  \item Generate $U \sim \mathrm{Uniform}([0,1])$
  \item Compute the smallest value of $i$ such that $U \leq (\bs{F})_{i,j}$.
  \item Return $x_i$.
\end{itemize}

\subsubsection{Optimal Sampling for Discrete Measures}\label{app:sampling-discrete-measures}

In previous sections, we described explicit procedures for generating independent samples $f^i\in \mathcal X$ from an optimal sampling measure $\mu$, as in \eqref{eq:opt weight and measure}. Doing so required both full knowledge of the approximation measure $\rho$ on $\mathcal X$ and the ability to define an $L^2_\rho$-orthonormal basis for a desirable approximation space $V.$ We have seen that if $\rho$ is a known product measure and $V$ is of the form \eqref{eq:V-linear} or \eqref{def: nonlinear approx space}, then this can be readily accomplished, provided that one can efficiently sample from the component measures $\rho_j$ and their induced counterparts $\tilde{\rho_j}$  defined in \eqref{eq:lin-ind-density} and \eqref{def: ind prod measure nonlin}. In general however, if $\rho$ is not a product measure, defining a class of orthonormal operators in $L^2_{\rho}$ is onerous, and even if some such orthonormal set were given, sampling from the resulting measures may be computationally challenging \cite{adcock2022towards}. These issues cannot be ignored, since in practical applications, one may have no control over $\emph{how}$ samples are generated. A simple remedy to these issues was provided in \cite{migliorati_multivariate_2021,adcock2020near} for the case of function approximation and \cite{adcock2022towards} for Hilbert space valued functions. Briefly, the idea is to replace the (typically continuous) measure $\rho$ by a discrete measure $\upsilon$ supported on a finite set, on which both sampling and constructing a orthonormal basis is simple.  

\paragraph{The Construction}
We assume, as we did implicitly in the more specific cases of \eqref{eq:V-linear} and \eqref{def: nonlinear approx space}, that for some index set $\Lambda,$
\begin{align}\label{def: emperical approx space} V  = \text{span}_{\bs \alpha \in \Lambda}\{\psi_{\bs \alpha}g_{\bs \alpha}(\cdot)\} \subset \mathcal{Y}\otimes G_{\Lambda},\end{align}
where $g_{\bs \alpha} \in L^2_\rho(\mathcal X;\R)$ and $G_{\Lambda} = \text{span}_{\bs \alpha \in \Lambda}g_{\bs \alpha}.$ Note that we do not assume that $\{g_{\bs \alpha} \}_{\bs \alpha \in \Lambda}$ is orthonormal. Now, let $X = \{x_i\}_{i=1}^S \subset \mathcal X,$ and define the discrete uniform measure $\upsilon$ on $\mathcal X$ as in \eqref{def: discrete measure tau}.
%\begin{align}\label{def: discrete measure tau}\upsilon = \frac{1}{S}\sum_{k\in [S]}\delta_{x_i}\end{align}
It is immediate that if $\{b_{\bs \alpha}\}_{\bs \alpha \in \Lambda}$ is an $L^2_\upsilon(\mathcal X;\R)$ orthonormal basis for $G_{\Lambda},$ then 
\begin{align*}
\Phi^{\upsilon}_{\bs \alpha} & \coloneqq \psi_{\bs \alpha}b_{\bs \alpha}(\cdot) & \text{span}_{\bs \alpha \in \Lambda}\{\Phi_{\bs \alpha}^{\upsilon}\} &= V & \ip{\Phi_{\bs \alpha}^\upsilon, \Phi_{\bs \beta}^{\upsilon}}_{L^2_{\upsilon}(\mathcal X; \mathcal Y)} = \delta_{\bs \alpha,\bs \beta}.
\end{align*}
All of the theory in \cref{sec: Main Results: LS-OP} holds for the discrete measure $\upsilon,$ so by \cref{cor:OptStab and Sample Complexity Estiamte}, if $\{ f^i\}_{i=1}^M\sim \dif \mu^M$ with $M \ge c_{\delta}  N \log(2N/\epsilon),$ then with probability at least $1-\epsilon$,
$$(1-\delta)\norm{\cdot}_{L^2_\upsilon(\mathcal X; \mathcal Y)}^2 \le \norm{\cdot}^2_{L^2_{\upsilon,M}} \le (1+\delta)\norm{\cdot}^2_{L^2_\upsilon(\mathcal X; \mathcal Y)} \hspace{0.5cm}(\forall A \in V)$$
so that the approximation theorems in \cref{ssec: Accuracy} may be applied. 
 Explicitly, the optimal measure is defined by $\dif \mu(f) = w(f)^{-1} \dif \upsilon(f)$ with the weights 
 \begin{align}\label{def: opt weights discrete}w( f) = \frac{N}{\sum_{\bs \alpha \in \Lambda}\norm{\psi_{\bs \alpha}b_{\bs \alpha}(f)}_{\mathcal Y}^2} = \frac{N}{\sum_{\bs \alpha \in \Lambda}\abs{b_{\bs \alpha}(f)}^2}.\end{align}
 Obviously, we would like to obtain an error estimate with respect to the original measure $\rho.$ This is possible whenever the $L^2_\upsilon$ norm is comparable to the $L^2_\rho$ norm over $V.$ One way to achieve this is to construct $X = \{x_i\}_{i=1}^{S}$ as a random Monte Carlo grid -- i.e., to sample $x_i\sim \rho$ and to take the weights defining the discrete norm $\norm{\cdot}_{L^2_{\rho,S}}$ to be identically 1. 
 Indeed, since 
 $$\norm{A}_{L^2_\upsilon}^2 =  \frac{1}{S}\sum_{i\in [S]}\norm{Ax_i}_{\mathcal Y}^2 = \norm{A}_{L^2_{\rho,S}}^2$$
 by \cref{thm:Samp Complexity}, if $S \ge c_{\delta} \kappa_{w\equiv 1}^{\infty}(V)\log(2N/\epsilon),$ then with probability at least $1-\epsilon$,
$$(1-\delta)^2\norm{\cdot}_{L^2_\rho(\mathcal X;\mathcal Y)}^2 \le \norm{\cdot}_{L^2_\upsilon}^2 \le (1+\delta)\norm{\cdot}_{L^2_\rho}^2.$$
Of course, the constant $\kappa_{w\equiv 1}^{\infty}$ may be very large (or even infinite) depending on the choice of $V$, so that a large grid $X\subset \mathcal X$ is required. However, once such a grid is obtained, the sample complexity of the least squares problem remains near-optimal, requiring only $M \gtrsim N\log(2N/\epsilon)$ samples from $X. $

\paragraph{The Induced Optimal Measure}%\label{sss: ind opt mesure discrete}
In order to sample from $\mu$, we must first construct a $L^2_\upsilon(\mathcal X; \R)$ orthonormal basis of $G_{\Lambda}.$ Since $\upsilon$ is a discrete measure, this can be accomplished by a simple $\text{QR}$ decomposition. Indeed, let
$\{\bs \alpha^1, \cdots, \bs \alpha^N\} = \Lambda$
be any enumeration of the index set. Then define
$$\bs V = \left(g_{\bs\alpha^j}(x_i)/\sqrt{S}\right)_{i\in [S], j\in [N]} \in \R^{S\times N},$$
and let $\bs V = \bs Q \bs R$ be a QR factorization of $\bs V$: i.e., $\bs Q \in \R^{S\times N}$ and $\bs R\in \R^{S\times S}$, with $\bs Q_{j}$ orthonormal columns and $\bs R$ upper-triangular. It follows that
$$b_{ \bs \alpha^i}(f) = \sum_{j=1}^i\left(\bs R^{-\top}\right)_{ij}g_{\bs \alpha^j}(f) \hspace{0.5cm}(i\in [N])$$
forms an $L^2_\upsilon(\mathcal X; \R)$-orthonormal basis for $G_{\Lambda}$, and for all $x\in X,$ we have that $b_{ \bs \alpha^j}(x_i) = \sqrt{S} Q_{ij}.$
Thus, by \eqref{def: opt weights discrete} and \eqref{def: discrete measure tau}, the optimal induced measure is
\begin{align}\label{def: ind disc opt meas}\dif \mu (f) = \sum_{i\in [S]}\left(\frac{1}{N}\sum_{j\in [N]}\abs{Q_{ij}}^2\right)\dif \delta_{x_i}(f).\end{align}
Hence, $\mu$ is a discrete probability measure on $X$ with mass function
$$\Pr\{f = x_i\} = \frac{1}{N}\sum_{j\in [M]} \abs{Q_{ij}}^2,$$
so optimally sampling is trivial. 

\subsection{Computational Complexity}\label{appendix: comp_complexity}

In this section, we give computational complexity estimates for the solution of the least squares problem \eqref{eq:discrete-ls} under the assumption that the approximation space $V$ is of a simple tensor product form. In particular, we restrict the domain and co-domain of the target operator $K\in L^2_\rho(\mathcal X; \mathcal Y)$ to the finite dimensional spaces
$$\mathcal X_h \coloneqq \text{span}_{i\in[d_{\mathrm{in}}]}\{\xi_i\} \hspace{0.25cm}\text{and}\hspace{0.25cm}\mathcal Y_h\coloneqq \text{span}_{i\in [d_{\mathrm{out}}]} \{\psi_i\},$$ and assume that
\begin{align*}
    V = \text{span}_{\lambda \in \Lambda}\{\Phi_{\lambda}\} &\subset \mathcal Y_h \otimes L^2_{\rho_h}(\mathcal X_h; \R)   & \Phi_{\lambda}(\cdot) &= \psi_{\lambda_0}k_{\lambda_{1:}}(\cdot) &  \abs{\Lambda} &= N, 
\end{align*}
where each of the above spans are over an orthonormal set and $\rho_h = \otimes_{j\in [d_{\mathrm{in}}]}\rho_j$ is the restriction of $\rho$ to $\mathcal X_h$. We note that both the finite rank projection operators \eqref{eq:Phin-linear} and the orthogonal polynomial operators \eqref{def: multivar PolyOP} fit into this framework with $$k_{\lambda_1:} = \rho_{j,2}^{-\frac{1}{2}}\xi_{\lambda_1} \hspace{0.25cm}\text{and}\hspace{0.25cm} k_{\lambda_{1:}} = \prod_{j\in \N} p^j_{\lambda_j}(\xi_j^*f)$$
respectively. 

\subsubsection*{Uniform Tensor Product Structure}
We begin by assuming that $V$ takes a uniform tensor product structure, induced by an index set of the form 
$$\Lambda = [d_{\mathrm{out}}]\times \Lambda_1,$$ for which we take the linear ordering
\begin{align}\label{Ord: Linear Ordering}\Lambda = \{\lambda_1^1,\ldots, \lambda_n^1,\lambda_1^2, \ldots, \lambda_n^2, \ldots, \lambda_1^{d_{\mathrm{out}}}, \ldots, \lambda^{d_{\mathrm{out}}}_n\},\end{align}
with $n\coloneqq \abs{\Lambda_1}$.  We now write $\Phi_{\lambda^i_j} = \psi_i k_{j}$. Under this assumption, the Gram matrix in \eqref{eq:normal-equations} admits a simple block diagonal structure. Indeed, we have
\begin{align*}G_{\lambda_j^i, \lambda^r_s} &= \ip{\Phi_{\lambda_j^i,}, \Phi_{\lambda^r_s}}_{L^2_{\rho_h,M}} = \frac{1}{M}\sum_{m\in [M]}w(f^m)k_j(f^m)k_s(f^m)\ip{\psi_i,\psi_r}_{\mathcal Y_h} \\ &= \frac{1}{M}\sum_{m\in [M]}w(f^m)k_j(f^m)k_s(f^m)\delta_{i,r}.\end{align*} Note that the above is independent of the values of $i$ and $r$. It follows that 
$\bs G  = \text{diag}(\bs B)$ with $\bs B\in \text{Sym}_n(\R)$ defined by 
\begin{align*}  B_{ij}& = \frac{1}{M}\sum_{m\in [M]}w(f^m)k_i(f^m)k_j(f^m). \end{align*} The vector equation $\bs G \bs c = \bs g$ in \eqref{eq:normal-equations} is therefore equivalent to the matrix equation
$$\bs B \bs C = \bs D,$$
where $\bs C,\bs D \in \R^{n\times d_{\mathrm{out}}}$ are defined by
$$C_{ij} = c^j_i \hspace{0.25cm}\text{and}\hspace{0.25cm} D_{ij} = g^j_i.$$ In the above, we again make use of the ordering \eqref{Ord: Linear Ordering}: e.g., the expansion coefficients $\bs c$ are ordered via
$$\bs c = \bmat{c^1_1 & \cdots & c^1_n & \cdots & c^{d_{\mathrm{out}}}_1 & \ldots& c_n^{d_{\mathrm{out}}}} = \bmat{\bs c^1 &\cdots & \bs c^{d_0}}^\top.$$ 
With this, the following least squares problems are equivalent: 
\begin{align}\label{eq: matrix LS problem}
    \argmin_{\bs c \in \R^N} \norm{\bs G \bs c - \bs g}_2 \cong \argmin_{\bs C \in \R^{n\times d_{\mathrm{out}}}}\norm{\bs B\bs C - \bs D}_F,
\end{align}
where $\norm{\cdot}_F$ is the Frobenius norm on $\R^{n\times d_{\mathrm{out}}}$. Explicitly, 
\begin{align*}
    \norm{\bs G\bs c - \bs g}_2^2 & = \norm{\bmat{\bs B\bs c^1 - \bs g^1 &\cdots & \bs B \bs c^{d_{\mathrm{out}}} - \bs g^{d_{\mathrm{out}}}}^\top}_2^2\\
    & = \sum_{j\in [d_{\mathrm{out}}]}\norm{\bs B \bs c^j - \bs g^j}^2_2 \\
    & = \sum_{j\in [d_{\mathrm{out}}]}\sum_{i\in [n]}\left( (\bs B \bs c^j)_i - g^j_i\right)^2 \\
    & = \sum_{j\in [d_{\mathrm{out}}]} \sum_{i\in [n]} \left((BC)_{ij} - D_{ij}\right)^2\\
    & = \norm{\bs B \bs C - \bs D}^2_F.
\end{align*} 
The solution to the matrix form of the least squares problem \eqref{eq: matrix LS problem} is similar to that of its vector valued counterpart, as the following lemma makes clear. 
\begin{lem}{(Matrix Least Squares)} Let $ \bs B \in \R^{m\times n}, \bs C \in \R^{n\times l}$ and $\bs D \in \R^{m\times l}.$ Then
$$\argmin_{\bs C \in \R^{n\times l}} \norm{\bs B \bs C - \bs D}_F = \bs B^{\dagger} \bs D,$$
where $\bs B^{\dagger} \coloneqq \left(\bs B^T\bs B\right)^{-1}\bs B^T$ is the pseudo-inverse of $\bs B.$
\end{lem}
\begin{proof}
Consider the objective function $h:\R^{n\times l }\to \R$ defined by 
$$h(\bs C) = \frac{1}{2}\norm{\bs B \bs C - \bs D }^2_F.$$
For all $\epsilon \in \R_{++}$ and $\bs A \in \R^{n\times l},$ we have
\begin{align*}
    f(\bs C + \bs A) &= \frac{1}{2}\norm{\bs{BC} - \bs{D} + \epsilon\bs{BA}}^2_F\\ & = \frac{1}{2}\norm{\bs B \bs C - \bs D}_F^2 + \ip{\bs B\bs C - \bs D, \epsilon \bs B\bs A}_F + \epsilon^2\norm{\bs B\bs A}_F^2,
\end{align*}
where the inner product $\ip{\bs A,\bs B}_F = \tr(\bs A^T\bs B)$ induces the Frobenius norm. It follows that
$$\lim_{\epsilon\to 0} \frac{ h(\bs C + \epsilon \bs A) - h(\bs C)}{\epsilon} = \ip{\bs{B}^T(\bs{BC} - \bs D), \bs A}_F,$$
so that the Frechet derivative of $h$ at $\bs C$ is
$$\mathrm{D}h(\bs C)(\cdot)  = \ip{\bs{B}^T(\bs{BC} - \bs D), \cdot }_F. $$ A first-order necessary condition for the minimization of $h$ is therefore that
$$\bs C = \bs B^{\dagger} \bs D.$$
\end{proof}
By the above lemma, we see that the solution to the least squares problem \eqref{eq: matrix LS problem} given by $\bs C = \bs B^{\dagger} \bs D$ is equivalent to the system of least squares problems
\begin{align*}\bs c^{j} = \bs B^{\dagger}\bs g^j \hspace{0.5cm}\text{for}\hspace{0.5cm} (j\in [d_{\mathrm{out}}])\end{align*}
corresponding to each block of the system $\bs G\bs c = \bs g.$

We finally turn to the issue of computational complexity. We begin with the complexity of solving the least squares problem. It is common for least squares problems to be solved by either a Cholesky or a QR factorization. The latter is more stable, and we shall assume that this is the method undertaken. We also assume that $\bs B$ is non-singular (which can be achieved with high probability under the sample complexity condition \eqref{cor:OptStab and Sample Complexity Estiamte}). In this event, the algorithm is as follows:
\begin{enumerate}
    \item Compute the  QR factorization $\bs B = \bs Q \bs R$
    \item Compute the matrix $\bs Q^T \bs D$
    \item Solve the upper triangular system $\bs R \bs C = \bs Q^T \bs D $ by back-substitution. 
\end{enumerate}
The first step bears the majority of the cost. It is shown in \cite{trefethen2022numerical} that forming the QR decomposition of an $n\times n$ matrix is $\mathcal O(n^3),$ whether one uses Householder reflectors or the modified Gram-Schmidt method. The matrix multiplication and the back-substitution steps each require $\mathcal O(n^2d_{\mathrm{out}})$ flops. We now turn to the complexity of forming the matrices $\bs B$ and $\bs D.$ Recall that
$$B_{ij} = \frac{1}{M}\sum_{m\in [M]}w(f^m)k_i(f^m)k_j(f^m) \hspace{0.25cm}\text{and}\hspace{0.25cm} D_{ij} = \frac{1}{M}\sum_{m\in [M]}w(f^m)\psi_j^*(g^m) k_i(f^m).$$
We assume that $\{f^m\}_{m\in [M]}$ is given and that the output data $\{g^m\}_{m\in[M]}$ is already represented in the basis $\{\psi_j\}_{j\in [d_{\mathrm{out}}]},$ so that $\xi^*_jg^m$ is $\mathcal O(1).$ Let $\kappa$ and $\omega$ be the maximal costs of computing $k_j(f^m)$ and $w(f^m)$, respectively.  One may store these values with a one time cost of $M(\omega + n\kappa)$. To form each of the $n(n+1)/2$ defining entries of the symmetric matrix $\bs B$ requires $M-1$ additions and $2M$ multiplications. Forming $\bs B$ therefore takes $\frac{1}{2}n(n+1)(3M-1)$ flops. Similarly, forming $\bs D$ requires $n^2(3M-1)$ flops.
Thus, the computational complexity of the least squares problem is: 
$$\mathcal O\left(\underbrace{n^3 + n^2d_{\mathrm{out}}}_{\text{LS Solve}} + \underbrace{n^2M}_{\text{Forming LS System}} + \underbrace{M(\omega + n\kappa)}_{\text{Computing Entries}}\right).$$
Observe that in order to satisfy the sample complexity bound \eqref{cor:OptStab and Sample Complexity Estiamte}, $M>n$, so forming the least squares system is  more computationally intensive than solving it. Finally, we note that if the input functions $f^m$ are sampled optimally, then
$$w(f^m) = \frac{N}{\sum_{\lambda \in \Lambda} \norm{ \Phi_{\lambda}(f)}_{\mathcal Y}^2} = \frac{n}{\sum_{j\in [n]} \abs{k_j(f^m)}^2}$$
so that $\omega \sim 2n$ -- assuming that the $k_j(f)$ are already computed. The cost $\kappa$ of these computations depends on the choice of operator basis. For example, in the case of orthogonal polynomial operators, we have 
$$\kappa_j(\cdot)  =\prod_{j\in [d_{\mathrm{in}}]}p^j_{\lambda_j}(\xi_j^* (\cdot)).$$
Assuming that the univariate degrees $\lambda_j$ are bounded by $r,$ we find (by a naive calculation)\footnote{The naive calculation assumes order $r^2$ complexity for evaluating a degree $r$ polynomial. Using Horner's method reduces this to linear complexity. } that $\kappa \sim d_{\mathrm{in}}r^2.$

\end{document}